\documentclass[11pt, leqno]{amsart}
\usepackage[utf8]{inputenc}
\usepackage{geometry}
\usepackage{etoolbox}
\usepackage{lipsum, bm}
\usepackage[english]{babel}
\usepackage{csquotes}
\usepackage{mathrsfs}
\usepackage{mathtools}
\usepackage{xcolor}
\usepackage{yfonts}
\usepackage{thmtools}
\usepackage[makeroom]{cancel}
\usepackage[toc,page]{appendix}
\usepackage{tocvsec2}

\usepackage[breaklinks=true]{hyperref}
\usepackage{cleveref}

\newcommand{\supp}[0]{\mathrm{supp}}
\newcommand{\conv}[0]{\mathrm{conv}}
\newcommand{\height}[0]{\mathrm{ht}}
\newcommand{\wt}[0]{\mathrm{wt}}
\newcommand{\ad}[0]{\mathrm{ad}}

\usepackage{amsmath, amsthm, amssymb}

\theoremstyle{plain}
\newtheorem{theorem}{Theorem}[section]
\newtheorem{lemma}[theorem]{Lemma}
\newtheorem{prop}[theorem]{Proposition}
\newtheorem{cor}[theorem]{Corollary}
\newtheorem{thmx}{Theorem}[section]

\theoremstyle{definition}
\newtheorem{definition}[theorem]{Definition}
\newtheorem{example}[theorem]{Example}
\newtheorem{algorithm}[theorem]{Algorithm}

\newtheorem*{p-psp}{parabolic-PSP}
\newtheorem{corollary-C}[theorem]{Corollary to Theorem C}

\newtheorem{remark}[theorem]{Remark}
\newtheorem{observation}[theorem]{Observation}

\numberwithin{equation}{section}

\newtheoremstyle{problem}{5pt}{5pt}{}{}{\normalfont}{\textbf{:}}{.5em}{}
\theoremstyle{problem}
\newtheorem{problem}{\textbf{Problem}}

\newtheorem*{fact}{\textbf{Fact}}

\begin{document}
\title[Minkowski difference weight-formulas for highest weight $\mathfrak{g}$-modules]{Minkowski difference weight formulas}
\author{G. Krishna Teja}

\begin{abstract}
Fix any complex Kac--Moody Lie algebra $\mathfrak{g}$, and Cartan subalgebra $\mathfrak{h}\subset \mathfrak{g}$.
We study arbitrary highest weight $\mathfrak{g}$-modules $V$ (with any highest weight $\lambda\in \mathfrak{h}^*$, and let $L(\lambda)$ be the corresponding simple highest weight $\mathfrak{g}$-module), and write their weight-sets $\wt V$. 
This is based on and generalizes the Minkowski decompositions for all $\wt L(\lambda)$ and hulls $\conv_{\mathbb{R}}(\wt V)$, of Khare [{\it J. Algebra.} 2016\ \& {\it Trans. Amer. Math. Soc.} 2017] and Dhillon--Khare [{\it Adv. Math.} 2017 \&  {\it J. Algebra.} 2022].
Those works need a freeness property of the Dynkin graph nodes of integrability $J_{\lambda}$ of $L(\lambda)$:
$\wt L(\lambda)\ -$ any sum of simple roots over $J_{\lambda}^c$ are all weights of $L(\lambda)$.
We generalize it for all $V$, by introducing nodes $J_V$ that record all the lost 1-dim. weights in $V$.
We show three applications (seemingly novel) for all $\big(\mathfrak{g}, \lambda, V\big)$ of our $J_V^c$-freeness:
1)~Minkowski decompositions of all $\wt V$, subsuming those above for simples.
1$'$)~Characterization of these formulas.
1$''$) For these, we solve the inverse problem of determining all $V$ with fixing $\wt V \ =$ weight-set of a Verma, parabolic Verma and $L(\lambda)$ $\forall$ $\lambda$.
2)~At module level (by raising operators' actions), construction of weight vectors along $J_V^c$-directions.
3) Lower bounds on the multiplicities of such weights, in all $V$.
\end{abstract}
\maketitle
\settocdepth{section}
\tableofcontents
\allowdisplaybreaks
\section{Introduction, and statements and a discussion of the Main results}\label{Section 1}
Throughout, $\mathfrak{g}$ is a general complex Kac--Moody Lie algebra (not necessarily symmetrizable), with a fixed Cartan subalgebra $\mathfrak{h}$, the root system $\Delta\subset \mathfrak{h}^*$, the Weyl group $W$ and the universal enveloping algebra $U(\mathfrak{g})$.
The set $\mathcal{I}$ indexes at once: the Dynkin diagram nodes, thereby simple roots $\Pi=\{\alpha_i\ |\ i\in \mathcal{I}\}$ and simple co-roots $\Pi^{\vee}= \{\alpha_i^{\vee}\ |\ i\in \mathcal{I}\}$, and also the Chevalley generators (raising) $e_i$ and (lowering) $f_i$ $\forall$ $i\in \mathcal{I}$. 
$V$ is always a general highest weight $\mathfrak{g}$-module with (any) highest weight $\lambda\in \mathfrak{h}^*$.

In this paper, we study these modules $V$ and their weight-sets $\wt V\subset \mathfrak{h}^*$ (the simultaneous eigenvalues for the $\mathfrak{h}$-action on $V$). 
We write uniform {\it Minkowski decomposition} formulas for $\wt V$ for all $V$ in our main results Theorems \ref{thmA} and \ref{thmC} for resolving Problem \ref{Problem freeness and weight for all V} below, via a notion of {\it free-directions} $J_V^c$ for weights introduced presently \big(see Definitions \ref{Definition freeness for weights} and \ref{Defn support of all holes of V}\big).
Some applications our approach include: 1)~Theorem~\ref{thmB} classifying modules $V$, with classical weight-sets of simples and parabolic Vermas etc; and 2) Proposition~\ref{Enumeration proposition- freeness at module level} on weight-multiplicities for all $V$.
Our weight-formulas (Theorems \ref{thmA} and \ref{thmC}) generalize the recent {\it Minkowski sum formulas} (see Subsection \ref{Subsection PVMs} below) of Khare \cite{Khare_Ad, Dhillon_arXiv, Khare_JA} etc., seemingly known only for $V=$ all  simple highest weight $\mathfrak{g}$-modules $L(\lambda)$ and parabolic Verma $\mathfrak{g}$-modules $M(\lambda, J)$ for all highest weights $\lambda$ . 
We study the actions of raising operators on weight spaces along free $J^c_V$-directions (Proposition \ref{Corollary local Weyl group invariance 1} and Theorem \ref{Corollary local Weyl group invariance 2} in Section~\ref{Section 6}), by which we write certain {\it slice-decompositions} Theorem \eqref{theorem J-slice decomp.} in Section~\ref{Section 6}, extending those for parabolic Vermas \eqref{Int. slice. decomp. PVM} and  simples \eqref{wt formula for simples} due to Khare et al. in Subsection \ref{Subsection PVMs}. 

Below we elaborately discuss: our main results and motivating questions to them, along with clarifying remarks and examples.
We begin by recalling weight-formulas of simples and parabolic Vermas, and those for shapes of weight-hulls in the literature.
Below $\mathbb{Z}_{\geq 0}=$ set of non-negative integers.
$M(\lambda)=$ Verma $\mathfrak{g}$-module with highest weight $\lambda\in \mathfrak{h}^*$ and highest weight vector $m_{\lambda}\neq 0\in M(\lambda)_{\lambda}$;  $M(\lambda)\twoheadrightarrow L(\lambda)$.
$M(\lambda)\twoheadrightarrow V$ denotes a highest weight $\mathfrak{g}$-module $V$ with highest weight $\lambda$.
Next, 
 $V_{\mu} = \ \big\{v\in V \ | \ h\cdot v = \mu(h)v\ \ \forall\ h\in \mathfrak{h} \big\} \  = \  \mu$-weight space in $V$. 
For $J\subseteq \mathcal{I}$, we set  $\Pi_J:=\{\alpha_j \ |\ j\in J\}$.
\subsection*{Weights and hulls, for simples and parabolic Vermas: Symmetries and freeness}\qquad

We begin with $V$ = integrable $L(\lambda)$ ($\lambda$ dominant and integral), important in representation theory, combinatorics and physics, to name a few areas.
Their $W$-{\it symmetric} weight-sets $\wt L(\lambda)$, and their weight-hulls $\conv_{\mathbb{R}}\big(\wt L(\lambda)\big) = \conv_{\mathbb{R}}\big(W \cdot \lambda\big)$ called {\it Weyl polytopes}, particularly  {\it root polytopes} $\conv_{\mathbb{R}}(\Delta)$, have been extensively studied for many decades.
In finite type, starting from 1960s, Satake \cite{Satake}, Borel--Tits \cite{Borel}, Casselman \cite{Casselman}, Vinberg \cite{Vinberg} etc, focused on these hulls, and Cellini--Marietti \cite{Cellini} recently on {\it root polytopes}, describing their shapes, faces, etc. 
The celebrated Weyl--Kac character formula gives the characters of integrable $L(\lambda)$ for symmetrizable $\mathfrak{g}$; and of maximal integrables $L^{\max}(\lambda)$ when $\mathfrak{g}$ is not symmetrizable.

However, for $V=$ non-integrable $L(\lambda)$ \big($\lambda$ not domninant integral\big), their weights $\wt L(\lambda)$, and even hulls $\conv_{\mathbb{R}}\big(\wt L(\lambda)\big)$ -- its
shape, faces and other extremal subsets of convex theoretic importance -- were seemingly not written explicitly until Khare \cite{Khare_JA} in finite type, and Dhillon--Khare \cite{Khare_Ad} and \cite{Dhillon_arXiv} in the general Kac--Moody setting.
The key objects in these works are the well-studied parabolic Verma $\mathfrak{g}$-modules $M(\lambda, J)$ (see Definition \ref{D2.1}), for $\lambda\in \mathfrak{h}^*$ and $J$ the directions of {\it integrability} for $M(\lambda,J)$.
These modules are induced from the maximal integrables $L_J^{\max}(\lambda)$ over parabolic sub algebras $\mathfrak{p}_J\subset \mathfrak{g}$: $M(\lambda,J) \underset{\mathfrak{g}-mod}{\simeq}U(\mathfrak{g})\otimes_{U(\mathfrak{p}_J)} L_J^{\max}(\lambda)$. 
Thus, their weights and weight-hulls are- 
\begin{prop}[Folklore] \label{Proposition wts of parabolic Vermas} 
Let $\mathfrak{g}$ be a Kac--Moody algebra, and fix $\lambda\in \mathfrak{h}^*$. 
Define $J_{\lambda}:=\{j\in \mathcal{I}\ |\ \lambda\big(\alpha_j^{\vee}\big)\in \mathbb{Z}_{\geq 0} \}$ to be the set of all integrable directions for $\lambda$. Then for any $J\subseteq J_{\lambda}$:  
\begin{equation}\label{Eqn Mink diff PVM}
\text{\bf Minkowski difference formula:}\qquad   \wt M(\lambda, J)\  \  =\  \ \wt L_J^{\max}(\lambda)\ -\ \mathbb{Z}_{\geq 0}\big(\Delta^+\setminus \Delta_J^+\big).\hspace*{1cm}
\end{equation}
\begin{center}
\hspace*{1.5cm} $\conv_{\mathbb{R}}\big(\wt M(\lambda, J) \big)\  =\  \conv_{\mathbb{R}}\big(W_J\cdot \lambda\big)- \mathbb{R}_{\geq 0}\big(W_J\cdot \Pi_{\mathcal{I}\setminus J}\big)$\ \  \ \big(\text{see eg.} \cite{Khare_AR, Khare_JA, Khare_Ad}\big),
 \end{center}
where $W_J$ is the parabolic Weyl-subgroup generated by simple reflections $\{s_j \ | \  j\in J\}$.
 \end{prop}
In \eqref{Eqn Mink diff PVM}, $\wt L_J^{\max}(\lambda)$ is well-understood, and the second cone is the sums of the roots outside the parabolic subroot system $\Delta_J$, namely the (negative) roots of $\mathfrak{g}$ missed in $\mathfrak{p}_J$.
\begin{definition}\label{Definition freeness for weights}
By the above formulas, we observe for weights of $M(\lambda, J)$:\medskip\\
1)~{\bf Symmetry in $J$-directions :}  $\wt M(\lambda, J)$ is invariant under  $W_J$. \medskip\\
2) $J^c\text{-{\bf Freeness}}$ (exploring which for {\it arbitrary} modules $V$ is the main theme of this paper) {\bf :}
\begin{equation}\label{Eqn freeness in PVMs}
 \qquad\quad    \wt M(\lambda, J)\ \  -\ \  \mathbb{Z}_{\geq 0}S\ \  \ \subseteq\ \wt M(\lambda, J)\qquad\quad \text{for }S=\Pi_{J^c}\quad \text{or any }\  S\ \subseteq\ \Delta^+\setminus \Delta_J^+.
\end{equation}
\end{definition}
In the structure theory of Borcherds--Kac--Moody (BKM) Lie algebras or BKM Lie super algebras $\mathfrak{g}$, {\it free imaginary roots}, and {\it free roots} in partially commutative subalgebras of $\mathfrak{g}$, and their root-multiplicities are well-studied; see the recent papers \cite{Venkatesh, Shushma} and the references there-in. 
Also inspired by this, we study for $\mathfrak{g}$-modules $V$, a freeness property for weights of $V$ (via Problems \ref{Problem freeness and weight for all V}, \ref{Problem all free root subsets for wt V} and \ref{Problem non-vanishing vectors in free-directions for any V}), for applications including computing $\wt V$ (Theorems \ref{thmA}, \ref{thmB} and \ref{thmC}) and studying weight spaces and weight-multiplicities (Proposition \ref{Enumeration proposition- freeness at module level} and Theorem \ref{Corollary local Weyl group invariance 2}) for any $V$. 

We recall the weight-formulas for simples written by Khare and Dhillon, which are important for the present discussion. 
Studying weights of parabolic Vermas, and working with the {\it integrability} $I_V$ of $V$ (see Definition \ref{Defn integrability of V}) and $\mathfrak{p}_{I_V}$-integrability of $V$, Khare et al. computed- 
(1)~weights of all non-integrable simples $L(\lambda)$ $\forall$ $\lambda$ \big(\cite[Theorem D]{Khare_JA} in finite type and \cite[Section 2.1]{Dhillon_arXiv} in general Kac--Moody setting\big).
(2) weight-hulls of every $V$ with any highest weight $\lambda$ \big(\cite{Khare_AR}, \cite[Theorem B]{Khare_JA}, \cite{Khare_Trans} in finite type, and \cite[Section 3.5]{Khare_Ad} and \cite[Theorem 2.9]{Dhillon_arXiv}\big):
\begin{theorem}[Khare et al]\label{Theorem wts of simples and hulls of all V}
Fix any $\mathfrak{g}\text{ and } \lambda$, and recall $ J_{\lambda}$, as in Proposition \ref{Proposition wts of parabolic Vermas}. 
\begin{equation}\label{Wts of simples as PVM}
\ \ \ \wt L(\lambda) \ \  = \ \ \wt M\big(\lambda, J_{\lambda}\big) \ =\  \underbrace{\wt L_{J_{\lambda}}^{\max}(\lambda)}_{ \wt L(\lambda)\ \cap\ \big[\{\lambda\}- \mathbb{Z}_{\geq 0}\Pi_{J_{\lambda}} \big]  = \wt \big(U(\mathfrak{p}_{J_{\lambda}})L(\lambda)_{\lambda}\big)}\ -\ \mathbb{Z}_{\geq 0}\big(\Delta^+\setminus \Delta_{J_{\lambda}}^+ \big).
\end{equation}
The weight-hull and all of its faces, for any $\mathfrak{g}$-module $M(\lambda)\twoheadrightarrow V$ are as follows: 
\begin{equation}\label{Eqn conv wt V}
\conv_{\mathbb{R}}(\wt V) \ \ = \ \ \conv_{\mathbb{R}}\big(M(\lambda, I_V)\big)\ =\   \conv_{\mathbb{R}}\big(W_{I_V}\cdot [\{\lambda\}-\mathbb{Z}_{\geq 0}\Pi_{\mathcal{I}\setminus I_V}] \big) \quad \big(W_{I_V}\text{-symmetric}\big).
\end{equation}
\begin{equation}\label{Eqn all faces of hulls of wtV}
    F_{w, J}\ \ = \ \ \conv_{\mathbb{R}}\big( w\cdot [\{\lambda\}\ - \ \mathbb{Z}_{\geq 0}\Pi_{J}]\big)\qquad \quad \text{for all }\ w\in W_{I_V}\text{ and }J\subseteq \mathcal{I}.
\end{equation}
\end{theorem}
Note, the cases $J=J_{\lambda}$ and $I_V$ reveal the above first-order picture of weights and hulls; we work in this paper with $J_V$ (Definition \ref{Defn support of all holes of V}).
Those Minkowski decompositions were needed in: 1)~Studing the shapes $\conv_{\mathbb{R}}(\wt V)$ \cite[Theorem 2.5]{Khare_Ad} (polyhedrality etc).
2)~Describing their faces of all dimensions \cite{Khare_JA, Khare_Ad}.
3)~Determining face inclusions for Weyl polytopes studied in the works of Satake to Cellini--Marietti \cite[Theorem 1.2]{Cellini} above, and which were completely determined for all $V$ by \cite[Theorem A]{Khare_Trans} and \cite[Theorem 2.2 and Section 5]{Khare_Ad}.
Namely, exploring (i) pairs $(w, J)$ and $(w', J')$ with $F_{w,J}\subseteq F_{w',J'}$, and (ii) the smallest and largest sets $J_{\min}$ and $J_{\max}$ among such $J$, $J'$; extremal rays \cite[Corollary 4.16]{Dhillon_arXiv} etc.
\subsection*{Our freeness problems, and the first Minkowski difference weight-formula for $V$}\qquad

 The above decompositions, hold simultaneously for weights and hulls, for $V=$ all (Vermas) parabolic Vermas and simple $L(\lambda)$s; but for general $V$ they were known only for hulls $\conv_{\mathbb{R}}(\wt V)$.
 Beyond these classical settings $V=$ parabolic Vermas and simples $L(\lambda)$, $\wt V$ seems to be known only in a few cases \cite[Theorems B and D]{Khare_JA} and \cite[Theorem 2.10]{Dhillon_arXiv}.
 Thus, 
we aim at upgrading formulas \eqref{Wts of simples as PVM} and \eqref{Eqn conv wt V} (of hulls) to finer sets $\wt V$, which might be helpful in solving the harder problem of computing $\wt V$ $\forall$ $V$.
For this, we begin with some questions of {\it free directions} for $\wt V$ in Problem \ref{Problem freeness and weight for all V} below, which lead us to study weights spaces in Problem \ref{Problem non-vanishing vectors in free-directions for any V} below and Slice-decompositions \eqref{Slice decomp. for all V} in Section \ref{Section 6} and more. 
We set for convenience 
\[
J\textbf{-supported weights :}\qquad \qquad 
\wt_J V \ \ :=\ \  \wt V\ \cap\ [\{\lambda\}\ -\ \mathbb{Z}_{\geq 0}\Pi_J]   \qquad \quad \text{ for }J\subseteq\mathcal{I}.\qquad \vspace*{-5mm}
\]
\begin{problem}\label{Problem freeness and weight for all V}
  Given $M(\lambda)\twoheadrightarrow V$, for which subsets $J\subseteq \mathcal{I}$:
\begin{itemize}
   \item[(a)] $\wt V \ - \ \mathbb{Z}_{\geq 0}\Pi_J  \   \subseteq   \ \wt V?$
   \smallskip
   \item[(b)] Weakening (a) in view of \eqref{Wts of simples as PVM}--\eqref{Eqn all faces of hulls of wtV}\ {\bf -} \ \  $\{\lambda\}\ - \ \mathbb{Z}_{\geq 0}\Pi_J\   \subseteq   \ \wt V?$
   \smallskip
   \item[(c)] More strongly, in view of \eqref{Wts of simples as PVM}\ {\bf -} \ \ $\wt V \ \ = \ \ \wt_J V \ - \ \mathbb{Z}_{\geq 0} \big(\Delta^+\setminus \Delta_{J}^+ \big)?$
    \end{itemize}
Indeed, all of these question turn-out to be closely related to each other, as we show.
\end{problem}
\begin{remark}\label{Remark Inspirations for Problem 1(b)}
Inspirations for Problem \ref{Problem freeness and weight for all V}(b): (1) A broader question of Amritanshu Prasad on weights lying on each face of $\conv_{\mathbb{R}}(\wt V)$, following the complete descriptions \eqref{Eqn all faces of hulls of wtV}. 
(2) Indeed \cite{Khare_JA, Dhillon_arXiv, Khare_Ad} solve Problem \ref{Problem freeness and weight for all V} for $\big(V,\  J\big)\ = \ \big(L(\lambda),\ J_{\lambda}\big)$.
(2$'$)~We generalizing their technique Lemma~\ref{Lemma freeness at module level} for proving our third main result Theorem \ref{thmC}\eqref{Eqn all free directions for V} \big(solving (b) $\forall$ $V$\big).
(3) And this led us to Proposition \ref{Enumeration proposition- freeness at module level} below {\bf -} some bounds for weight-multiplicities along the free $J$-directions in (b), given by counting some {\it independent} subsets of $J$ (no edges in their induced Dynkin subgraphs). 
\end{remark}
Our first main result \eqref{Eqn A1}: 1) positively solves Problem \ref{Problem freeness and weight for all V}(a)--(c) for $J=J_{\lambda}$; 2)~and thus yields the first uniform weight-formula (novel to our knowledge) for every $V$, generalizing formula \eqref{Wts of simples as PVM} for simples. 
Moreover, the equivalence of questions Problem  \ref{Problem freeness and weight for all V} (a) and (c) for all $J$, follows by \eqref{Eqn A2}.
\begin{thmx}\label{thmA}
For $\mathfrak{g}$ a general Kac--Moody Lie algebra, $\lambda\in\mathfrak{h}^*$, and any $M(\lambda)\twoheadrightarrow{ }V$ and $I\subset \mathcal{I}$:
\begin{equation*}\label{Eqn A1}
\tag{A1} \wt V \ \ =\ \ \wt _{J_{\lambda}}V\ -\ \mathbb{Z}_{\geq 0}\big(\Delta^+\setminus \Delta_{J_{\lambda}}^+\big).
\end{equation*}
  \begin{equation*}\label{Eqn A2}
  \tag{A2} \text{Moreover,}\qquad  \wt_JV\ -\ \mathbb{Z}_{\geq 0}\Pi_{J^c}\ \ \subseteq\ \  \wt V\quad \  \  \iff \quad \  \  \wt V\  \ =\ \ \wt_J V\ -\ \mathbb{Z}_{\geq 0}\big(\Delta^+\setminus \Delta_J^+\big).
  \end{equation*}
\end{thmx} 
\begin{remark}
(1) \eqref{Eqn A1} simplifies the problem of writing weights of $M(\lambda)\twoheadrightarrow{ }V$, to the case when $\lambda$ is dominant and integral.
(2) In the spirit of the seminal paper of Fernando \cite{Fernando} and \cite{Futorny} classifying all simple weight $\mathfrak{g}$-modules (and also of Theorem \ref{thmB} below), we note setting $\tilde{V} = $ the induced $\mathfrak{g}$-module $U(\mathfrak{g})\otimes_{U(\mathfrak{p}_{J_{\lambda}})} U(\mathfrak{p}_{J_{\lambda}})V_{\lambda}$ : by the PBW theorem $\tilde{V}\twoheadrightarrow V$ and $\wt \tilde{V}=\wt V' =\wt V$ for all $\tilde{V}\twoheadrightarrow V'\twoheadrightarrow V$.
(3) \eqref{Eqn A2} improves on \eqref{Eqn A1} when more information of the structure of $V$ is known. 
Ex: it applies when $V \ = $ parabolic Verma $M(\lambda,J')$ or $V\ =$ simple $L(\lambda)$, with $J = J'\ \text{or}\ J=  J_\lambda$. 
\end{remark}
Theorem \ref{thmA} is shown in Section \ref{Section proof of thmA}. Our freeness-problems (a)--(c) are at weights-level; i.e., $V$ may not be $\mathfrak{p}_J$-free, see Example \ref{Ex non-Verma V with full wts} (we also need it in the discussions on Theorem \ref{thmB} and Problem \ref{Problem non-vanishing vectors in free-directions for any V}).
\begin{example}\label{Ex non-Verma V with full wts}
    Let $\mathfrak{g}=\mathfrak{sl}_3(\mathbb{C})$ and $\mathcal{I}=\{1,2\}=J$.
    Recall, $f_1\cdot m_0$, $f_2\cdot m_0$ and $f_1^2f_2\cdot m_0$ are {\it maximal} vectors in $M(0)$ (killed by $e_1, e_2$).
    Fix the $\mathfrak{g}$-module $V=\frac{M(0)}{U(\mathfrak{g}) f_1^2f_2\cdot m_0}$.
    Since for $i\in \{1,2\}$, $V_{-\alpha_i}$ $=f_iV_0$ is spanned by maximal vectors,      $U(\mathfrak{g})V_{-\alpha_i}\twoheadrightarrow L(-\alpha_i)$, thereby $\wt L(-\alpha_i) \subseteq \wt V$ $\forall$ $i$. 
    Note, $J_{-\alpha_1}=\{2\}$ and $J_{-\alpha_2}=\{1\}$.
    Although $V$ is not free \big($V\neq M(0)$\big), it has ``full weights'': 
    \[
    \wt M(0) =  -\mathbb{Z}_{\geq 0}\{\alpha_1,\alpha_2\}\  = \ \{0\}\bigsqcup\ \  \bigcup_{\mathclap{i\in \{1,2\}}}  \ \ \Big(\underbrace{\wt L(-\alpha_i) {\small = -\mathbb{Z}_{\geq 0}\{\alpha_i, (\alpha_1+\alpha_2)\}\setminus 
\{0\} }}_{\text{By formula } \eqref{Wts of simples as PVM}}\Big)
   \ \  \ \subseteq  \wt V.\]
    \end{example}
\noindent
Our results explore these phenomena uniformly for all quotients $M(\lambda)\twoheadrightarrow V$ over all types of $\mathfrak{g}$.

Formula \eqref{Eqn A1} was important in a follow-up work \cite{WFHWMRS} for:\quad 
(1)~Classifying weak-faces and $(\{2\};\{1,2\})$-closed subsets of all $\wt V$ \big(some combinatorial analogues for faces of hulls\big), studied by Chari et al. \cite{Chari_contm, Chari_JPAA, Khare_JA, Khare_AR} for writing the characters of Kirillov--Reshetikhin $U_q(\widehat{\mathfrak{g}})$-modules, etc. 
(2)~Showing their equivalence with the (weights lying on) classical faces of $\conv_{\mathbb{R}}(\wt V)$ studied in \cite{Borel, Casselman, Cellini,  Khare_Ad, Dhillon_arXiv, Khare_JA, Satake, Vinberg}.
Now the classical faces essentially being the hulls of weight subsets in Problem \ref{Problem freeness and weight for all V}(b), we explore questions (b) and (c) also in connection to Remark \ref{Remark Inspirations for Problem 1(b)}(1), which might help in writing these combinatorial subsets.
Furthermore, we study simple root strings through weights in Section \ref{Section 6} (Theorem \ref{Corollary local Weyl group invariance 2} etc), motivated by \cite[Problem 4.8]{WFHWMRS} exploring their continuity.   
\subsection*{Intervals of highest weight modules with given weight-sets}\qquad\qquad

Solving Problems \ref{Problem freeness and weight for all V} (a) and (c), needs solving part (b), or equivalently its base case ($J=\mathcal{I}$) Problem \ref{Problem full weights of Vermas}(a) Khare asked, exploring in reverse to the direction of finding weights of given $V$:
\begin{problem}[Khare]\label{Problem full weights of Vermas}
    For a $\mathfrak{g}$-module $M(\lambda)\twoheadrightarrow V$ when is, or what are all $M(\lambda)\twoheadrightarrow V$ with:
    \begin{itemize}
    \item[(a)] $\wt V \ = \ \wt M(\lambda)$?
    \item[(b)] More generally, $\wt V\ =\  \wt M(\lambda, J)$ for (any) fixed $J\subseteq J_{\lambda}?$\quad  In particular, $\wt V \ = \ \wt L(\lambda)$?  
    \end{itemize}
\end{problem}
Our second result Theorem \ref{thmB} below completely solves this, revealing ``intervals'' of $\mathfrak{g}$-modules $V$ with $\wt V = X$, for $X= \wt M(\lambda)\ \text{or}\ \wt M(\lambda, J) \ \forall\  J\subseteq J_{\lambda}$, thereby $X=\wt L(\lambda)$ $\forall$ $\lambda$.
 In this, we are inspired by such posets (peripheral to ours) in Khare \cite[Theorem E]{Khare_JA} for shapes $X=$ $\conv_{\mathbb{R}}\big(\wt V \big)$, and also \cite[Theorems A and 4.3]{Khare_Trans}, \cite[Theorem 2.2]{Khare_Ad} for $X=$ faces of $\wt L(\lambda)$, $\wt M(\lambda, J)$ and $\conv_{\mathbb{R}}(\wt V)$ for any $\lambda, \ V$.
Theorem \ref{thmB} is proved in Section \ref{S4} and is novel to our knowledge. 
 
 Crucial in Theorem \ref{thmB} are two submodules $N(\lambda)\subseteq  N(\lambda,J)\subset M(\lambda)$ defined below;
reminiscent of the construction of the maximal submodule in $M(\lambda)$ yielding $L(\lambda)$, recovered now as $N(\lambda,J_{\lambda})$, see Observation \ref{O4.1} in Section \ref{S4}.
We define for $x=\sum_{i\in \mathcal{I}}x_i\alpha_i\in \mathbb{C}\Pi$, $\supp(x):=\{i\in \mathcal{I}\ |\ x_i\neq 0\}$. 
 \begin{definition}\label{Defn of N(lambda,J)}
Fix $\lambda\in\mathfrak{h}^*$ and $J\subseteq J_{\lambda}$. 
We define $N(\lambda)$ and resp. $N(\lambda,J)$ to be the largest proper submodules of $M(\lambda)$ w.r.t. the below properties {\bf :}
\begin{align*}\label{0property 2}
\tag{N0} \qquad \mu\ \in\ \wt  N(\lambda)\quad \implies\quad \supp(\lambda-\mu)\text{ is not independent.}
\quad\quad \text{And } \ \ N(\lambda,\emptyset):=N(\lambda). 
\end{align*}
\begin{align*} \label{0property 1}
\tag{NJ} \   \text{for }J\neq \emptyset,\ \  \mu \in \wt N(\lambda, J) \  \text{and} \ \supp(\lambda-\mu) \text{ independent}  \implies   \supp(\lambda-\mu)\cap J\neq \emptyset.
\end{align*}
\end{definition}
The existence and useful properties of $N(\lambda,J)$ are discussed in Section \ref{S4}. 
For a quick understanding here, check : 1) the (usual) $\mathfrak{g}$-submodule $\sum\limits_{j\in J}U(\mathfrak{g})f_j^{\lambda\big(\alpha_j^{\vee}\big)+1} m_{\lambda}\  \subseteq\ N(\lambda,J)$; 
2)~all 1-dim. weight spaces of $M(\lambda)$ survive in $\frac{M(\lambda)}{N(\lambda)}$. 
Next, in Example \ref{Ex non-Verma V with full wts}, the submodule $U(\mathfrak{n}^-)f_1^2f_2\cdot m_0$ satisfies \eqref{0property 2}.
\begin{thmx}\label{thmB}
Let $\mathfrak{g}$ be a Kac--Moody algebra, $\lambda\in \mathfrak{h}^*$, $J\subseteq J_{\lambda}$, and $M(\lambda)\twoheadrightarrow V$.
Then 
\[
V_J^{\max}(\lambda) \ := \ M(\lambda, J)\qquad  \text{and}\qquad  V_J^{\min}(\lambda)\ :=\ \frac{M(\lambda)}{N(\lambda, J)}
\]
are resp. the largest and smallest highest weight $\mathfrak{g}$-modules \big(in all the quotients of $M(\lambda)$\big) -- where $V_{\emptyset}^{\max}(\lambda)=M(\lambda)$ and $V_{J_{\lambda}}^{\min}(\lambda)=L(\lambda)$ -- having their weight-sets equalling $\wt M(\lambda, J)$ : 
\begin{equation*}\label{Theorem B V with full weights}
\tag{{B0}}
\wt V\ =\ \wt M(\lambda) \ \    \iff  \ \   
 V \twoheadrightarrow\ V_{\emptyset}^{\min}(\lambda)  \ \ \iff  \ \  \lambda-\mathbb{Z}_{\geq 0}\Pi_I\ \subseteq\ \wt V\  \ \ \forall\ I\  \text{independent in } \mathcal{I}.
\end{equation*}
\begin{equation*}\label{Theorem Binterval for PVM wts}
\tag{{BJ}}
\wt V \ =\  M(\lambda, J)\qquad \iff \qquad  V_J^{\max}(\lambda)\ \ \twoheadrightarrow\ \   V\ \ \twoheadrightarrow \ \ V_J^{\min}(\lambda).
\end{equation*}
\end{thmx}
\subsection*{All Minkowski difference weight-formulas for every highest weight module}\qquad\qquad

Aided by Theorems \ref{thmA} and \ref{thmB}, our main contribution in the paper Theorem~\ref{thmC} (proved in Section~\ref{Section 5}), solves Problem \ref{Problem freeness and weight for all V}(c), thereby (a).
The key tool here is, working with the below defined set $J_V\subseteq J_{\lambda}$, 
which accounts for all the minimal 1-dim. weight spaces lost in the passage $M(\lambda) \twoheadrightarrow V$:
\begin{definition}\label{Defn support of all holes of V}
        Given $M(\lambda)\twoheadrightarrow V$ over any Kac--Moody $\mathfrak{g}$, we define $J_V$ to be the {\bf union of all} 
        \begin{equation}
        I \ \ minimal \   
 independent \ \text{ subset of }\ J_{\lambda},\quad  \text{ with }\ \  \lambda\ -\ \sum_{i\in I}\big[\lambda(\alpha_i^{\vee})+1\big]\alpha_i\ \  \boldsymbol{\notin \wt V}. 
        \end{equation}
                   \end{definition}
                   See Section \ref{Section 5} Definition \ref{Defn of support of all holes in Section 5} for a meticulous description of $J_V$, and the discussions from Example~\ref{Ex type A1xA1} to Example~\ref{Ex of JV} on its properties. 
                \begin{remark}\label{Remark JV generalizing integrability}
                  Our tool $J_V$ (its study and applications  seem novel in the literature) generalizes the integrability $I_V$ (Definition \ref{Defn integrability of V}) that was crucial for weight-formulas of Khare et al. above:           1)~$J_{L(\lambda)}\ =\ J_{\lambda}$ and 2) $I_V\subseteq J_V$. 
                  See Example \ref{Ex of JV} for more settings.
                  Note $J_V=\emptyset$ in Example \ref{Ex non-Verma V with full wts}.
                  \end{remark}
                  \begin{thmx}\label{thmC}
        Let $\mathfrak{g}$ be a general Kac--Moody algebra and $\lambda\in \mathfrak{h}^*$.
        For any $\mathfrak{g}$-module $M(\lambda)\twoheadrightarrow V$, we have the following refinement of weight-formula \eqref{Eqn A1}, which was for $J=J_{\lambda}$, to $J=J_V$: 
   \begin{equation*}\label{Minkowski formula for supp of holes}
   \tag{CV}
    \wt V \ = \ \wt V \cap \big[ \lambda-\mathbb{Z}_{\geq 0}\Pi_{J_V} \big]\ - \ \mathbb{Z}_{\geq 0}\left(\Delta^+\setminus \Delta^+_{J_V}\right).
    \end{equation*} 
    More strongly, for any $J\subseteq \mathcal{I}$, we have (completely resolving Problem \ref{Problem freeness and weight for all V}(c)):
     \begin{equation*}\label{All Minknowski formulas}
     \tag{CJ}
    \wt V \ = \ \wt V \cap \big[ \lambda-\mathbb{Z}_{\geq 0}\Pi_J \big]\ - \ \mathbb{Z}_{\geq 0}\left(\Delta^+\setminus \Delta^+_J\right) \qquad \iff \qquad J\supseteq J_V.
    \end{equation*}
    Therefore, we have an interval of subsets $J_V\ \subseteq\  J\ \subseteq\  \mathcal{I}$ for which the above decompositions hold (similar to the intervals for each face of hulls).\\ 
Now here are some free simple-root directions for $\wt V$ (concerning Problem \ref{Problem freeness and weight for all V}(b)):
    \begin{equation*}\label{Eqn all free directions for V}
    \tag{C0}
    \{\lambda\}\ - \ \mathbb{Z}_{\geq 0}\Pi_J \ \ \subseteq \ \
 \wt V \qquad \text{ for all }\  J\ \subseteq J_V^c.
 \end{equation*}
      \end{thmx}
        \begin{remark}
          Philosophically, Theorem \ref{thmC}\eqref{All Minknowski formulas} says $J_V$ is the ``smallest union of all the obstacles'' to our desired Minkowski decompositions for $\wt V$ w.r.t. $J\subseteq \mathcal{I}$  \big(in the sense of Problem \ref{Problem freeness and weight for all V}(c)\big).
        I.e., if $\wt V$ does not admit the Minkowski decomposition w.r.t. some $J\subseteq \mathcal{I}$, then $J\cap J_V\neq \emptyset$ by \eqref{Theorem B V with full weights} (see the discussions in Section~\ref{Section 5}). 
        We cannot extend \eqref{Minkowski formula for supp of holes} for $J$ further in $J_V$, see
          Example \ref{Ex type A1xA1}.
      \end{remark}
      \begin{cor}
      Theorem \ref{thmC} \eqref{Minkowski formula for supp of holes} yields an explicit uniform weight-formula for every $\mathfrak{g}$-module $V$ (and $\lambda$) in the following cases :
      (1) Assume $J_V$ to be independent. 
      So, $U\big(\mathfrak{n}_{J_V}^-\big)$ is the polynomial algebra $R=\mathbb{C}\big\{f_j\ |\ j\in J_V \big\}$, and every $V$ (being a weight module) is $\frac{R}{I}$ for some monomial ideal $I$ in $R$.
      This simplifies the problem of determining $\wt_{J_V}V$, thereby $
      \wt V$, to a question in commutative algebra, of finding surviving monomials in the quotient $\frac{R}{I}$.
      (2) Particularly when rank of  $\mathfrak{g}$ is 2.\\
    Importantly, obstacles to extend this result for over $\mathfrak{g}$ of rank 3, arise in the following manner:
      \begin{example}\label{Example for weights inside JV}
      $\mathfrak{g} = \mathfrak{sl}_4(\mathbb{C})$ and $\mathcal{I}= \big\{1,2,3\big\}$ (successive nodes adjacent in the Dynkin diagram).
      So $f_1f_3$ and $f_2$ do not commute.
      Consider $M(0)\ \twoheadrightarrow\  V\  = \  \frac{M(0)}{  U(\mathfrak{n}^-)\ \big\{ \ f_1f_3\cdot m_0, \ f_2\cdot m_0\  \big\}}$,
       $J_V =  \mathcal{I}$.  
      \end{example}
      \end{cor}
      Writing $\wt V$ in Example \ref{Example for weights inside JV} might help us to understand $\wt_{J_V}V$ for every $V$ in view of \eqref{Minkowski formula for supp of holes}, for the bigger goal of computing all $\wt V$. 
         Next, in our up-coming follow-up work on weights of modules $V$ over general Borcherds--Kac--Moody (BKM) algebras -- and particularly in extending the formulas for simple $V$ of Khare et al. to BKM setting -- the below were crucial: (1)~Identifying the candidates for parabolic Vermas over BKM $\mathfrak{g}$. 
      (2)~For these modules, writing Minkowski decompositions as above (w.r.t. the analogue of $J_V$ there-in) for computing their weights explicitly.  
      
      With Theorems~\ref{thmA}--\ref{thmC} solving Problem~\ref{Problem freeness and weight for all V}, we conclude our above discussion on weights-side, with a  question on the cones in Theorem~\ref{thmC}, and turn our focus to the weight spaces in $V$:
      \begin{problem}\label{Problem all free root subsets for wt V}
           For which root-subsets $S\subseteq \Delta^+$, \quad $\wt V \ - \ \mathbb{Z}_{\geq 0}S\ \subseteq \ \wt V$?
      \end{problem}
      \subsection*{Along the free-directions: weight vectors and multiplicities}\qquad\qquad 

            We now study a version of the above freeness questions at module level.
      Aligned on these lines, our final main result Proposition~\ref{Enumeration proposition- freeness at module level} proved in Section \ref{Section 7} (also Theorem \ref{Corollary local Weyl group invariance 2} in Section \ref{Section 6}) is an application of working with $J_V$ and of our proof for result \eqref{Eqn all free directions for V}, as mentioned in Remark~\ref{Remark Inspirations for Problem 1(b)}): Lower bounds on the multiplicities  of weights that are in the free $J_V^c$-directions for $V$ (for $\wt V$). 
      
     We begin with a technical lemma \big(see proof of Theorem D \cite{Khare_JA}\big) important in Khare's above works, for showing {\it integrable slice decompositions} for parabolic Vermas \eqref{Int. slice. decomp. PVM} and simples \eqref{wt formula for simples}:
     \begin{lemma}[\cite{Khare_JA}]\label{Lemma freeness at module level}
     Let $\mathfrak{p}_{J_{\lambda}}^-$ be the negative part of $\mathfrak{p}_{J_{\lambda}}$. Fix $\lambda\in \mathfrak{h}^*$, and suppose $|J_{\lambda}^c|=n$. Then-
      \begin{align*}\label{non-vanishing wt vectors in simple}
      \tag{C00}
      \begin{rcases}
      \begin{aligned}
      &\text{for}\textbf{ any enumeration }     J_{\lambda}^c = \{ i(1),\ldots, i(n) \},\\
      &\text{numbers }\ c(1),\ldots, c(n)\in \mathbb{Z}_{\geq 0}, \text{ and } F\in U(\mathfrak{p}_{J_{\lambda}}^-)\\
      &\text{yielding weight vector }F\cdot L(\lambda)_{\lambda}\neq 0,   
      \end{aligned}           
      \end{rcases}
       \qquad 
      \begin{aligned}
         F\ \cdot\   f_{i(1)}^{c(1)}\ \cdots \ f_{i(n)}^{c(n)}\ \cdot\  L(\lambda)_{\lambda}\ \neq 0.
      \end{aligned}
      \end{align*}
      \end{lemma}
     \eqref{non-vanishing wt vectors in simple} implies the freeness property in Problem \ref{Problem freeness and weight for all V}(b) for $V=L(\lambda)$ and $J=J_{\lambda}^c$; 
      Lemma~\ref{L4.2} in Section \ref{Section proof of thmA} generalizes \eqref{non-vanishing wt vectors in simple} to all $V$.
      \cite{Venkatesh, Shushma} computed the multiplicities of free-roots in BKM Lie super algebras, via chromatic polynomials and graph-colorings, and moreover constructed bases for these root-spaces via super Lyndon heaps.
     Inspired by this, and following our result \eqref{Eqn all free directions for V} and to strengthen it at module level, 
     we ask for some surviving vectors along our free $J_V^c$-directional weight spaces in $V$, in Problem~\ref{Problem non-vanishing vectors in free-directions for any V}(a) below.
    The recipe of enumerations in it is motivated by \eqref{non-vanishing wt vectors in simple}, which solves this question in the special case $V=L(\lambda)$.
            Furthermore, it might be interesting to explore \eqref{non-vanishing wt vectors in simple} for $V=$ parabolic Vermas $M(\lambda,J)$, with arbitrary free-directions $J^c\supseteq J_{\lambda}^c$, and for general $V$; for which we pose Problem \ref{Problem non-vanishing vectors in free-directions for any V} inspired by the proof technique of Lemma \ref{Lemma freeness at module level} in \cite{Khare_JA}. 
       \begin{problem}\label{Problem non-vanishing vectors in free-directions for any V}
       (a) Given $V$ with $n=|J_V^c|>0$, \  for which orderings $J_V^c\ = \ \{i(1),\ldots, i(n)\}$, 
       \[ f_{i(1)}^{c(1)}\ \cdots\  f_{i(n)}^{c(n)}\ \cdot\  v_{\lambda}\ \neq 0\ \quad \text{for  all }\ \big(c(j)\big)_{j\in J_V^c}\neq 0\  ?\]
       (b) Suppose we have one such enumeration. 
       Can we reach back $v_{\lambda}$ to by $\mathfrak{sl}_2$-theory, i.e. -
       \[ e_{j_1}\ \cdots \ e_{j_m}\ \ \cdot \ \  f_{i(1)}^{c(1)}\cdots f_{i(n)}^{c(n)}\cdot v_{\lambda} \ \ \neq 0\ \in \  V_{\lambda}\  \quad \text{for some }\ j_1,\ldots, j_m\in \mathcal{I}\  ?\]
       Or how high can we go above the $\lambda-\sum_{i\in J_V^c}c(i)\alpha_i$-weight space using $e_j$s?
       \end{problem}
       Clearly all orderings -- which work when $J_V^c$ is independent, or when $J_V^c=J_{\lambda}^c$ -- do not work for part (a); see Example \ref{Ex non-Verma V with full wts} where-in $J_V^c=\{1,2\}$.
       Our Proposition \ref{Enumeration proposition- freeness at module level} below solves part (a), with key tool- exploring the question (b).
    \underline{Our applications and motivations behind part (b)}:
    \begin{itemize}
    \item[(1)] Corollaries Proposition \ref{Corollary local Weyl group invariance 1} and Theorem~\ref{Corollary local Weyl group invariance 2}(b) to Theorem~\ref{thmC} in Section~\ref{Section 6}, for  studying simple root strings through weights in $V$ (see Remark~\ref{Remark seeking non-vanishijng weight vectors in free-directions for Theorem integrable strings for weights (c')}), with two broader motivations : 
    i) \cite[Problem 4.8]{WFHWMRS} investigating the continuity of these strings through general $\mu\in \wt V$.
    \end{itemize}
    \begin{problem}
        In particular, does $\mu\pm \alpha_i\in \wt V$ for some $i\in \mathcal{I} $ imply $ \mu \in \wt V$?
    \end{problem}
    \begin{itemize}
    \item[ ]ii) Study of the top weights in these strings, see \cite[Proposition 4.11]{WFHWMRS}.  \item[(2)] Writing the $\mathfrak{g}$-submodules of $V$ in Theorem \ref{Corollary local Weyl group invariance 2} for slice decompositions Theorem \ref{theorem J-slice decomp.}.
    \item[(3)] Some bounds for weight-multiplicities along free $J_V^c$-dircetions in Proposition \eqref{Enumeration proposition- freeness at module level} \eqref{Eqn bound of free-directional weight- dimensions}. 
    \item[(4)] Strengthening the well-known (existential) result from Kac's book \cite{Kac}:
    \end{itemize}
       \begin{lemma}[{\cite[Lemma 9.6]{Kac}}]\label{Lemma local composition series for V} 
       For $\mathfrak{g}$-modules $V$ (or any module in category $\mathcal{O}$) and $\mu\in \wt V$, we have a filtration $V_{r+1}\ \subset\  V_{r}\ \subset\  \cdots \ \subset\  V_1=V$ such that- i) subfactors $\frac{V_k}{V_{k+1}}\ \simeq $ $L(\lambda_k)$ whenever $\mu\in \wt \Big(\frac{V_k}{V_{k+1}}\Big)$ for $k\leq r$, ii) with $(V_r)_{\mu}=0$ and iii) so $\dim V_{\mu}\ = \ \sum\limits_{k \text{ s.t. }\mu\in \wt L(\lambda_k)} \dim L(\lambda_k)_{\mu}$.
       \end{lemma}
       \begin{problem}\label{Problem determining simple containing a given weight}
           In the above filtration for $V$, can one determine all the simples $L(\lambda_k)$ containing any fixed weight $\mu$?
           A weaker version of this- given $\mu$, can we write-down examples of such $L(\lambda_k)$s?
       \end{problem}
       Our result \eqref{Eqn bound of free-directional weight- dimensions} in Proposition \ref{Enumeration proposition- freeness at module level} answers this weaker problem for $\mu$ along the free directions.   
     For stating these, we develop algorithms for the orderings of $J=J_V^c$ sought-for in Problem \ref{Problem non-vanishing vectors in free-directions for any V}: \\ Let $P^+$ denote the $\mathbb{Z}_{\geq 0}$-cone of dominant integral weights in the integral weight lattice $P\subset \mathfrak{h}^*$.  
       \begin{algorithm}\label{Algorithm for enumerations}
          Let $\mathfrak{g}$ be a Kac--Moody algebra, $\lambda\in P^+$ and $M(\lambda)\twoheadrightarrow V$ with $J_V^c\neq \emptyset$, and a sequence  of numbers $c(i)\in \mathbb{Z}_{\geq 0}$ $\forall$ $i\in J_V^c$. \quad 
           We proceed in the below steps:\\ 
          \big(For arbitrary $\lambda\in \mathfrak{h}^*$ the procedure is more involved, and is showed in Algorithm \ref{Section 6 Algorithm for enumerations} in Section~\ref{Section 7}.\big)
           \begin{itemize}
               \item[Step 0.] 
           We set $I_1= J_V^c$.
           While $I_1$ is independent, we choose (distinct) $i_1,\ldots, i_{|J_V^c|}\in I_1$ arbitrarily.\\
               If $I_1$ is not independent, we proceed to Step 1. 
               \item[Step 1.] We pick any two adjacent nodes in (the Dynkin subdiagram of) $I_1$, say $a_1,a_2\in I_1$.\\ We choose $i_1\in \{a_1,\ a_2\}$ such that $c(i_1)=\min \{c(a_1), c(a_2) \}$.
               \item[Step 2.] We update $I_2= I_1\setminus \{i_1\}$. We repeat (Step 0 and) Step 1 for $I_1$ in place of $I_0$, for node `$i_2$'.\\
               \hspace*{1cm}\vdots
               \end{itemize}
               \hspace*{1.25cm} At the $k^{\text{th}}$-step, $I_k=I_{k-1}\setminus \{i_{k-1}\}= J_V^c\setminus \{i_1,\ldots, i_{k-1}\} $.\\
               The algorithm terminates after $k$-steps, once $I_k$ is independent, yielding $i_1,\ldots, i_{k-1}$, and we then arbitrarily choose (distinct) $i_{k+1},\ldots,  i_{|J_V^c|}\in I_k$. 
       \end{algorithm}
       \begin{prop}\label{Enumeration proposition- freeness at module level}
           let $\mathfrak{g}$ be a Kac--Moody algebra, $\lambda\in \mathfrak{h}^*$ and $M(\lambda)\twoheadrightarrow V$.
           Assume that $J_V^c\neq \emptyset$, and fix numbers $c(i)\in \mathbb{Z}_{\geq 0}$ $\forall$ $i\in J_V^c$. 
           Let ` $\bullet$ '  denote the {\it dot}-action of $W$ on $\mathfrak{h}^*$.
           \begin{enumerate}
           \item[(a)] Here and for result (b) below, we assume that $\lambda\in P^+$.
           Then for any enumeration $J_V^c=\{i_1,\ldots, i_n\}$ given by Algorithm \ref{Algorithm for enumerations}, 
           \[
           f_{i_1}^{c(i_1)}\cdots f_{i_n}^{c(i_n)}\ \cdot V_{\lambda}\ \ \neq\  0.
           \]
           \item[(b)] 
           Recall, enumerations in Algorithm \ref{Algorithm for enumerations} always end in independent sets (possibly $\emptyset$). So:
           \begin{itemize}
          \item[1] In this procedure(s), say $H_1', \ldots, H_r'$ are the independent sets that we obtain.
          \item[2] Correspond to each $H_t'$ above, we define $H_t: = \{i\in H_t'\ |\ c(i)> \lambda(\alpha_i^{\vee})\}$ (possibly $\emptyset$) $\forall\ t$.
          \item[3] Let $H_1,\ldots, H_s$ be the distinct independent sets among $H_1,\ldots , H_r$. 
          \end{itemize}
          $\text{Then for the given weight }\ \mu\ =\ \lambda-\sum\limits_{i\in J_V^c}c(i)\alpha_i\ = \ \lambda-\sum_{t=1}^n c(i_t)\alpha_{i_t}$ :
          \begin{equation}\label{Eqn bound of free-directional weight- dimensions}
         \hspace*{2cm}  \dim\big( V_{\mu} \big) \ \ \geq \ \ \sum_{t=1}^s\underbrace{\dim\left[ L\big( w_{H_t}\ \bullet \lambda \big)\right]_{{\large\mu}}}_{>0}. \qquad \quad \text{Where }\ w_{H_t}:=\prod\limits_{i\in H_t}s_i.           \end{equation}
           Above (1) we can replace each $\mathfrak{g}$-simple $L\big(w_{H_t}\bullet \lambda\big)$ by the simple $L_{J_V^c\setminus H_t}\big(w_{H_t}\bullet \lambda\big)$ over the subalgebra $\mathfrak{g}_{J_V^c\setminus H_t}$, via the orderings in Algorithm \ref{Algorithm for enumerations}; and importantly note (2) the $\mu$-multiplicities are positive in all of these simples \big(addressing Problem \ref{Problem determining simple containing a given weight}\big).
           \item[(c$'$)]
           Now for general $\lambda\in \mathfrak{h}^*$, as a first step, applying Algorithm \ref{Algorithm for enumerations} for $J_{\lambda}\cap J_V^c$ (in place of $J_V^c$ there-in), yields us enumeration(s) $J_{\lambda}\cap J_V^c\ =\ \{i_1,\ldots, i_m\}$, for $m=|J_{\lambda}\cap J_V^c|\geq 0$ so that
               \[
              \underbrace{f_{i_1}^{c(i_i)}\cdots f_{i_m}^{c(i_m)}}_{J_{\lambda}\cap J_V^c}\ \cdot\ \bigg(\prod_{i\in J_\lambda^c\cap J_V^c }f_i^{c(i)}\bigg)\ \cdot V_{\lambda}\ \ \neq \ 0.
               \]
               Above, we multiply $f_i^{c(i)}$s from $J_{\lambda}^c\cap J_V^c$ in any fixed order, in view of \eqref{non-vanishing wt vectors in simple}.
           \item[(c)] Here are our complete results in the general case of $\lambda\in \mathfrak{h}^*$.
          We use the more technical Algorithm \ref{Section 6 Algorithm for enumerations} in Section \ref{Section 7}, for enumerations $J_V^c = \{j_1,\ldots, j_n\}$.
          Let all the possible terminating independent sets $H_t$s be defined as there-in.
          Then the results analogous to (a) and (b) above hold true for any $\mathfrak{g}$-module $M(\lambda)\twoheadrightarrow V$ and any highest weight $\lambda\in \mathfrak{h}^*$. 
           \end{enumerate}
       \end{prop}
This proposition is proved in Section \ref{Section 7}. Above, if for no $i\in H_1\cup \cdots\cup H_r$, $c(i)>\lambda\big(\alpha_i^{\vee}\big)$, then we set $s=1$ and $H_1'=\emptyset$, and then \eqref{Eqn bound of free-directional weight- dimensions} shows that $\mu\in \wt L(\lambda)$, recovering the traditional bound: \quad  \[
     \dim\big(V_{\mu}\big)\ \ \geq \ \ \dim\big(L(\lambda)_{\mu}\big).
     \]
       Next, these bounds might be attainable in smaller cases of $\mu$ with weight-multiplicities 1 or 2.
       It might be interesting to explore the settings where-in these bounds can be attained.
       \begin{remark}
       (a) Consider a composition series for $M(\lambda)\twoheadrightarrow V$ over finite type $\mathfrak{g}$, or its local version in the Kac--Moody setting by Lemma \ref{Lemma local composition series for V}.
    By Theorem \ref{thmC} (and Theorem \ref{Corollary local Weyl group invariance 2}), treating the \underline{finitely many} simple sub-factors $L(\lambda')$ in the series, with $\lambda'= w\bullet \lambda$ for $w\in W$ a product of some commuting reflections $s_i$, might be helpful to study $\wt V$ and $J_V^c$-directional weight spaces in $V$.\\ 
       (b) How can one extend and show result \eqref{non-vanishing wt vectors in simple} for general $V$ for any $J=J_V^c$ and $F\neq 1 \in U\big(\mathfrak{p}_{J_V^c}^-\big)$?\\
       (c) The above problem might be hard, due to non-empty $H_t$s arising.
       Nevertheless, such $H_t$s yield us applications: Proposition \ref{Enumeration proposition- freeness at module level}(b) multiplicity-bound \eqref{Eqn bound of free-directional weight- dimensions}.\\
       (d) If $H_1=\cdots = H_s=\emptyset$, Proposition \ref{Enumeration proposition- freeness at module level} yields the trivial solution $L(\lambda_k)=L(\lambda)$ to Problem~\ref{Problem determining simple containing a given weight}. 
       In this case, it might be interesting to explore if originally there is a lower $L(\lambda_k)$ containing $\mu$. 
      \end{remark}
      With the above descriptions, motivations and applications of our results, we conclude Section~\ref{Section 1}.
       
\section{Notations and Preliminaries}
To prove the results of this paper, we need the following notations, facts and preliminary results: \\
The results noted in Subsection \ref{Subsec. useful results} at the end of this section are important in our proofs.  
\subsection{Notations for vector spaces}\label{Subsection 2.1}
	\begin{itemize}
	\item $\mathbb{N}$, $\mathbb{Z}$, $\mathbb{R}$ and $\mathbb{C}$ : The sets of natural numbers, integers, real numbers and complex numbers.
	\item $\mathbb{R}_{\geq0}, \mathbb{R}_{>0}$, $\mathbb{Z}_{\geq 0}$ and $\mathbb{Z}_{>0}$ : The sets of non-negative and positive reals, similarly for integers.
	\item For $S\subseteq\mathbb{R}$ and $B$ subset of an $\mathbb{R}$-vector space, the $S$-linear combinations of elements in $B$ : 
	\[
	SB:=\left\{\sum_{j=1}^{n}r_jb_j\text{ }\Bigg|\text{ }n\in \mathbb{N},\text{ } b_j\in B,\ r_j\in S\text{ }\ \forall  \text{ }j\in [n]\right\}.
	\]
	\item {\bf Minkowski sum}, and resp. {\bf difference}, of any two sets $C$ and $D$ in an abelian group : $C\pm D:=\{c\pm d$ $|$ $c \in C,\text{ }d\in D\}$.\hspace*{1cm} $x\pm D$ : Shorthand notation for $\{x\}\pm D$.
\end{itemize}
   	\subsection{Notations for Kac--Moody algebras $\mathfrak{g}$}
   	\begin{itemize}
   	\item $\mathfrak{g}=\mathfrak{g}(A)$ : A fixed Kac--Moody $\mathbb{C}$-Lie algebra, corresponding to a generalized Cartan matrix $A$ assumed always to be finite-sized for simplicity.
    Our results hold for infinite $A$ as well.
   	\item $\mathfrak{h}$, $(\mathfrak{h},\Pi ,\Pi^{\vee})$, $\mathfrak{n}^-\oplus\mathfrak{h}\oplus\mathfrak{n}^+$, $\Delta$, $U(\mathfrak{g})$ : A fixed Cartan subalgebra, a realisation, the
	triangular decomposition and the root system w.r.t $\mathfrak{h}$, and the universal enveloping algebra for $\mathfrak{g}$.\\
	Whenever we make 
	additional assumptions on $\mathfrak{g}$, such as $\mathfrak{g}$ is indecomposable or semisimple or of finite/affine/indefinite type, we will clearly mention them.
	\item Indexing set $\mathcal{I}$ : The vertices/nodes in the Dynkin diagram for $A$ or 
	$\mathfrak{g}$; for simplicity, we assume $\mathcal{I}$ is finite. 
	We also fix an ordering in $\mathcal{I}$, and use $\mathcal{I}$ to index the rows and columns of $A$, and hence also to index the simple roots, simple co-roots, as well as the simple reflections. \\
   Ex: when $\mathfrak{g}=\mathfrak{sl}_n(\mathbb{C})$, $\mathcal{I}=\{1,\ldots, n-1\}$, and when $\mathfrak{g}=\text{ (affine) } \widehat{\mathfrak{sl}_n(\mathbb{C})}$ $\mathcal{I}=\{0,1,\ldots, n-1\}$.
	\item $\Pi :=\{\alpha_i $ $|$ $i\in \mathcal{I} \}$ and $\Pi^{\vee}:=\{\alpha_i^{\vee}$ $|$ 
	$i\in\mathcal{I}\}$ : The sets of simple roots and of simple co-roots.
	\item $\Delta^+$ and $\Delta^-$ : The subsets of positive and negative roots of $\Delta$ w.r.t $\Pi$.
	\item $e_i 
	,f_i,\alpha_i^{\vee}$ $\forall$ $i\in\mathcal{I}$ : The Chevalley generators for $\mathfrak{g}$; they generate the derived subalgebra~$[\mathfrak{g},\mathfrak{g}]$. 
	\item $W=\langle\{s_i $ $|$ $i\in \mathcal{I}\}\rangle$ : The Weyl group of $\mathfrak{g}$ generated by the simple reflections $s_i$ (about the plane perpendicular to $\alpha_i$) $\forall$ $i\in \mathcal{I}$. 
	\item  $\Delta^{re}=W\Pi$ and $\Delta^{im}$ : The subsets of real  and imaginary roots of $\Delta$. 
 \item $P$, $P^+$ : the integral weight lattice in $\mathfrak{h}^*$, the cone of dominant integral weights in it.
	\item $(\cdot | \cdot)$ : A fixed non-degenerate symmetric invariant bilinear form for symmetrizable $\mathfrak{g}$.
 \item 	The (usual) partial order `$\ \prec\ $' on $\mathfrak{h}^*$ : $x\prec y \in \mathfrak{h}^* \iff y -x \in \mathbb{Z}_{\geq 0}\Pi$. 
         \item The height (function on) for a vector $x=\sum_{i\in \mathcal{I}}c_i \alpha_i$ for some $c_i \in \mathbb{C}$ :
	$\height(x):=\sum_{i\in \mathcal{I}}c_i$.
 \item The {\it dot}-action of $W$ on $\mathfrak{h}^*$ : Fix any $\rho\in \mathfrak{h}^*$ with $\rho(\alpha_i^{\vee})=1$ $\forall$ $i\in \mathcal{I}$. Then for any $\lambda\in \mathfrak{h}^*$ and $w\in W$, $w\bullet (\lambda)\ :=\ w(\lambda+\rho)-\rho$.
	\item Intervals of weights :	For $ \mu\in\mathfrak{h}^*,\ \alpha\in \Delta \text{ and } k\in \mathbb{Z}_{\geq0}$, we define
	\begin{equation}\label{E2.5}
  \hspace*{0.5cm} \big[\mu-k\alpha,\text{ }\mu\big]:=\{\mu-j\alpha\text{ }\big|\text{ }0\leq j\leq k\}.\qquad (\text{Not to be confused with the Lie bracket in }\mathfrak{g}.)
  \end{equation}
	   	\end{itemize}
     \begin{remark}\label{R2.9} (1)  Recall, if $\bar{\mathfrak{g}}(A)$ is the Lie algebra generated by $e_i, f_i, \mathfrak{h}$ modulo only the {\it Serre relations}, then $\mathfrak{g}=\mathfrak{g}(A)$ is the further quotient of $\overline{\mathfrak{g}}(A)$ by the largest ideal not intersecting $\mathfrak{h}$.\\
     (2) Our results, while stated and proved over $\mathfrak{g}=\mathfrak{g}(A)$, hold equally well uniformly over any ``intermediate'' Lie algebra $\overline{\mathfrak{g}}(A)\twoheadrightarrow\tilde{\mathfrak{g}}\twoheadrightarrow\mathfrak{g}=\mathfrak{g}(A)$.
	The results of \cite{Dhillon_arXiv} holding over all such $\tilde{\mathfrak{g}}$ clarifies this: 
	namely, $\Delta$, and $\wt M(\lambda,J)$ and $\wt L(\lambda)$ for fixed $\lambda,J$, do not change with varying $\tilde{\mathfrak{g}}$. 
	  \end{remark}
\subsection{Notations for Kac--Moody and parabolic $\mathfrak{g}$-subalgebras}
 Our study of highest weight $\mathfrak{g}$-modules $V$ involves working with the actions of Kac--Moody subalgebras of $\mathfrak{g}$ on $V$, and below are some notations for them.
 In what follows, we fix $\emptyset \neq I\subseteq \mathcal{I}$.
	\begin{itemize}
	    \item $A_{I\times I}$ and $\mathfrak{g}_I:=\mathfrak{g}(A_{I\times I})$ : The submatrix of $A$ determined by $I$-rows and $I$-columns, and we identify the corresponding Kac-Moody Lie algebra $\mathfrak{g}\left(A_{I\times I}\right)$ \big(by Kac's book \cite{Kac}, Exercise~1.2\big) with a $\mathfrak{g}$-subalgebra denoted $\mathfrak{g}_I$, having the below attributes (similar to $\mathfrak{g}$).
      \item $\mathfrak{l}_I:= \mathfrak{g}_{I}+\mathfrak{h}$ and $\mathfrak{p}_I:=\mathfrak{g}_I +\mathfrak{h}+\mathfrak{n}^+$ : Standard Levi and parabolic $\mathfrak{g}$-subalgebras of $\mathfrak{g}$ for $I$. 
	   	    \item $\mathfrak{h}_I \subseteq \mathfrak{h}$ and $(\mathfrak{h}_I, \Pi_I,\Pi_I^{\vee})$ : In $\mathfrak{g}$, a suitable Cartan subalgebra (in $\mathfrak{h}$) and realisation for $\mathfrak{g}_I$.  
	   	     \item $\Pi_{I}:=\{\alpha_i$ $|$ $i\in I\}\subseteq \Pi=\Pi_{\mathcal{I}}$ and $\Pi_I^{\vee}:=\{\alpha_i^{\vee}$ $|$ $i\in I\}\subseteq \Pi^{\vee}$ : The sets of simple roots and of simple co-roots of $\mathfrak{g}_I$. 
		   	   \item $\Delta_{I}:=\Delta\cap\mathbb{Z}\Pi_I$ : The root system of $\mathfrak{g}_I$ = the subroot system of $\Delta$ generated by $\Pi_I$. 
		   	   \item $W_I=\langle\{ s_i\ | \ i\in I\}\rangle$, and $e_i ,f_i ,\alpha^{\vee}_i$, $\forall$ $i\in I$ : The parabolic subgroup of $W$ generated by $s_i$ $\forall$ $i\in I$, which is the Weyl group for $\mathfrak{g}_I$, and the Chevalley generators of $\mathfrak{g}_I$.
	  \item For $I=\emptyset$ : $\Pi_I=\Pi_I^{\vee}=\Delta_{I}=\emptyset$, $\mathfrak{g}_I=\mathfrak{h}_I=\{0\}$, and $W_I=\{1\}$ the trivial subgroup of $W$. 
	\end{itemize}
	\subsection{Highest weight $\mathfrak{g}$-modules} We fix (any) Kac--Moody $\mathfrak{g}$ and weight $\lambda\in \mathfrak{h}^*$.
	\begin{itemize} 
	\item Weights of an $\mathfrak{h}$-module $M$ : The $\lambda$-weight space, and the weight-set of $M$ are-
		\[ 
	M_{\lambda}:=\{m\in M\text{ }|\text{ }h\cdot m= \lambda(h)m\text{ }\forall\text{ }h\in \mathfrak{h}\},\quad\text{and}\quad \wt M:=\big\{\mu\in\mathfrak{h}^*\text{ }|\text{ }M_{\mu}\neq\{0\}\big\}.
	\]
	\item Weight modules : $\mathfrak{h}$-modules $M$  such that $M=\bigoplus\limits_{\mu\in\wt M}M_{\mu}$.
	And a $\mathfrak{g}$-module $M$ is called a weight module if it is a weight module over $\mathfrak{h}$.
	\item The character of a weight $\mathfrak{h}$-module $M$ : When each weight space of $M$ is finite-dimensional, we define $\text{char}M:=\sum\limits_{\mu\in\wt M}\dim(M_{\mu})e^{\mu}$ to be the formal character of $M$. 
	\item $M(\lambda)$ and $L(\lambda)$ : Verma module and (its unique simple quotient) the simple highest weight module over $\mathfrak{g}$, with highest weight $\lambda$. 
    $m_{\lambda}\ \neq 0 \in M(\lambda)_{\lambda}$ = highest weight vector in $M(\lambda)$. 
		\item $M(\lambda)\twoheadrightarrow V$ \big($M(\lambda)$ surjecting onto $V$\big) : Shorthand notation for a {\bf non-trivial highest weight $\mathfrak{g}$-module $V$} with the highest weight $\lambda$; 
  Thus $M(\lambda)\twoheadrightarrow V \twoheadrightarrow L(\lambda)$. 
  $v_{\lambda}\neq 0\in  V_{\lambda}$ is a highest weight vector in $V$. So, we have the trivial weight-bounds for $V$:
  \[
  \wt L(\lambda)\ \ \subseteq \ \ \wt V \ \ \subseteq \ \ \wt M(\lambda)\ = \ \lambda -\mathbb{Z}_{\geq 0}\Pi. 
  \]
  \end{itemize}
  \begin{remark}
  We aim at exploring ways to go closer to general $\wt V$, focusing on the below relations and notions for modules $V$: i) freeness property \eqref{Eqn freeness in PVMs}; ii) Definitions \ref{Defn of N(lambda,J)}, \ref{Defn support of all holes of V} in Section~\ref{Section 1}; iii) Definition \ref{Defn of support of all holes in Section 5} in Section \ref{Section 5} (which were not studied in the literature earlier to our knowledge).
  \end{remark}
  \begin{itemize}
    \item  Let $V$ be a highest weight $\mathfrak{p}_I$ or $\mathfrak{l}_I$-module with highest weight $\lambda\in\mathfrak{h}^*$.
    $V$ is a highest weight $\mathfrak{g}_I$-module, and its highest weight the restriction of $\lambda$ to $\mathfrak{h}_I$, is denoted by $\lambda$ for simplicity. 
         \end{itemize}
         \begin{definition}\label{Defn wts wrt I}
         Fix a $\mathfrak{g}$-module $M(\lambda)\twoheadrightarrow V$, and $I\subseteq \mathcal{I}$.
         Let $\mathfrak{s}=\mathfrak{p}_I$ or $\mathfrak{l}_I$ \big(or $\mathfrak{g}_I$\big).
        Then:
         \begin{equation}\label{Eqn wts wrt I}
             \wt \big(U(\mathfrak{s})\cdot v_{\lambda}\big)\ = \   \wt_I V \ \ :=\ \  \wt V\ \cap\ [\lambda\ -\ \mathbb{Z}_{\geq 0}\Pi_I]\quad \subset \mathfrak{h}^* \ \ \big(\text{resp. } \mathfrak{h}_I^*\big). 
         \end{equation}
         \end{definition}
  \subsection*{Integrablility}
		We begin with the set of all {\it integrable nodes/directions} for $\lambda$ (or for $L(\lambda)$) :
		\begin{equation}\label{E1.2}
	J_{\lambda}:=\left\{i \in \mathcal{I} \text{ }|\text{ } \langle \lambda ,\alpha^{\vee}_i\rangle \in \mathbb{Z}_{\geq0}\right\}.
	\end{equation}
  $\langle\lambda,\alpha_i^{\vee}\rangle\ = \ \lambda\big(\alpha_i^{\vee}\big)$ is the pairing/evaluation of $\lambda\in \mathfrak{h}^*$ at $\alpha_i^{\vee}\in \mathfrak{h}$.\qquad 
 $\lambda\in P^+$ $\iff$ $J_{\lambda}=\mathcal{I}$.
 \begin{definition}\label{Defn integrability of V}
The {\it integrability} or the set of integrable directions for $M(\lambda)\twoheadrightarrow V$ : \begin{equation}	
	I_V:=\left\{i\in \mathcal{I}\text{ }\Big|\text{ }\langle\lambda,\alpha_i^{\vee}\rangle\in\mathbb{Z}_{\geq0}\text{ and }f_i^{\langle\lambda,\alpha_i^{\vee}\rangle+1}v_{\lambda}=\{0\}\right\}\qquad \subseteq \ J_{\lambda}.
		\end{equation}
  The integrabilities of simples $L(\lambda)$ and of parabolic Vermas $M(\lambda,J)$ (Definition \ref{D2.1})\ : $J_{\lambda}$ and $J$.  
  \end{definition}
  \begin{lemma}[Folklore]\label{L2.2}
\begin{itemize}
\item[(1)] Let $j\in \mathcal{I}$. Then $f_j$ acts locally nilpotently on $V$ if and only if $f_j$ acts nilpotently on the highest weight space $V_{\lambda}$.
\item[(2)] $V$ is $\mathfrak{g}_J$-integrable, iff $f_j$ ($e_j$ anyway does) acts locally nilpotently on $V$ $\forall$ $j\in J$, iff the character of $V$ is $W_J$-invariant.
\item[(3)] Thus, $V$ is $\mathfrak{g}_{I_V}$-integrable, and moreover, $\wt V$ is $W_{I_V}$-invariant.
\end{itemize}
\end{lemma}
\noindent 
Integrability completely determines the weights of all simple $L(\lambda)$s, and weight-hulls for all $V$ \big(formulas \eqref{Wts of simples as PVM}--\eqref{Eqn conv wt V} above and Theorem~\ref{T2.3} below\big). 
\subsection{Parabolic Vermas, and on weights of $V$} \label{Subsection PVMs}
Parabolic Verma $\mathfrak{g}$-modules were introduced and studied in finite type by Lepowsky \cite{lepo1, L_JA}, and in the Kac--Moody setting \cite{GaLe, KuBGG, WaRo}.
See Humphrey's book \cite[Chapter 9]{Hump_BGG} for their properties. 
Their weights are central for us.
\begin{definition}\label{D2.1}
   (1) Fix $\lambda\in\mathfrak{h}^*$ and $\emptyset\neq J\subseteq J_{\lambda}$. 
 Let $L^{\max}_J(\lambda)$ and $L_J(\lambda)$ be the largest integrable highest weight module and and highest weight simple, with highest weight(s) $\lambda$, over $\mathfrak{g}_J$ 
	\big(equivalently over $\mathfrak{p}_J=(\mathfrak{h}+\mathfrak{g}_J)\oplus (\bigoplus_{\alpha\in \Delta^+\setminus\Delta^+_J}\mathfrak{g}_{\alpha})$, via the natural action of $\mathfrak{h}$ and trivial action of its nilradical $\bigoplus_{\alpha\in\Delta^+\setminus\Delta^+_J} \mathfrak{g}_{\alpha}$\big).
  By \cite[Corollary 10.4]{Kac}, when $\mathfrak{g}_J$ is symmetrizable, $L^{\max}_J(\lambda)=L_J(\lambda)$.\smallskip\\ 
    (2) For $J\subseteq J_{\lambda}$, the
	parabolic Verma $\mathfrak{g}$-module with highest weight and integrability $\big(\lambda, J\big)$:
		\[
		M(\lambda,J)\ \ \  :=\ \ \  U(\mathfrak{g})\otimes_{U(\mathfrak{p}_J)}L^{\max}_J(\lambda)\ \ \ 
		\underset{\mathfrak{g}-mod}{\simeq}\ \ \  M(\lambda)\bigg/\bigg(\sum_{j\in
		J}U(\mathfrak{g})f_{j}^{\langle\lambda,\alpha_j^{\vee}\rangle+1}m_{\lambda}\bigg).\]
   $L^{\max}_J(\lambda)$ = 
		parabolic $\mathfrak{g}_J$-Verma 
		corresponding to $(\lambda, J)$.
   $M(\lambda)\twoheadrightarrow V$ is $\mathfrak{g}_J$-integrable $\iff$ $M(\lambda,J) \twoheadrightarrow V$.
\end{definition}
In our results using the parabolic Vermas $M(\lambda,J)$, or more precisely their weights:
We make use of the $\mathfrak{g}_{J}$-integrability of $M(\lambda,J)$, and the following important slice-formulas for its weights (crucial for Khare et al. for obtaining weights of simples \eqref{Wts of simples as PVM}); in addition to the Minkowski difference formula \eqref{Eqn Mink diff PVM} 
(See e.g. \cite[Proposition 3.7 and Section 4]{Dhillon_arXiv} for their proofs and applications):
\begin{equation}\label{Int. slice. decomp. PVM}
\textbf{Integrable slice decomposition:}\hspace*{1cm}\wt M(\lambda,J)= \bigsqcup\limits_{\xi\ \in\ \mathbb{Z}_{\geq0}\Pi_{J^c}}\wt L^{\max}_J(\lambda-\xi). \hspace*{1.5cm} 
\end{equation}
In Theorem \ref{theorem J-slice decomp.} in Section \ref{Section 6}, we try to prove similar slice-decompositions more generally for any $M(\lambda)\twoheadrightarrow V$ along any $J_V\subseteq J \subseteq J_{\lambda}$.
Theorem \ref{T2.3} below, the complete result of Theorem \ref{Theorem wts of simples and hulls of all V} by Dhillon--Khare, while being inspirational for our paper, extends at the same time our results (Ex: Theorem \ref{thmB}) involving parabolic Vermas, for a large class of $\mathfrak{g}$-modules $V$ including all simple $V$. 
\begin{theorem}[{\cite[Theorems 2.9, 2.10]{Dhillon_arXiv} \& \cite[Theorem 3.13]{Khare_Ad}}]\label{T2.3}
Fix $\mathfrak{g}=\mathfrak{g}(A)$, $\lambda\in\mathfrak{h}^*$ and $J \subseteq J_{\lambda}$.
\begin{itemize}
\item[(a)] All $\mathfrak{g}$-modules $M(\lambda)\twoheadrightarrow V$ with integrability $J$, have the same weights, \ \ if and only if \\
the Dynkin subdiagram on $J_{\lambda}\setminus J$ is complete, i.e.,  $\langle\alpha_i,\alpha_j^{\vee}\rangle =A_{ji} \neq0,$ $\forall$ $i, j \in J_{\lambda}\setminus J$. 
\begin{equation}\label{wt formula for simples}
\wt L(\lambda) \ \ =\ \ \wt M(\lambda, J_{\lambda}) \ \ =  \wt L^{\max}_{J_{\lambda}}(\lambda)\ -\ \mathbb{Z}_{\geq0}(\Delta^+\setminus\Delta^+_{J_{\lambda}})\ \  = \ \ \bigsqcup\limits_{\xi\ \in\ \mathbb{Z}_{\geq0}\Pi_{J_{\lambda}^c}}\wt L^{\max}_{J_{\lambda}}(\lambda-\xi).
\end{equation}
\item[(b)] For $M(\lambda)\twoheadrightarrow V$, the following ``first-order'' data are equivalent:
\[
(1)\ \  I_V.\qquad\qquad  (2)\ \ \conv_{\mathbb{R}}(\wt V).\qquad \qquad (3)\ \ 
\text{The stabilizer of }\ \conv_{\mathbb{R}}(\wt V)\  \text{ in }\ W.
\]
Above, the hull and stabilizers are that of $M(\lambda, I_V )$ and the parabolic subgroup $W_{I_V}$.
\end{itemize}
\end{theorem}
		\subsection{Useful Tools}\label{Subsec. useful results}
Here are important lemmas for proving our results on weights: 
	\begin{lemma}\label{Lemma reaching lambda from mu}
	 Let $\mathfrak{g}$ be a Kac--Moody algebra, and fix $\lambda\in\mathfrak{h}^*$, $M(\lambda)\twoheadrightarrow V$ and $\mu\in \wt V$. Then:
  \begin{itemize}
  \item[(a)] The weight space $V_{\mu}$ is spanned by the weight vectors of the form- 
		\begin{equation}\label{Eqn seq. of simple roots from mu to lambda at module level}
	 	f_{i_1}\ \cdots\  f_{i_n}\ \cdot\ v_{\lambda}\qquad\quad \text{for }\ f_{i_j}\in\mathfrak{g}_{-\alpha_{i_j}}\  \text{ and } \ i_j\in\mathcal{I} \ \forall\ j\in \{1,\ldots, n\},\ \ \text{ and }\ \ \lambda-\mu= \sum_{j=1}^{n}\alpha_{i_j}.
		\end{equation}
  \item[(b)] Hence, for every $\mu\in\wt V$, there exists a sequence of simple roots $\alpha_{i_1},\ldots, \alpha_{i_n} \in \Pi$, such that $\lambda-\sum\limits_{t=1}^k\alpha_{i_t} \ \in\  \wt V $ \ $\forall$ $k\leq n$ (each ``partial sum'' gives rise to a weight), and $\mu= \lambda-\sum\limits_{t=1}^n \alpha_{i_t}$.
  \end{itemize}
	\end{lemma}
 An important lemma on weight-strings for going-up by $e_i$s, which will be employed in almost every proof in this paper: 
  \quad \big(Recall the interval notation in \eqref{E2.5}.\big) 
 \begin{lemma}\label{L3.1}
		Let $\mathfrak{g}$ and $V$ be as above.
		Fix a real root $\alpha\in \Delta$ and a weight $\mu\in\wt V$.
		\begin{equation}
  \langle\mu,\ \alpha^{\vee}\rangle\geq 0 \qquad \implies \qquad \big[\mu-\lceil\langle\mu,\ \alpha^{\vee}\rangle\rceil\alpha,\ \mu\big]\ \subseteq\ \wt V.
  \end{equation}
\begin{equation} 
  \text{So}, \qquad s_{\alpha}\mu\ \ =\ \ \mu-\langle\mu,\ \alpha^{\vee}\rangle\alpha \ \ \in \ \wt V \quad \text{when}\quad \langle\mu,\ \alpha^{\vee}\rangle\in\mathbb{Z}_{\geq 0}.
   \end{equation}
   Where $\lceil \cdot\rceil$ = lowest integer function on $\mathbb{Z}$, and  $s_{\alpha}$ = reflection about the hyperplane $\textnormal{kernel}(\alpha)$.  
	\end{lemma}
	\begin{proof}[\textnormal{\textbf{Proof}}]
		Let $\mu$ and $\alpha$ be as in the lemma.
  Fix $e\neq 0\in \mathfrak{g}_{\alpha}$ and $f\neq 0
		\in \mathfrak{g}_{-\alpha}$ -- recall, $\dim(\mathfrak{g}_{\pm\alpha})=1$ since $\alpha$ is real --  and a weight vector $x\neq 0\in V_{\mu}$.
		Recall, $e$ acts locally nilpotently on $V$, and nilpotently on $x$. 
		So, we fix $n\in\mathbb{Z}_{\geq 0}$ with $e^n x\neq 0$ and $e^{n+1}x=0$.
		Now $U(\mathfrak{g}_{-\alpha})(e^{n}x)$ is a highest weight module over $\mathfrak{sl}_{\alpha}=\mathfrak{g}_{-\alpha}\oplus \mathbb{C}\alpha^{\vee}\oplus\mathfrak{g}_{\alpha}\ \big(\simeq \mathfrak{sl}_2(\mathbb{C})\big)$, with the highest weight $\mu+n\alpha$ \big(or $\langle\mu+n\alpha,\ \alpha^{\vee}\rangle$\big). Since $\langle\mu,\alpha^{\vee}\rangle\geq 0$, note that if $\langle\mu+n\alpha,\alpha^{\vee}\rangle=\langle\mu,\alpha^{\vee}\rangle+2n\in\mathbb{Z}_{\geq 0}$, then $\langle\mu,\alpha^{\vee}\rangle\in\mathbb{Z}_{\geq 0}$, implying $\lceil\langle\mu,\alpha^{\vee}\rangle\rceil=\langle\mu,\alpha^{\vee}\rangle$. 
  Observe by the standard calculation over $\mathfrak{sl}_2(\mathbb{C})$, $f^{k}(e^nx)\neq 0$ for any: 1)~$k\in\mathbb{Z}_{\geq 0}$, if $\langle\mu,\alpha^{\vee}\rangle+2n\notin \mathbb{Z}_{\geq 0}$; 2)~$k\in \big\{0,\ldots,  \langle\mu,\alpha^{\vee}\rangle+2n\big\}$, if $\langle\mu,\alpha^{\vee}\rangle+2n\in\mathbb{Z}_{\geq 0}$. 
  This is because $e^n x\neq 0$ is a maximal vector for $e$, and so it generates a Verma module or a $\big(\mu(\alpha^{\vee})+2n+1\big)$-dim. simple over $\mathfrak{sl}_{\alpha}$, where-in clearly this phenomenon is true.
  Now,
		\[f^k(e^nx)\neq 0\quad \text{for all }\ k\in\big\{n,\ldots, n+\lceil \langle\mu,\alpha^{\vee}\rangle\rceil\big\} \ \ \implies\ \  \big[\mu-\lceil\langle\mu,\alpha^{\vee}\rangle\rceil\alpha,\ \mu\big]\subseteq \wt V. \vspace*{-0.55cm}
		\]
	\end{proof}
 \begin{remark}
   The above proof relies on going-up from $V_{\mu}$ by $e_i$s, which worked under the hypothesis of the lemma that $\mu$ is a weight of the ``top part'' or the simple/integrable quotient (over $\mathfrak{sl}_{\alpha_i}$) of $U(\mathfrak{sl}_{\alpha_i}) V_{\mu+n\alpha}$.
   In general, $\mu$ need not be in the top part of any $\alpha_{i_1}$-string, Ex: 
     \[
\mathfrak{g}=\mathfrak{sl}_{\alpha_{i_1}}\ \simeq \mathfrak{sl}_2(\mathbb{C}),\ V=M(0) \ \simeq \mathbb{C}[f_{i_1}], \ \mu = \lambda - 2\alpha_{i_1}. \text{ Then }e_{i_1}^2 \cdot f_{i_2}^2\cdot m_0\ = \ -2e_{i_1}\cdot f_{i_1}\cdot m_0  \ = \ 0.
\]
 (1) Such cases arising make it difficult to work at the module level by the actions of $e_is$ and $f_i$s; thus, exploring Lemma 4.7 case (c) in \cite{WFHWMRS} might be interesting.
 So we work based on (induction on) heights of weights here.
(2) On the other hand, along free-directions, these raising techniques work, as explored in Proposition \ref{Enumeration proposition- freeness at module level}, in Theorem \ref{Corollary local Weyl group invariance 2} and in the discussion in Section \ref{Section 7}.
 \end{remark}
 Next, we recall the {\it Parabolic Partial sum property} of Kac--Moody root systems from a previous work \cite{Teja_ArXiv} (and from extended abstract \cite{Teja_Fpsac})  which generalizes the well-known {\it partial sum property} of root systems and yields a minimal description for weights of all parabolic Verma and simple $V$, as shown there-in.
Now in the present work, it is a key tool in the proof of Theorem \ref{thmA}.
\begin{itemize}
\item We begin by extending the height function. 
For $I\subseteq \mathcal{I}$:
\begin{equation}\label{Eqn defn. I-height}
\height_I\bigg(\sum_{i\in \mathcal{I}} c_i\alpha_i\bigg)\ \  =\ \  \sum_{i\in I}c_i\qquad \forall\ \ c_i\in \mathbb{C}.
\end{equation}
\end{itemize}
 \begin{definition}[P-PSP]\label{p-psp}
  	A Kac--Moody root system $\Delta$ is said to have the \textit{Parabolic-partial sum property} if:
   given any $\emptyset\neq I\subseteq\mathcal{I}$ and a root $\beta\in\Delta^+$ with $\height_I(\beta)>1$, there exists a root $\gamma\in \Delta_{I,1}:=\{\alpha \in \Delta \text{ }|\text{ } \height_I(\alpha)=1\}$ such that $\beta -\gamma \in \Delta$.
   \end{definition}
   \begin{theorem}[{\cite[Theorem A]{Teja_ArXiv} and \cite{Teja_Fpsac}}]\label{Thm P-Psp} Let $\mathfrak{g}, \Delta, \lambda, J$ as above. The P-PSP holds in all Kac--Moody root systems.
   \begin{equation}
       \Delta_{I, 1} \ \ \text{ is the smallest generating set for the cone }\ \ \mathbb{Z}_{\geq 0}\left(\Delta^+\setminus \Delta_I^+\right).\  \text{Thereby,}
   \end{equation}
     \begin{eqnarray}
     \text{Minimal descriptions }:\qquad \quad  \wt M(\lambda, J) & \  = \ \ \wt L^{\max}_J(\lambda)\ - \ \mathbb{Z}_{\geq 0}\Delta_{J^c, \ 1},\qquad \\
       \wt L(\lambda) \ &  = \ \ \wt L^{\max}_{J_{\lambda}}(\lambda)\ -\ \mathbb{Z}_{\geq 0}\Delta_{J_{\lambda}^c, \ 1}.\label{Minimal description for wt of simples by p-psp}
   \end{eqnarray}
   \end{theorem}
   Here is a version of Lemma \ref{Lemma reaching lambda from mu}(b) and of the P-PSP for roots subsets $X=\Delta_{I,1}$, for going-up inside their cones, which is needed in proving our weight-formulas (Theorems \ref{thmA} and \ref{thmC}):
   \begin{lemma}\label{Lemma PSP in unit I-ht roots}
       Let $\mathfrak{g}, \Delta, \emptyset \neq I\subseteq \mathcal{I}$ be as above. 
       Fix a root $\beta\in \Delta_{I,1}$ with $\height_{I^c}(\beta)>0$. 
       There exists $j\in I^c \cap \supp(\beta)$ with $\beta-\alpha_j\ \in \Delta_{I,1}$.
       In turn, we have a chain of roots in $\Delta_{I,1}$ below $\beta$: 
       \[
       \beta_1\ = \ \beta\ \succneqq\  \beta_1\ \succneqq \ \cdots\  \succneqq \ \beta_m\ \in \Delta_{I,1},\quad \text{with }\ m= \height(\beta),\ \  \beta_{t}-\beta_{t+1}\in \Pi_{I^c}\ \forall\ t.  
       \]
       
       \end{lemma}
       \begin{proof}
           Let $I$ and $\beta$ be as in the lemma, and $\{i\}=\supp(\beta)\cap I$.
           We show the result for roots in $\Delta_{I,-1}=\big\{ \alpha\in
           \Delta\ \big|\ \height_I(\alpha)=-1\big\}$; by symmetry this proves the lemma. 
           Note, $f_i$ is a maximal vector for $\mathfrak{n}_{I^c}^+$-action, and it generates an integrable highest weight $\mathfrak{g}_{I^c}$-module $M=U\big(\mathfrak{n}^-_{I^c}\big) \cdot f_i$ in $\mathfrak{g}$, with highest weight $-\alpha_i\big|_{\mathfrak{h}_{I^c}^*}$.
           Now $\wt M\ = \ \Delta_{I,-1}$ \big(in $\mathfrak{h}_{I^c}^*$, and so in $\mathfrak{h}^*$\big), which is easily seen as follows.
           Any root $\gamma\in \Delta_{I,-1}$ can be $W_{I^c}$-conjugated to a dominant integral weight $\gamma'$ in $\mathfrak{h}_{I^c}$; i.e., $\gamma'(\alpha_j^{\vee})\in \mathbb{Z}_{\geq 0}$ $\forall$ $j\in I^c$.
           Now $\gamma'\preceq -\alpha_i\ \in \Delta_{I,-1}$ and  $\supp(\gamma')$ is connected.
          So by weight-formula {\color{black}\cite[Proposition 11.2]{Kac}} for integrable $V$,  $\Delta_{I,-1}\ = \wt M\ = \ \wt L_{I^c}(-\alpha_i) \ =\ \wt L_{I^c}^{\max}(-\alpha_i)$.
           Finally, apply Lemma \ref{Lemma reaching lambda from mu} for \big($\mathfrak{g}_{I^c},\ M,\ \beta\in \wt M$\big) to complete the proof. 
       \end{proof}
    With all the above notations and tools in-hand, we begin writing the proofs of our results. 
   \section{Proof of Theorem \ref{thmA}: First Minkowski difference weight-formula for $V$}\label{Section proof of thmA}
	The goal of this section is to show the following uniform weight-formula (Theorem \ref{thmA}) for all highest weight $\mathfrak{g}$-modules $M(\lambda)\twoheadrightarrow V$, for all $\lambda\in \mathfrak{h}^*$: \ \ \  \big(See Definition \ref{Defn wts wrt I} for the notation $\wt_{J_{\lambda}}V$.\big)
		\begin{equation}\label{Eqn refined Min. formula via P-PSP}
	\qquad\wt V \ \ =\ \  \wt _{J_{\lambda}}V\ -\ \mathbb{Z}_{\geq 0}\Delta_{J_{\lambda}^c,1}\ \ =\ \ \wt _{J_{\lambda}}V-\mathbb{Z}_{\geq 0}(\Delta^+\setminus \Delta_{J_{\lambda}}^+). \quad \big(\text{Extending }\eqref{wt formula for simples}, \eqref{Minimal description for wt of simples by p-psp}.\big) 
		\end{equation}
 	\begin{remark}
 		Evidently by the above formulas, for studying the weights of arbitrary modules $M(\lambda)\twoheadrightarrow{}V$, it suffices to assume $\lambda$ to be dominant and integral; in which case, $J_{\lambda}=\mathcal{I}$ and all the $\mathbb{Z}_{\geq 0}$-cones in \eqref{wt formula for simples} and \eqref{Eqn refined Min. formula via P-PSP} are 0.
   Philosophically, \eqref{Eqn refined Min. formula via P-PSP} re-assures that this classical assumption, 
 		 which helps in understanding integrable $V$ in a first course on representation theory of Lie algebras, is once again all that is needed in determining the weights of general $V$! 
    This is also supported by the fact $J_{\lambda}=\emptyset\implies \wt V = \wt M(\lambda)$ for all $M(\lambda)\twoheadrightarrow V$ (by Lemma \ref{Lemma freeness at module level} or Theorem~\ref{thmB}).
 		\end{remark}
 	A key result needed in the proofs of \eqref{Eqn refined Min. formula via P-PSP} and Theorem \ref{thmB}, is the following lemma, which also generalizes the enumeration result Lemma \ref{Lemma freeness at module level} of Khare to all $V$ mentioned in the introduction: 
\begin{lemma}\label{L4.2}
Let $\mathfrak{g}$ be a Kac--Moody algebra and $\lambda\in\mathfrak{h}^*$, and $M(\lambda)\twoheadrightarrow{ }V$. Then 
\[\hspace*{2cm} \mu-\sum\limits_{i\in J_{\lambda}^c}c_i\alpha_i\ \ \in\ \wt V\qquad\quad \text{ for any } \mu\in \wt_{J_{\lambda}}V \text{ and sequences }(c_i)_{i\in J_{\lambda}^c}\in(\mathbb{Z}_{\geq 0})^{|J_{\lambda}^c|}.  \]
\end{lemma}
\begin{proof}[\textnormal{\textbf{Proof}}]
Our approach is similar to that in the proof of Theorem 5.1 in \cite{Dhillon_arXiv}. 
Fix $\mu$ and $c_i$s as in the lemma, and a homogeneous element $F\in U(\mathfrak{n}^-)$ with $0\neq Fv_{\lambda}\in V_{\mu}$.
We show $\big(F\prod_{i\in J_{\lambda}^c}f_i^{c_i}\big)v_{\lambda}\neq 0$, to prove the lemma; 
note, $f_i^{0}=e_i^{0}:=1\in U(\mathfrak{g})$.
Consider $x=\big(\prod_{i\in J_{\lambda}^c}e_i^{c_i}\big)\big(F\prod_{i\in J_{\lambda}^c}f_i^{c_i}\big)v_{\lambda}$;
$e_i$s and $f_is$ are multiplied in the (any) same order. 
$e_i$ and $f_j$ commute $\forall$ $i\neq j\in \mathcal{I}$, so $x=(F\prod_{i\in J_{\lambda}^c}e_i^{c_i}f_i^{c_i})v_{\lambda}$. 
Now each $e_i^{c_i}f_i^{c_i}$ successively act on $v_{\lambda}$ by a non-zero scalar \big(as in \eqref{Eqn Proof of main proposition (b) sl2 calculation}\big) by $\mathfrak{sl}_2$-theory and as $i\in J_{\lambda}^c$.
Thus, $Fv_{\lambda}\neq 0$ implies $x\neq 0$, thereby $(F\prod_{i\in J_{\lambda}}f_i^{c_i}v_{\lambda})\neq 0$ in $V$. 
\end{proof}
 The forward inclusions in our Minkowski decompositions in Theorems \ref{thmA} and \ref{thmC} are easy to see:
 \begin{lemma}\label{Lemma forward inclusion in wt formulas}
     In the above notation, for any $J\subseteq \mathcal{I}$: 
		\[\wt V\ \ \subseteq\ \  \wt _J V-\mathbb{Z}_{\geq 0}(\Delta^+\setminus \Delta_J^+)\ \ =\ \  \wt_J V\ -\ \mathbb{Z}_{\geq 0}\Delta_{J^c,1}.\] 
 \end{lemma}
\begin{proof}[\textnormal{\textbf{Proof}}]
    The result is easily seen, using a PBW basis of $U(\mathfrak{n}^-)$ with $U(\mathfrak{n}^-_J)$-terms at the right end -- via $U(\mathfrak{n}^-)=U\left(\bigoplus_{\alpha\in \Delta^-\setminus \Delta_J^- }\mathfrak{g}_{\alpha}\right)\ \otimes\ U(\mathfrak{n}^-_J)$ -- and by using Theorem \ref{Thm P-Psp} for the second equality.
 \end{proof}
 Thus, the question of the reverse inclusions to the above is natural, solved below:
	\begin{proof}[\textnormal{\textbf{Proof of Theorem \ref{thmA}}}]
Let $\mathfrak{g}, \lambda$ and $M(\lambda)\twoheadrightarrow{ }V$ be as in the theorem. 
Observe that the reverse implication (for any $J$), in the characterization result \eqref{Eqn A2} recalled below, is obvious.
\begin{equation}\label{Eqn all Mink decomp. characterization equivalence}
\wt_JV\ -\ \mathbb{Z}_{\geq 0}\Pi_{J^c} \ \ \subset\  \wt V\qquad \iff \qquad \wt V\ \ =\ \   \wt_JV\ -\ \mathbb{Z}_{\geq 0}\Delta_{J^c,1}.
\end{equation}
The below lines prove formula~\eqref{Eqn A1} for $J=J_{\lambda}$; simultaneously, the forward implication in \eqref{Eqn A2} can be proved traversing along the same steps.
Formula \eqref{Eqn A1} is manifest in the extreme cases: (i) $J_{\lambda}=\emptyset$, as $\lambda-\mathbb{Z}_{\geq 0}\Pi=\wt L(\lambda)\subset \wt V$; (ii) $J_{\lambda}=\mathcal{I}$, as $\mathbb{Z}_{\geq 0}\Delta_{J^c_{\lambda},1}=\{0\}$.
So we assume throughout $\emptyset\neq J_{\lambda}\subsetneqq \mathcal{I}$.
We show in view of Lemma \ref{Lemma forward inclusion in wt formulas}: $\wt_{J_{\lambda}}V\  -\ \mathbb{Z}_{\geq 0}\big(\Delta^+\setminus \Delta^+_{J_{\lambda}} \big)\ \ \subseteq \wt V$, via proving for any $\gamma_1,\ldots, \gamma_n\ \in\Delta_{J_{\lambda}^c,1}$, $n\in\mathbb{N}$ \ \ \big(for result \eqref{Eqn all Mink decomp. characterization equivalence}, replace $J_{\lambda}$ here by $J$\big) - 
\begin{equation}\label{Eqn proof forward inclusion (A1)}
\hspace*{1.5cm} \wt_{J_{\lambda}}V\ -\ \sum\limits_{t=1}^n\gamma_t\ \ \subset\  \wt V\qquad\quad 
\text{by induction on }\text{ } \height_{J_{\lambda}}\bigg(\sum_{t=1}^n\gamma_t\bigg)\geq 0.\end{equation}
Base step: $\height_{J_{\lambda}}(\sum_{t=1}^n\gamma_t)=0$, and so $\gamma_t\in\Pi_{J_{\lambda}^c}$ $\forall t$. Therefore, the result follows by Lemma \ref{L4.2}.
For the base step for forward implication in \eqref{Eqn all Mink decomp. characterization equivalence} (any $J$), note we already assumed it in the hypothesis.\\
 Induction step: Let $\mu\in\wt_{J_{\lambda}}V$, and $\gamma_1,\ldots,\gamma_n\in\Delta_{J_{\lambda}^c,1}$ such that $\height_{J_{\lambda}}(\sum_{t=1}^n\gamma_t)$ $>0$. The result follows once we prove that $\mu-\sum_{t=1}^n\gamma_t\in\wt V$. Without loss of generality we will assume that $\height_{J_{\lambda}}(\gamma_1)>0$.
 By Lemma \ref{Lemma PSP in unit I-ht roots}, pick $j\in J_{\lambda}$ such that $\gamma_1-\alpha_j\in\Delta_{J_{\lambda}^c,1}$.
Note by the induction hypothesis applied to $\gamma_1-\alpha_j+\sum_{t= 2}^n\gamma_t$ that $\wt_{J_{\lambda}}V-(\gamma_1-\alpha_j+\sum_{t= 2}^n\gamma_t)\subset \wt V$ (when $n=1$, treat the term $\sum_{t=2}^n\gamma_t$ as 0 throughout the proof). Now, if $\mu-\alpha_j\in\wt_{J_{\lambda}} V$, then by the previous sentence $\mu-\sum_{t=1}^n\gamma_t=\mu-\alpha_j-(\gamma_1-\alpha_j+\sum_{t= 2}^n\gamma_t)\in\wt V$, and therefore we are done.\\
So, we assume that $\mu-\alpha_j\notin\wt V$, which forces $f_jV_{\mu}=\{0\}$. 
Thus, the $\mathfrak{g}_{\{j\}}$-module $U(\mathfrak{g}_{\{j\}})V_{\mu}$ is finite-dimensional by $\mathfrak{sl}_2$-theory, and hence it is $\mathfrak{g}_{\{j\}}$-integrable. Define $\Tilde{\mu}$ and $\Tilde{\gamma_t}$ as follows.
    \begin{align*}
    \begin{aligned}
    \Tilde{\mu}=
    \begin{cases}
    \mu&\text{if } \langle\mu,\alpha_j^{\vee}\rangle\geq 0,\\
    s_j\mu &\text{if }\langle\mu,\alpha_j^{\vee}\rangle<0,
    \end{cases}
    \end{aligned}\quad\quad
    \begin{aligned}
    \Tilde{\gamma_1}=
    \begin{cases}
    s_j\gamma_1&\text{if }\langle\gamma_1,\alpha_j^{\vee}\rangle>0,\\
    \gamma_1-\alpha_j&\text{if }\langle\gamma_1,\alpha_j^{\vee}\rangle\leq0,
    \end{cases}
    \end{aligned}\quad\quad
    \begin{aligned}[t]
    &\Tilde{\gamma_t}=
    \begin{cases}
    \gamma_t &\text{if }\langle\gamma_t,\alpha_j^{\vee}\rangle\leq 0,\\
    s_j\gamma_t &\text{if }\langle\gamma_t,\alpha_j^{\vee}\rangle> 0, 
    \end{cases}\\
    &\text{ }\text{when }n\geq 2\text{ and }2\leq t\leq n.
    \end{aligned}
\end{align*}
One verifies the below properties of $\tilde{\mu}$ and $\tilde{\gamma_t}$:\ \ 
(a) $s_j\mu\in\wt V$, by the $\mathfrak{g}_{\{j\}}$-integrability of $U(\mathfrak{g}_{\{j\}})V_{\mu}$, and so $\Tilde{\mu}\in\wt V$. \  \ 
(b) $\Tilde{\gamma_t}\preceq \gamma_t\in\Delta_{J_{\lambda}^c,1}$ $\forall$ $t$, and $\Tilde{\gamma_1}\precneqq\gamma_1$. So, $\height_{J_{\lambda}}(\sum_{t=1}^n\Tilde{\gamma_t})<\height_{J_{\lambda}}(\sum_{t=1}^n\gamma_t)$.
 \ \ (c)~$\langle\Tilde{\mu},\alpha_j^{\vee}\rangle\geq 0$, $\langle\Tilde{\gamma_1},\alpha_j^{\vee}\rangle<0$ and $\langle\Tilde{\gamma_t},\alpha_j^{\vee}\rangle\leq 0$ $\forall$ $t\geq 2$. 
 \   \  (d) $\mu\in [s_j\Tilde{\mu},\ \Tilde{\mu}]$ and $\gamma_t\in[\Tilde{\gamma_t},\ s_j\Tilde{\gamma_t}]$ $\forall$ $t$.

Consider $\Tilde{\mu}-\sum_{t=1}^n\Tilde{\gamma_t}$. In view of points (a) and (b), the induction hypothesis yields $\Tilde{\mu}-\sum_{t=1}^n\tilde{\gamma_t}\in\wt V$. By the $\mathfrak{g}_{\{j\}}$-action on $V_{\Tilde{\mu}-\sum_{t=1}^n\Tilde{\gamma_t}}$, $\mathfrak{sl}_2$-theory and by point (d), Lemma \ref{L3.1} says
\[\mu\ -\ \sum_{t=1}^n\gamma_t\ \ \in\  \left[s_j\bigg(\Tilde{\mu}\ -\ \sum_{t=1}^n\Tilde{\gamma_t}\bigg),\ \text{ }\Tilde{\mu}\ -\ \sum_{t=1}^n\Tilde{\gamma_t}\right] \ \ \subset\ \wt V.\]
Hence the proof of formula \eqref{Eqn A1} \big(also of formula \eqref{Eqn A2}\big) is complete.
\end{proof}
Observe, any Minkowski decomposition as in \eqref{Eqn A1} inside $\wt_{J_{\lambda}}V$, easily extends to whole $\wt V$:
\begin{cor}\label{Remark Mink. decom. inside Jlambda}
Fix $\mathfrak{g}, \lambda\in\mathfrak{h}^*$, $V$ as above, and $J\subseteq J_{\lambda}$.
Suppose in the $\mathfrak{g}_{J_{\lambda}}$-submodule $U(\mathfrak{g}_{J_{\lambda}})v_{\lambda}$,
\begin{equation}\label{Eqn Mink. decom. in Jlambda}
\wt_{J_{\lambda}}V\ \ =\ \  \wt_{J}V\ -\ \mathbb{Z}_{\geq 0}\big(\Delta_{J_{\lambda}}^+\setminus  \Delta_J^+\big).
\end{equation}
By the P-PSP,\  $\Delta_{J^c_{\lambda},1}\subset \mathbb{Z}_{\geq0}\Delta_{J^c,1}$. 
Now observe by formula \eqref{Eqn A1} and by Lemma \ref{Lemma forward inclusion in wt formulas}, equation \eqref{Eqn Mink. decom. in Jlambda} extends to a Minkowski decomposition for whole $\wt V$: 
\[
\wt V\ \ =\ \  \wt_{J}V \ - \ \mathbb{Z}_{\geq 0}\Delta_{J^c,1}\ \ =\ \ \wt_{J}V\ -\ \mathbb{Z}_{\geq 0}\big(\Delta^+\setminus\Delta_J^+\big).
\]
By \eqref{Eqn A2}, this result holds even if we start with the weaker hypothesis $\wt_J V - \mathbb{Z}_{\geq 0}\Pi_{J_{\lambda}\setminus J}\ \subseteq \wt V$. 
\end{cor}
			\section{Proof of Theorem \ref{thmB}: All highest weight modules $V$ with given weight-sets}\label{S4}
	 Here we prove our second result Theorem \ref{thmB}: Given the weight-set $X$ of a parabolic Verma $M=M(\lambda,J)$ (or $M=M(\lambda)$, $L(\lambda)$), we find all highest weight $\mathfrak{g}$-modules $V$ with $\wt V= X$; 
	 for all Kac--Moody $\mathfrak{g}$ and all $\lambda\in \mathfrak{h}^*$.
  We show the first part \eqref{Theorem B V with full weights} in it, and using this we prove the second part \eqref{Theorem Binterval for PVM wts}; after Observation \ref{O4.1} elaborately discussing about the $\mathfrak{g}$-submodules $N(\lambda, J)\subset M(\lambda)$.
  
 Fix (any) $\lambda\in \mathfrak{h}^*$, and $\emptyset\neq J\subseteq J_{\lambda}$ for the whole section. 
  For $j\in J_{\lambda}$ we define $m_j:=\langle\lambda,\alpha_j^{\vee}\rangle+1\in\mathbb{Z}_{>0}$.
	 Recall the interval notation for weight-strings from \eqref{E2.5}.
	We freely use henceforth the fact: 
\begin{fact} 
If $S,T$ are two weight modules over $\mathfrak{g}$, then $\wt (S + T)\ =\ \wt S\ \cup\ \wt T$.
\end{fact}
Recall: i)~defining properties \eqref{0property 2} and \eqref{0property 1} of $\mathfrak{g}$-submodules $N(\lambda)$ and resp. $N(\lambda,J)$ of $M(\lambda)$ in Definition \ref{Defn of N(lambda,J)}; and ii)~the corresponding largest and smallest quotients $V_J^{\max}(\lambda)\ =\ M(\lambda, J)$ and $V_J^{\min}(\lambda)\ = \ \frac{M(\lambda)}{N(\lambda,J)}$ in Theorem \ref{thmB}. 
 By the above fact, $N(\lambda,J)$ \big(similarly, $N(\lambda)$\big) exists, and equals the sum of all the $\mathfrak{g}$-submodules of $M(\lambda)$ each of whose weight-sets satisfy the property \eqref{0property 1} \big(resp. \eqref{0property 2}\big), in place of $\wt N(\lambda,J)$ \big(resp. $\wt N(\lambda)$\big). 
 The non-triviality of $N(\lambda,J)$s (if $J\neq \emptyset$) is easily seen by the $\mathfrak{g}$-submodule $\sum\limits_{j\in J}U(\mathfrak{g})f_j^{\langle\lambda,\alpha_j^{\vee}\rangle+1}m_{\lambda}\ \subset N(\lambda,J)$ satisfying \eqref{0property 1}. 
\begin{observation}\label{O4.1}
(1) If the Dynkin diagram on $\mathfrak{g}$ has no edges ($\mathcal{I}$ is independent), $N(\lambda)=\{0\}$.\smallskip\newline
(2) $N(\lambda)\ \subseteq\ N(\lambda,J)\ \forall\ J\subseteq J_{\lambda}$, and moreover $N(\lambda,J)\ \subseteq \ N(\lambda,J')$ for all $J\subseteq J'\subseteq J_{\lambda}$. 
\smallskip\newline
(3) Generalizing Example \ref{Ex non-Verma V with full wts}: For nodes $i\neq j\in J_{\lambda}$ adjacent in the Dynkin diagram, the maximal vector $f_i^{\langle s_j\bullet\lambda,\alpha_i^{\vee}\rangle+1}f_j^{\langle\lambda,\alpha_j^{\vee}\rangle+1}m_{\lambda}\in M(\lambda)$ generates a $\mathfrak{g}$-submodule $M'\ \simeq  M\big((s_is_j)\bullet \lambda)$ in $M(\lambda)$.
Now, $\{i,j\}\subseteq\supp(\lambda-\mu)$ $\forall$ $\mu\in\wt M'$. 
So, for any $\mu\in\wt M'$ the Dynkin subdiagram on $\supp(\lambda-\mu)$ is a non-isolated graph.
Thus, $M'\xhookrightarrow{ } N(\lambda)$, and hence $N(\lambda)$ is non-trivial in this case.\smallskip\newline
(4) Fix a weight $\mu\preceq \lambda$ with $\supp(\lambda-\mu)$ non-independent.
Suppose $x\in M(\lambda)_{\mu}$ is a maximal vector. Then $x\in N(\lambda)$. 
Ex : over $\mathfrak{g}=\mathfrak{sl}_4(\mathbb{C})$ \big(where $\alpha_{i+1}(\alpha_i^{\vee})\neq 0\ \forall\ i$\big), check $f_3^3 f_2^2 f_1 \cdot m_0 \in N(0)$.
\smallskip\newline
(5) For any $V_J^{\max}(\lambda)\twoheadrightarrow V \twoheadrightarrow V_J^{\min}(\lambda)$, integrability $I_V=J$:  $f_j^{\langle\lambda,\alpha_j^{\vee}\rangle+1}m_{\lambda}\in N(\lambda,J)\iff  j\in J$.
\newline\smallskip
(6) $\mathfrak{sl}_2^{\oplus n}$-theory : Fix a proper $\mathfrak{g}$-submodule $0\neq K\not\subset N(\lambda)$ of $M(\lambda)$.
The definition of $N(\lambda)$ yields a weight $\mu=\lambda-\sum_{j\in J}c_j\alpha_j\in\wt K$ with the Dynkin subdiagram on $J=\supp(\lambda-\mu)$ independent. 
As $M(\lambda)_{\mu}$ is 1-dim., $K_{\mu}  = \mathbb{C}\prod\limits_{j\in J}f_j^{c_j}\cdot  m_{\lambda}$.
i) Clearing-off all the $f_j$s supported over $J_{\lambda}^c$ by  $e_js$ as in Lemma \ref{L4.2} shows $0\neq \prod_{j\in J\cap J_{\lambda}}f_j^{c_j}\cdot m_{\lambda}\in K$.
ii) Next by clearing-off $f_j$s from $I=\big\{ j\in J\  \big|\ \height_{\{j\}}(\lambda-\mu)\leq \lambda(\alpha_j^{\vee}) \in \mathbb{Z}_{\geq 0} \big\}\subseteq J_{\lambda}$, using the standard $\mathfrak{sl}_2$-calculation that $e_j^k f_j^n\cdot m_{\lambda} \in \mathbb{Z}_{>0} f_j^{n-k}\cdot m_{\lambda}$ \big(as in \eqref{Eqn Proof of main proposition (b) sl2 calculation}\big) $\forall$ $0\leq k\leq n \leq \lambda(\alpha_j^{\vee})\in \mathbb{Z}_{\geq 0}$, yields $0\neq \prod_{j\in (J\cap J_{\lambda})\setminus I}f_j^{c_j}\cdot m_{\lambda}\in K$.
Hence, for any proper submodule $K$ of $M(\lambda)$ and any $\mu\in \wt K$ with $J=\supp(\lambda-\mu)$ independent :
\begin{equation}\label{Eqn for simplicity of N(lambda, JLambda)}
J\cap J_{\lambda}\neq \emptyset \ \  \big(\text{else by i) }m_{\lambda}\in K\big),\ \ \ \ \text{moreover } \ \  \left\{ j\in J\cap J_{\lambda}
 \ \big|\  \height_{\{j\}}(\lambda-\mu)> \lambda(\alpha_j^{\vee})  \right\}\  \neq\emptyset. 
\end{equation}
iii) Finally for $j\in (J\cap J_{\lambda})\setminus I$, since $c_j\geq \lambda(\alpha_j^{\vee})+1$, we see that the maximal vector $0\neq \prod_{j\in (J\cap J_{\lambda})\setminus I}f_j^{\lambda(\alpha_j^{\vee})+1}\cdot m_{\lambda}$ is in $K$; using $e_j^k f_j^n \cdot m_{\lambda} \in \mathbb{Z}_{<0} f_j^{n-k}\cdot m_{\lambda}$ $\forall$ $n- \lambda(\alpha_j^{\vee})-1\geq k\geq 0$.
Hence,
\[ 
M\bigg(\prod_{j\in (J\cap J_{\lambda})\setminus I}s_j\ \bullet \lambda\bigg) \xhookrightarrow{} K.\qquad \text{So, }\ \  \bigg(\lambda-\ \ \ \ \sum_{\mathclap{j\in (J\cap J_{\lambda})\setminus I}}\ \ \big(\langle\lambda,\alpha_j^{\vee}\rangle+d_j\big)\alpha_j\bigg)\ \notin \wt\left(\frac{M(\lambda)}{K}\right)\quad \forall \ d_j\in\mathbb{Z}_{>0}.
\]
\smallskip\newline
(7) $N(\lambda,J_{\lambda}) =$ the maximal submodule $K=\mathrm{Ker}\big(M(\lambda)\twoheadrightarrow L(\lambda)\big)$ : 
If $K\subseteq N(\lambda)$, this is manifest.
So, assume that $K\not\subset N(\lambda)$.
As $K\subsetneqq M(\lambda)$, we run through Steps i) and ii) in Point (6) above for $K$ with any weight $\mu\in \wt K$ with independent support, which by \eqref{Eqn for simplicity of N(lambda, JLambda)} shows that $K$ satisfies Property $\eqref{0property 1}$ w.r.t. $J=J_{\lambda}$, thereby $K\subseteq N(\lambda, J_{\lambda})$.  
\end{observation}
We show part \eqref{Theorem B V with full weights} \big(solving Problem \ref{Problem full weights of Vermas}(a), related to Problem \ref{Problem freeness and weight for all V}(b)\big) : All $V$ with full weights.
\begin{proof}[\textnormal{\textbf{Proof of Theorem \ref{thmB} part \eqref{Theorem B V with full weights}:}}] 
Let $\lambda\in\mathfrak{h}^*$ and $0\rightarrow N_V\rightarrow M(\lambda)\rightarrow{ } V\rightarrow0$ be exact, and assume $\wt V=\wt M(\lambda)$. 
Observe, $\mu'\in\wt M(\lambda)\setminus \wt N_V$ whenever $\supp(\lambda-\mu')$ is independent, as $\dim\big(M(\lambda)_{\mu'}\big)=1$ for such $\mu'$. 
By the definition of $N(\lambda)$, this proves that $N_V\subset N(\lambda)$ whenever $\wt V=\wt M(\lambda)$, thereby the forward implication in \eqref{Theorem B V with full weights}. 
Conversely, to show $\wt \left( V_{\emptyset}(\lambda)^{\min}\right)=\wt M(\lambda)$, we prove by induction on $\height(\lambda-\mu)\geq 1$ that $\mu\in\wt M(\lambda)\implies\mu\in \wt \left( V_{\emptyset}(\lambda)^{\min}\right)$.\\
Base step: When $\mu=\lambda-\alpha_t$ for $t\in\mathcal{I}$, as $\{t\}$ is independent, $\mu\notin \wt N(\lambda)$. 
So, $\mu\in \wt \left(V_{\emptyset}^{\min}(\lambda)  \right)$.\\
Induction step: Assume $\height(\lambda-\mu)>1$. Write $\mu=\lambda-\sum\limits_{l\in \supp(\lambda-\mu)}c_l\alpha_l$ for some $c_l\in\mathbb{Z}_{> 0}$. 
If $\supp(\lambda-\mu)\cap J_{\lambda}$ is independent (possibly = $\emptyset$), then by the definition of $N(\lambda)$, clearly
\[
\lambda-\ \ \ \ \ \sum\limits_{\mathclap{l\in \supp(\lambda-\mu)\cap J_{\lambda}}}\ \ \ c_l\alpha_l\ \  \notin \wt N(\lambda)\qquad \implies \qquad \lambda-\ \ \ \ \ \sum\limits_{\mathclap{l\in \supp(\lambda-\mu)\cap J_{\lambda}}}\ \ \ c_l\alpha_l\ \ \in \wt \big(V_{\emptyset}^{\min}(\lambda)\big).
\]
By Lemma \ref{L4.2}, $\mu\in \wt \left( V_{\emptyset}^{\min}(\lambda) \right)$. 
So, we assume hereafter that $\supp(\lambda-\mu)\cap J_{\lambda}$ is non-independent, and let $i,j\in \supp(\lambda-\mu)\cap J_{\lambda}$ be two adjacent nodes in the Dynkin diagram.
Consider $c_i$ and $c_j$, and assume without loss of generality $c_i\geq c_j$. 
As $c_j>0$, $\height(\lambda-\mu-c_j\alpha_j)<\height(\lambda-\mu)$.
So, by the induction hypothesis $\mu+c_j\alpha_j\in \wt \left(V_{\emptyset}^{\min}(\lambda)\right)$.
Now one verifies the following implication.
\[
 \langle\lambda,\alpha_j^{\vee}\rangle\geq 0,\text{ }\langle\alpha_l,\alpha_j^{\vee}\rangle\leq 0\text{ }\forall\text{ }l\in\supp(\lambda-\mu)\setminus\{i,j\}\text{ and }\langle\alpha_i,\alpha_j^{\vee}\rangle\leq -1\implies \langle\mu+c_j\alpha_j,\alpha_j^{\vee}\rangle\geq c_i\geq c_j.\] 
By Lemma \ref{L3.1} \big(for $\mu+c_j\alpha_j$ and $\alpha_j$\big),\  $\mu\in [s_j(\mu+c_j\alpha_j),\  \mu+c_j\alpha_j]\subset \wt\left(V_{\emptyset}^{\min}(\lambda)\right)$, as desired.
\end{proof}
We now show Theorem \ref{thmB}, using (freeness) part \eqref{Theorem B V with full weights} and slice-decomposition \eqref{Int. slice. decomp. PVM}.
\begin{proof}[\textnormal{\textbf{Proof of Theorem \ref{thmB}}}]
Let $\lambda\in\mathfrak{h}^*$. 
We first show the forward implication in the theorem.
For any $M(\lambda)\twoheadrightarrow{ }V'$, it can be easily seen by the definition of $M(\lambda,J)$ and of $I_{V'}$, that 
\[
M(\lambda,J)\twoheadrightarrow{ }V' \ \ \iff \ 
 \ \wt V'\subset \wt M(\lambda,J).
 \]
Next, recall $\wt M(\lambda,J_{\lambda})=\wt L(\lambda)$ by \eqref{wt formula for simples}. 
So $\wt M(\lambda,J_{\lambda})=\wt V'$ for all $M(\lambda,J_{\lambda})\twoheadrightarrow{ }V'$. 
We are done by the above lines if $J=J_{\lambda}$. 
We assume henceforth $J\subsetneq J_{\lambda}$; so $J^c\neq \emptyset$. \\
Fix $V=\frac{M(\lambda)}{N_V}$ with $\wt V=\wt M(\lambda,J)$. 
Then $\mu'\in\wt V$ and $\mu'\notin \wt N_V$ when $\supp(\lambda-\mu')$ is independent in $J^c$; as $\dim(M(\lambda))_{\mu'}=1$ and $\lambda-\mathbb{Z}_{\geq 0}\Pi_{J^c}\subset\wt M(\lambda,J)$ by $J^c$-freeness \eqref{Eqn Mink diff PVM}.
Now by property \eqref{0property 1}, $N_V\subset N(\lambda,J)$ when $\wt V=\wt M(\lambda,J)$, proving the forward implication.

Conversely, we begin by noting by the $\mathfrak{g}_J$-integrability of $V_J^{\min}(\lambda)\ = \  \frac{M(\lambda)}{N(\lambda,J)}$, we have $V_J^{\max}(\lambda)= M(\lambda,J)\twoheadrightarrow V_J^{\min}(\lambda)$.
So, $\wt \big(V_J^{\min}(\lambda)\big)\subseteq \wt M(\lambda,J)$. 
We now show that every slice $\wt L_J(\lambda-\xi)$ in \eqref{Int. slice. decomp. PVM} is contianed in $\wt \big(V_J^{\min}(\lambda)\big)$. 
This shows $\wt \left(V_J^{\min}(\lambda)\right)=\wt M(\lambda,J)$, which immediately proves the reverse implication of \eqref{Theorem Binterval for PVM wts}, since $M(\lambda,J)\twoheadrightarrow{ }V$ when $J\subseteq I_V$. 
For this, we consider $\left[\wt \left(V_J^{\min}(\lambda)\right)\right]\cap(\lambda-\mathbb{Z}_{\geq 0}\Pi_{J^c})$, and the $\mathfrak{g}_{J^c}$-module $N(\lambda,J)\cap  U(\mathfrak{g}_{J^c})m_{\lambda}$.

Observe, $U(\mathfrak{g}_{J^c})m_{\lambda}$ is isomorphic to the Verma $\mathfrak{g}_{J^c}$-submodule with highest weight $\lambda$ (or $\lambda\big|_{\mathfrak{h}_{J^c}^*}$). 
Now analogous to $N(\lambda)$, let $N_{J^c}(\lambda)$ be the largest $\mathfrak{g}_{J^c}$-submodule of $U(\mathfrak{g}_{J^c})m_{\lambda}$ w.r.t. \eqref{0property 2}. 
By part \eqref{Theorem B V with full weights} (for $\mathfrak{g}_{J^c}$), we have $\wt\left[\frac{\left(U(\mathfrak{g}_{J^c})m_{\lambda}\right)}{\left(N_{J^c}(\lambda)\right)}\right]\ =\ \lambda-\mathbb{Z}_{\geq0}\Pi_{J^c}$. 
Now, note by the definition of $N(\lambda,J)$ that if $\mu\in\big(\wt N(\lambda,J)\big)\cap (\lambda-\mathbb{Z}_{\geq 0}\Pi_{J^c})$, 
then $\supp(\lambda-\mu)$ is necessarily non-independent.
So, the $\mathfrak{g}_{J^c}$-submodule $N(\lambda,J)\cap U(\mathfrak{g}_{J^c})m_{\lambda}$ of $U(\mathfrak{g}_{J^c})m_{\lambda}$ satisfies \eqref{0property 2} (over $\mathfrak{g}_{J^c}$). So, \begin{equation}\label{E4.2}
N(\lambda,J)\ \cap\  U(\mathfrak{g}_{J^c})m_{\lambda}\ \ \subseteq\ N_{J^c}(\lambda),\quad \text{ so that }\ \ N(\lambda,J)_{\bar{\mu}}\ \subseteq\  N_{J^c}(\lambda)_{\bar{\mu}}\text{ }\ \ \forall\text{ }\bar{\mu}\in\lambda-\mathbb{Z}_{\geq 0}\Pi_{J^c}.\end{equation}
Recall for any submodule $K\subseteq M(\lambda)$ and $\mu''\in\wt M(\lambda)$, that $\Big(\frac{M(\lambda)}{K}\Big)_{\mu''}$ and $\frac{M(\lambda)_{\mu''}}{K_{\mu''}}$ are isomorphic as vector spaces; as $K$ is a weight module. 
So, by the inclusions \eqref{E4.2} observe for $\bar{\mu}\in\lambda-\mathbb{Z}_{\geq 0}\Pi_{J^c}$:  
\begin{equation*}
\left[\frac{\left(U(\mathfrak{g}_{J^c})m_{\lambda}\right)}{N_{J^c}(\lambda)}\right]_{\bar{\mu}}\ \simeq\  \frac{\big( U(\mathfrak{g}_{J^c})m_{\lambda}\big)_{\bar{\mu}}}{N_{J^c}(\lambda)_{\bar{\mu}}}\ \neq\{0\}\ \ 
\implies\ \  \left(V_J^{\min}(\lambda)\right)_{\bar{\mu}}\ \simeq\  \left[\frac{\left(U(\mathfrak{g}_{J^c})m_{\lambda}\right)}{N(\lambda,J)\cap U(\mathfrak{g}_{J^c})m_{\lambda}}\right]_{\bar{\mu}}\ \neq \{0\}.
\end{equation*}
Hence, $\wt\left( V_J^{\min}(\lambda) \right)\supseteq\lambda-\mathbb{Z}_{\geq 0}\Pi_{J^c}$. 
Now for $\xi\in\mathbb{Z}_{\geq0}\Pi_{J^c}$,  every $0\neq x\in \left(V_J^{\min}(\lambda)\right)_{\lambda-\xi}$ is a maximal vector for $\mathfrak{g}_J$. 
So, $U(\mathfrak{g}_J)x$ is a highest weight $\mathfrak{g}_J$-module with highest weight $\lambda\big|_{\mathfrak{h}_J}$.
Therefore,
\[\wt\left(V_J^{\min}(\lambda)\right)\ \ \supseteq\ \  \wt\left[U(\mathfrak{g}_J)\cdot \left(V_J^{\min}(\lambda)\right)_{\lambda-\xi}\right]\ \ \supset\ \  \wt L_J(\lambda-\xi),\qquad \text{as }U(\mathfrak{g}_J)x\twoheadrightarrow L_J(\lambda-\xi),
\]
for all $\xi\in\mathbb{Z}_{\geq0}\Pi_{J^c}$. 
So, $\wt\left(V_J^{\min}(\lambda)\right)\ \supseteq\ \bigsqcup_{\xi\in\mathbb{Z}_{\geq 0}\Pi_{J^c}}\wt L_J(\lambda-\xi)$, completing the proof.
\end{proof}
\section{Proof of Theorem \ref{thmC}: All Minkowski difference weight-formulas for every $V$}\label{Section 5}
We prove Theorem \ref{thmC}, that yields all possible Minkowski weight-decompositions w.r.t. $J\subseteq \mathcal{I}$ for all $M(\lambda)\twoheadrightarrow V$, via the {\it smallest} set $J^{\min}=J_V$ (Definition \ref{Defn of support of all holes in Section 5}) among such $J$s offering the finest weight-formula in those decompositions.
It strengthens \eqref{Eqn A1}, solving Problem \ref{Problem freeness and weight for all V}(c) recalled here: 
\begin{problem}
    Given $M(\lambda)\twoheadrightarrow V$, for which $J\subseteq \mathcal{I}$ (in view of Corollary \ref{Remark Mink. decom. inside Jlambda}, precisely which $J\subseteq J_{\lambda}$),
    \begin{equation}\label{Eqn Min. decomp. for all V in Section 5}
    \wt V \ \ =\  \ \wt_JV\ - \ \mathbb{Z}_{\geq 0}\left(\Delta^+\setminus \Delta^+_J\right)\ ?
    \end{equation}
    \end{problem}
    To solve this problem for every $V$, one needs to look at deeper relations (Ex : those in the below prototypical example), than the integrability relations $f_i^{\lambda(\alpha_i^{\vee})+1}v_{\lambda} =0$ occurring in $V$ for $i\in I_V$; recall, $I_V$-relations determine such decompositions for hulls $\conv_{\mathbb{R}}(\wt V)$ and their faces.
    \begin{example}\label{Ex type A1xA1}
        Let $\mathfrak{g} \ =\ \mathfrak{sl}_2(\mathbb{C})^{\oplus 2}$ (type $A_1\times A_1$) and $\mathcal{I}=\{1,2\}$. Consider the $\mathfrak{g}$-module
        \[
        V\ = \ \frac{M(0)}{U(\mathfrak{g})\cdot f_1 f_2\cdot m_0 }.
        \]
        It can be easily verified \big(as $U(\mathfrak{n}^-)$ is the polynomial algebra $\mathbb{C}[f_1,  f_2]$\big) that $\wt V\  =\  \big[0  -\mathbb{Z}_{\geq 0}\alpha_1   \big]\ \cup\  \big[ 0- \mathbb{Z}_{\geq 0}\alpha_2 \big]$.
        We have freeness in each direction $J=\{1\}$ and $J=\{2\}$ just from the top weight 0, but not the general property  \eqref{Eqn freeness in PVMs} we began with.
        Thus, $V$ cannot admit the decompositions \eqref{Eqn Min. decomp. for all V in Section 5} for any $J\subsetneqq \{1,2\}$ evidently or by \eqref{Eqn all Mink decomp. characterization equivalence}; and for $J=\{1,2\}$ the decomposition is manifest.
           \end{example}
     Observe such obstacles to decompositions w.r.t. $J$ for $V$ \underline{can arise}, via the higher length relations:
        \[
        I\text{ independent in }\mathcal{I} \ \ \  \text{ with }\ \ \bigg(\prod_{i\in I}f_i^{\lambda(\alpha_i^{\vee})+1}\bigg)\cdot v_{\lambda}\ = \ 0\quad \text{and}\quad I\cap J^c\neq \emptyset.
        \]
    Now accounting for all such obstacles leads us to the node-set $J_V$ \big(introduced in Definition \ref{Defn support of all holes of V} earlier\big) Definition \ref{Defn of support of all holes in Section 5} below; similar to $I_V$ being the ``union of all the unit length relations''.
    By the above observation, identifying $J_V$ simplifies the problem of finding $J$s that support Minkowski decompositions.
    Here is a useful fact:
    \begin{itemize}
 \item $M(\lambda)_{w\bullet \lambda}$, for $w=\ s_{i_1}\cdots s_{i_n}\text{(reduced)\ }\in W_{J_{\lambda}}$ and $i_1,\ldots , i_n\in J_{\lambda}$, has a maximal vector:
 \begin{equation}\label{Eqn maximal vector example}
     f_{i_n}^{\langle s_{i_{n-1}}\cdots s_{i_1}\bullet \lambda,\ \alpha_{i_n}^{\vee}\rangle +1} \cdots f_{i_{t+1}}^{\langle s_{i_t}\cdots s_{i_1}\bullet \lambda,\ \alpha_{i_{t+1}}^{\vee}\rangle +1}\cdots f_{i_1}^{\langle \lambda,\ \alpha_{i_1}^{\vee}\rangle+1 } \cdot m_{\lambda}\ \neq 0. 
 \end{equation}
\end{itemize}
    \begin{definition}\label{Defn of support of all holes in Section 5}
          For $\mathfrak{g}$-module $M(\lambda)\twoheadrightarrow V$, $J_V$ is the {\bf union }of $I\ \subseteq \ J_{\lambda}$ with: 
        \begin{itemize}
            \item[(1)] $I$ is independent; so, $\dim M_{\mu} =1$ $\forall$ $\mu\in \lambda-\mathbb{Z}_{\geq 0}\Pi_I$.\\
            We focus on independent $I$, as the quotients of $M(\lambda)$ along non-independent directions do not loose weights \big(as in Example \ref{Ex non-Verma V with full wts}\big) by Theorem \ref{thmB}\eqref{Theorem B V with full weights}.     
            \item[(2)] $\bigg(\prod\limits_{i\in I}s_i\bigg)\bullet \lambda\ = \  
 \lambda\ -\ \sum\limits_{i\in I}\big( \lambda( \alpha_i^{\vee})+1\big) \alpha_i \ \ \boldsymbol{\notin \wt V}$.\\
 Its weight space in $M(\lambda)$ is maximal and 1-dim., and is contained in  $\text{kernel}\big(M(\lambda)\twoheadrightarrow V\big)$. 
            \item[(3)]  We only need minimal sets $I$ having the property in condition (2) above: \ \ So, \\  $\bigg(\prod\limits_{k\in K}s_k\bigg)\bullet \lambda\ = \ \lambda-\sum\limits_{k\in K}\big(\lambda( \alpha_k^{\vee})+1\big) \alpha_k\ \boldsymbol{\in \wt V}$\quad $\forall$ $K\subsetneqq I$. 
    \end{itemize}
    \end{definition}\begin{example}\label{Ex of JV}
                (1) \underline{Special cases of $J_V$}: Over any Kac--Moody $\mathfrak{g}$, and for any $\lambda\in \mathfrak{h}^*$ and $J\subseteq J_{\lambda}$-
                \[
                J_{M(\lambda)}=\emptyset,\qquad \qquad                  J_{L(\lambda)}= J_{\lambda},\qquad\qquad   J_{M(\lambda,J)}=J.
 \]
  (2) For $V$ in Examples \ref{Ex non-Verma V with full wts} and \ref{Ex type A1xA1}, $J_V \ =\ \emptyset \text{ and resp. } \ J_V=\mathcal{I}$. Below is a slightly bigger example.
 (3) Let $\mathfrak{g}=\mathfrak{sl}_5(\mathbb{C})$, $\mathcal{I}=\{1,2,3,4\}$ (with successive nodes adjacent in the Dynkin diagram), and
                \[   
                V \ = \ \frac{M(\varpi_2)}{ U(\mathfrak{n}^-)\cdot \big\{\ f_1f_3 m_{\varpi_2},\ \  f_2^2 m_{\varpi_2},\ \ f_3^2f_4 m_{\varpi_2}, \ \ f_2^2f_4 m_{\varpi_2}  \ \big\}  },\]
               \ \  \  \ where $\varpi_2$ is the $2^{\text{nd}}$-fundamental weight.\qquad   Then $J_V\ =\ \{1,3\}\cup \{2\}\ = \ \{1,2,3\}$.
            \end{example}
            \begin{lemma}
                Fix any $\mathfrak{g}$-module $M(\lambda)\twoheadrightarrow V$ ($\neq 0$).
                $J_V=\emptyset\ \iff\ \wt V =\wt M(\lambda)$.         
                For $\mu\preceq \lambda$ with $\supp(\lambda-\mu)$ independent, $V_{\mu}=0$  implies $\Big(\prod_{i\in I}s_i\Big)\bullet\lambda
                $-weight space in $V$ is 0 for some $I\subseteq \supp(\lambda-\mu)$; following Observation \ref{O4.1}(6).
                \end{lemma}
              Important in Theorem \ref{thmC} in addition to $J_V$, is the below application of ``$\mathfrak{sl}_3$-theory'': 
\begin{lemma}\label{Lemma Serre}
Fix a Kac--Moody $\mathfrak{g}$, number $N\in \mathbb{Z}_{\geq 0}$, node $i\in \mathcal{I}$, and an independent subset $J\subseteq \mathcal{I}$. 
Then for any integers $c_j> -N\langle\alpha_i,\alpha_j^{\vee}\rangle\geq 0$ $\forall$ $j\in J$, we have:
\[
\Bigg(\prod_{j\in J}f_j^{c_j}\Bigg) f_i^N \ \in \ U(\mathfrak{n}^-)\Bigg(\prod_{j\in J}f_j^{c_j+N\langle\alpha_i,\alpha_j^{\vee}\rangle}\Bigg). 
\]
\end{lemma}

\begin{proof}[\textnormal{\textbf{Proof}}]
We enumerate for convenience $J=\{1,\ldots, n\}$ for $n\in \mathbb{N}$, and let $i$ be as in the lemma.
The proof of the lemma follows upon showing the below identity for any $b_1, \ldots, b_n\in \mathbb{Z}_{\geq 0}$:
\begin{equation}\label{ESerre2}
\left(f_n^{b_n}\cdots f_1^{b_1} \right) f_i^N
 \ =\  \sum_{t_n= 0}^{b_n}\cdots\sum_{t_1=0}^{b_1}
\left[\binom{b_n}{t_n}
\cdots \binom{b_1}{t_1}\right]\ \left[ \text{ad}_{f_n}^{\circ t_n}\circ \cdots \circ \ \text{ad}_{f_1}^{\circ t_1} \big(f_i^N\big)\right] \bigg[\prod_{j=1}^nf_j^{b_j-t_j} \bigg]. 
\end{equation}
\big(\eqref{ESerre2} holds in any associative algebra.\big) 
We work with extended operators $\text{ad}_{f_t} : U(\mathfrak{g}) \rightarrow U(\mathfrak{g})$; note $f_t\otimes f_a\otimes f_b = \ad_{f_t}\big(f_a\otimes f_b\big) + f_a\otimes f_b\otimes f_t = [f_t,f_a]\otimes f_b + f_a\otimes [f_t, f_b] + f_a\otimes f_b\otimes f_t$.
Above $\ad_{f_1},\ldots,\ad_{f_n}$ commute.
One can prove \eqref{ESerre2} by inducting on $\sum_{j=1}^n b_j\geq 0$, using the derivation action of $\ad_{f_j}$ on $U(\mathfrak{n}^-)$ and additivity of the binomial coefficients in the Pascal's triangle.

Observe by applying formula \eqref{ESerre2} for $b_j=c_j> -N\langle\alpha_i,\alpha_j^{\vee}\rangle$, by the Serre-relations $\text{ad}_{f_j}^{-\langle\alpha_i,\alpha_j^{\vee}\rangle+1}$ $f_i=0$ and by $\mathfrak{sl}_2$-theory (raising by $e_i$s) and as $f_1,\ldots, f_n$ commute: 
\[
 \text{ad}_{f_n}^{t_n}\circ \cdots \circ \ \text{ad}_{f_1}^{t_1} \big(f_i^N\big)  \ =\ 0 \quad \text{ if and only if }\ t_j>-N\langle\alpha_i,\alpha_j^{\vee}\rangle \ \text{for some }j\in \{1,\ldots, n\}.
\]
\big(For $N=1$ and $t_j\leq -\alpha_i(\alpha_{t_j}^{\vee})$, this vector is a root vector in the \big($\alpha_i+\sum_{t=1}^n t_j\alpha_j$\big)-th root space in $\mathfrak{g}$.\big) 
So, $c_j-t_j\geq c_j+N\langle\alpha_i,\alpha_j^{\vee}\rangle$ in all the non-zero summands in \eqref{ESerre2} for $b_j=c_j$ $\forall$ $j$, as desired. 
\end{proof}                  
We are now ready to show Theorem \ref{thmC} for general $V$: 1) working in several steps; 2) repeated applications of Lemmas \ref{L3.1} and \ref{Lemma Serre}; 3) the core of the proof - going down from weights of the ``top simple $L(\lambda)$'', to those of the lower simple sub-quotients $L(w\bullet \lambda)$ for $w\in W$ having independent supports; 4) the freeness property in formula \ref{wt formula for simples} for these simples. 
\begin{proof}[\textnormal{\textbf{Proof of Theorem }\ref{thmC}}]
   Observe, weight-formula \eqref{Minkowski formula for supp of holes} easily 
 implies weight-formulas \eqref{All Minknowski formulas} $\forall$ $J\supseteq J_V$, by using: (i) any PBW basis of $U(\mathfrak{n}^-_J)$ having $U(\mathfrak{n}^-_{J_V})$-terms always at the right end -- via $U(\mathfrak{n}^-_J)=U\left(\bigoplus_{\alpha\in \Delta^-_J\setminus \Delta_{J_V}^- }\mathfrak{g}_{\alpha}\right)\ \otimes\ U(\mathfrak{n}^-_{J_V})$ -- so that $\wt_{J}V\  \subseteq\  \wt_{J_V} V -\mathbb{Z}_{\geq 0}(\Delta^+_J\setminus\Delta^+_{J_V})$; and (ii)~$\Delta^+\setminus \Delta^+_J \subseteq \Delta^+\setminus \Delta^+_{J_V}$.
 So it suffices to show formula \eqref{Minkowski formula for supp of holes}.\bigskip

\noindent
 {\bf Proof of Theorem \ref{thmC}\eqref{Minkowski formula for supp of holes}}.
Let $\mathfrak{g}, \lambda, V$ be as in the theorem.
We show in view of Theorem \ref{thmA}\eqref{Eqn A2}:
 \begin{equation}\label{Eqn Freeness inclusion with JV in Theorem C proof}
 \wt_{J_V} V \ - \  \mathbb{Z}_{\geq 0}\Pi_{\mathcal{I}\setminus J_V}\ \ \subseteq \ \ \wt V.
 \end{equation}
 \[
 \text{I.e., for any weight }\ \mu \ = \  \lambda\ -\ \sum\limits_{\ \ \mathclap{t\in \supp(\lambda-\mu)}}\ \ c_t\alpha_t\ \in  \wt_{J_V}V \ \ \   \& \ \ \text{ root-sum }\ x\ = \ \sum\limits_{\ \ \mathclap{t\in \supp(x)}}\ \ c_t\alpha_t\ \in \mathbb{Z}_{\geq 0}\Pi_{J_V^c}, 
 \]
 \vspace*{-3mm}
\[
\hspace*{2cm} \text{we prove } \ \mu - x \in\  \wt V,\qquad 
\text{by induction on}\quad \height(\lambda-\mu +x)\geq 0.
\]
(Note, $\supp(\lambda-\mu)\cap \supp(x)=\emptyset$.)\ \ 
In the base step $\mu=\lambda$ and $x=0$, the result is manifest.\\
\underline{Induction step for $\mu-x\precneqq \lambda$} {\bf :}
Observe, $\mu-x\in \wt V$: i) when $x=0$ trivially.
Next ii) when $\mu=\lambda$ by Theorem \ref{thmB}\eqref{Theorem B V with full weights}, since all the 1-dim. weight spaces along $J_V^c$-directions in $M(\lambda)$ \big(or precisely in the $\mathfrak{g}_{J_V^c}$-Verma $U(\mathfrak{n}^-_{J_V^c})\cdot m_{\lambda}$\big) survive in $V$.
Finally iii) when $\supp(x)\cap J_{\lambda}^c\neq \emptyset$ by Lemma \ref{L4.2}, since $\supp(\lambda-\mu)\subseteq J_V\subseteq J_{\lambda}$ and as $\mu \ - \  \ \sum\limits_{\ \ \ \ \mathclap{t\in \supp(x)\cap J_{\lambda}}}\  c_t\alpha_t$ by the induction hypothesis.\\
We assume henceforth:\   $\mu\precneqq \lambda,\   x\neq 0, \   \supp(x)\subseteq J_{\lambda}\cap J_V^c$. 
We work in many steps and cases below.
\medskip \\
\underline{Step 1}: Here we show our claim assuming $\supp(\lambda-\mu)$ to be independent; proceeding as in the proof of part \eqref{Theorem B V with full weights}.
Suppose $\supp(\lambda-\mu+x)$ is independent and $V_{\mu-x}=0$.
Let $y$ be a minimal weight with $\lambda \succeq y\succeq \mu-x$ s.t. $V_y=0$. 
By the raising trick in the proof of Lemma \ref{L4.2}, see that $\supp(\lambda-y)\subseteq J_{\lambda}$. 
Now the definitions of $J_V \supseteq \supp(\lambda-\mu)$ and $\supp(x)\subseteq J_{\lambda}\cap J_V^c$ together lead to $y\succeq \mu$, which contradicts $V_{\mu}\neq 0$.
So, we assume the Dynkin subdiagram on $\supp(\lambda-\mu+x)$ to have at least one edge for the reminder of this step.
Say it is between nodes $a,b\in \supp(\lambda-\mu+x)$, and without loss of generality let $c_b\geq c_a$ (both are positive).
Note, either $a\in \supp(\lambda-\mu)$, or $a\in \supp(x)$.
Nevertheless, $V_{\mu}\ =\ \mathbb{C}\prod\limits_{\ \ \ \  \mathclap{t\in \supp(\lambda-\mu)}}\ f_t^{c_t}\ \  v_{\lambda}\neq 0 \implies\ \ \ \prod\limits_{\ \ \mathclap{a\neq t\in \supp(\lambda-\mu)}}\ f_t^{c_t}\ \ v_{\lambda}\neq 0$ \big(as $f_t$s commute $\forall$ $t\in \supp(\lambda-\mu)$\big), and so $\mu\preceq \lambda-\ \  \ \sum\limits_{\ \ \ \mathclap{a\neq t\in \supp(\lambda-\mu)}}\ \  c_t\alpha_t\ \in \wt V$.
Next,\ \ \ $x\succeq\sum\limits_{\ \ \ \  \mathclap{a\neq t\in \supp(x)}}\ c_t\alpha_t\in \mathbb{Z}_{\geq 0}\Pi_{J_V^c}$, and moreover
\[
\mu-x+c_a\alpha_a= \lambda-\ \ \  \sum\limits_{\ \ \ \  \mathclap{a\neq t\in \supp(\lambda-\mu)}}\ c_t\alpha_t \ \  - \  \ \ \sum\limits_{\ \ \ \mathclap{a\neq t\in \supp(x)}}\ c_t\alpha_t\in \wt V
\qquad \text{by the induction hypothesis.}
\] 
Finally, Lemma \ref{L3.1} (for $\mu-x+c_a\alpha_a$ and $\alpha_a$) and the below inequalities yield $\mu-x\in \wt V$:  
\[
\langle \mu-x+c_a\alpha_a,\ \alpha_a^{\vee}\rangle \geq c_b\langle-\alpha_b,\alpha_a^{\vee}\rangle\geq c_a >0 \qquad \text{as }\ \langle\lambda,\alpha_a^{\vee}\rangle\geq 0\ \text{ and as }\   a\notin \supp(\lambda-\mu+x-c_a\alpha_a).  
\]
This completes the proof for Step 1.
Henceforth, we assume that $\supp(\lambda-\mu)$ is not independent. 
Our \underline{goal} in the below steps is to find (for the given weight $\mu\in \wt_{J_V}V$):
\begin{equation}\label{Eqn strategy in Theorem C proof}
\text{ a node }\ i \ \text{ and \ \  weight }\ \widehat{\mu}\in \wt_{J_V}V\  \ \text{ with }\ 
\mu \precneqq \widehat{\mu},\ \ \ \langle\widehat{\mu},\alpha_i^{\vee}\rangle >0\ \ \   \text{and}\ \ \ \mu\ \in \big[s_i\widehat{\mu},\ \widehat{\mu}\big].  
\end{equation}
Then the induction hypothesis and Lemma \ref{L3.1} yield $\mu-x\ \in \big[s_i(\widehat{\mu}-x),\ (\widehat{\mu}-x)\big]\ \subseteq \wt V$.
\begin{remark}
The below lines in the proof moreover show Proposition \ref{Corollary local Weyl group invariance 1} in Section \ref{Section 6},
which studies the actions of the generators of $\mathfrak{g}$ on the weight spaces of $V$, needed in Proposition \ref{Enumeration proposition- freeness at module level}.
\end{remark}
A key result in the below steps is Lemma \ref{Lemma local composition series for V} \big(or its original version \cite[Lemma 9.6]{Kac} for $J\neq \emptyset$ there-in\big), which yields for $V $ in category $\mathcal{O}$ and $\mu\in \wt V$, a weight $\lambda'\in \wt V$ with $\mu\in \wt L(\lambda') \subseteq \wt V$.
As one anticipates, and as we will see below, it suffices to work with $\lambda'$ in the linkage class $W\bullet \{\lambda\}$.
Recall  
$\lambda'$ is dominant integral along $J_{\lambda'}$-directions, $L(\lambda')$ is $\mathfrak{g}_{J_{\lambda'}}$-integrable, and 
\begin{equation}\label{Simple wt formula for lambda'}
\wt L(\lambda')\ \ = \ \ \wt L_{J_{\lambda'}}^{\max}(\lambda')-  \mathbb{Z}_{\geq 0}\Big(\Delta^+\setminus \Delta_{J_{\lambda'}}^+\Big)\ \ =\qquad \bigsqcup_{\mathclap{\xi\ \in\ \mathbb{Z}_{\geq 0}\Pi_{\mathcal{I}\setminus J_{\lambda'}}}}\ \ \wt L_{J_{\lambda'}}^{\max}(\lambda' -  \xi ). 
\end{equation}
\underline{Step 2}: We write the proof when $\mu\in \wt L(\lambda') \subseteq \wt V$ and $\height_{J_{\lambda'}}(\lambda'-\mu)>0$ \big(includes the case $\mu\in \wt L(\lambda)$\big). 
Note, $\mu$ is not in the free part $\lambda'-\mathbb{Z}_{\geq 0}\Pi_{\mathcal{I}\setminus J_{\lambda'}}$ of $\wt L(\lambda')$.
So, by formula \eqref{Simple wt formula for lambda'} and Lemmas \ref{Lemma reaching lambda from mu} and \ref{Lemma PSP in unit I-ht roots}, fix a node $j'\in J_{\lambda'}$ with $\mu+\alpha_{j'}\in \wt L(\lambda')$.
As $L(\lambda')$ is $\mathfrak{g}_{\{j'\}}$-integrable, by $s_j$-conjugating $\mu$ or $\mu+\alpha_{j'}$ if needed, we get a weight $\mu'$ in $\alpha_{j'}$-string $\big(\mu\ +\ \mathbb{Z}\alpha_{j'}\big)\ \cap\ \wt L(\lambda')$ with
\[
s_{j'}\mu' \ \precneqq \ \mu'\ \ 
  \big(\text{that is, }\langle\mu'\alpha_{j'}^{\vee}\rangle>0\big)\quad \text{ and }\quad \mu\  \in \ \big[ s_{j'}\mu', \ \mu' \big] \subseteq \wt L(\lambda').
\]
Now $\mu'-x\in \wt V$ by the induction hypothesis.
Checking $\langle x, \alpha_{j'}^{\vee}\rangle \in \mathbb{Z}_{\leq 0}$, and moreover $\langle\mu'-x,\alpha_{j'}^{\vee}\rangle\geq \langle\mu',\alpha_{j'}^{\vee}\rangle>0$ (since $j'\in \supp(\lambda-\mu) \subseteq J_V$ and $\supp(x)\subseteq J_V^c$) and applying Lemma \ref{L3.1} yields:
\[
\mu-x\ \in \ \big[s_{j'}\mu',\ \mu'
 \big]-x\ \subseteq \ \big[s_{j'}(\mu'-x),\ \mu'-x
 \big] \  \subseteq  \ \wt V.
\] 
We assume henceforth the below implication \big(Step 3 (1)--(4) note its important consequences\big): 
 \begin{equation*}
 \tag{\text{C1}}\label{Condition (C1)}
 \lambda'\preceq \lambda \ \  \text{ and }\ \ 
 \mu \in \wt L(\lambda') \subseteq  \wt V\ \ \implies \ \  \lambda'-\mu\in \mathbb{Z}_{\geq 0}\Pi_{\mathcal{I}\setminus J_{\lambda'}}.
 \end{equation*}
  \underline{Step 3}: Assumption \eqref{Condition (C1)} particularly prohibits $\mu\in \wt L(\lambda)$, as $\emptyset\neq \supp(\lambda-\mu)\subseteq J_V\subseteq J_{\lambda}$.
  So by 
  \[
  \wt L(\lambda)\ =\ \wt \left(M(\lambda,J_{\lambda})\ = \  \frac{M(\lambda)}{\sum\limits_{j\in J_{\lambda}}U(\mathfrak{n}^-)f_j^{\lambda\big(\alpha_j^{\vee}\big)+1}\cdot m_{\lambda}}\right),
  \]
  we have for some $i_1\in J_{\lambda}$: i) $v_1 :=  f_{i_1}^{\lambda\big(\alpha_{i_1}^{\vee}\big)+1}v_{\lambda}\neq 0$ is a maximal vector generating highest weight $\mathfrak{g}$-submodule $V_1:=U(\mathfrak{n}^-)v_1 \subset V$; 
  ii)~$\mu\in \wt V_1$.
  Note, $i_1\in J_V$ as $\supp(\lambda-\mu)\subseteq J_V$; also $V_1\twoheadrightarrow L\big(s_{i_1}\bullet \lambda\big)$.
  Now, either $\mu\in \wt L\big(s_{i_1}\bullet \lambda\big)$, or $\mu\notin \wt L\big(s_{i_1}\bullet \lambda\big)$.
  When $\mu\notin \wt L(s_{i_1}\bullet \lambda)$, in turn 
  \[
 \wt M\left( s_{i_1}\bullet \lambda, \ J_{s_{i_1}\bullet \lambda} \right) \ =\ \wt L\left(s_{i_1}\bullet \lambda\right)\ = \ \wt \left( \frac{V_1}{\sum\limits_{j\in J_{s_{i_1}\bullet\lambda}} \ U(\mathfrak{n}^-) f_j^{s_{i_1}\bullet\lambda\big(\alpha_j^{\vee}\big)+1}v_1} \right),
  \]
  yields $i_2\in J_{s_{i_1}\bullet \lambda}\cap J_V$ (similar to $i_1$) s.t.: i) $v_2\ := \ f_{i_2}^{s_{i_1}\bullet\lambda\big(\alpha_{i_2}^{\vee}\big)+1}v_1\neq 0$ is a maximal vector in $V_1$, and it generates a highest weight $\mathfrak{g}$-submodule $V_2:=U(\mathfrak{n}^-)v_2$ of $V_1$; ii) $\mu\in \wt V_2$.
  Again, either $\mu\in \wt L(s_{i_2}s_{i_1}\bullet \lambda)$, or $\mu\notin \wt L(s_{i_2}s_{i_1}\bullet \lambda)$.
  We repeat the above procedure iteratively and obtain say in $k$-steps a sequence of nodes $i_k,\ldots,i_1\in J_V$ -- satisfying properties similar to i) and ii), which are written down elaborately in conditions (1)--(4) below -- such that $\mu\notin \wt L(\lambda)\cup \cdots \cup \wt L\big(s_{i_{k-1}}\cdots s_{i_1}\bullet \lambda\big)$ but $\mu\in \wt L\big(s_{i_k}\cdots s_{i_1}\bullet \lambda\big)$; such a stage `$k$' must exist as $\height(\lambda-\mu)<\infty$.
  More precisely, we obtain (not necessarily distinct) $i_k, i_{k-1}, \ldots, i_1\in J_V\subseteq J_{\lambda}$ with the properties: 
  \begin{itemize}
      \item[(1)] Powers $p_{r+1}:= \big\langle(s_{i_r}\cdots s_{i_1})\bullet\lambda,\  \alpha_{i_{r+1}}^{\vee}\big\rangle+1 \ \geq 1$ $\forall$ $r\in \{0,\ldots, k-1\}$; with $s_{i_0}=1\in W$.
      So, 
      \[
      s_{i_k}\cdots s_{i_1}\bullet\lambda\  \precneqq\  \cdots\  \precneqq\  s_{i_{r+1}} \cdots  s_{i_1}\bullet\lambda\  \precneqq s_{i_r}\cdots s_{i_1}\bullet\lambda\ \precneqq\  \cdots\  \precneqq\
       \lambda.
      \]
      \item[(2)]  $\left(f_{i_r}^{p_r}\cdots f_{i_1}^{p_1}\right)v_{\lambda}\neq 0$ \big(as in \eqref{Eqn maximal vector example}\big) is a maximal vector in $V_{s_{i_r}\cdots s_{i_1}\bullet \lambda}$ for each $r$ and powers $p_1,\ldots,p_r$.
      So, $V_r:=U(\mathfrak{n}^-)\left(f_{i_r}^{p_r}\cdots f_{i_1}^{p_1}\right)v_{\lambda}$ is a highest weight $\mathfrak{g}$-submodule of $V$, with highest weight $\big(s_{i_r}\cdots s_{i_1}\big)\bullet \lambda$.
      Also, $V_{r+1}$ is a highest weight $\mathfrak{g}$-submodule of $V_r$.
      \item[(3)] For each $r$, we have $V_{\mu}\supseteq (V_r)_{\mu}\neq\emptyset$, i.e., $\mu\in \wt V_r$. 
      Importantly, 
      \[
      \mu\notin \wt L\big(s_{i_r}\cdots s_{i_1}\bullet\lambda\big) \ \ \text{ for any}\ \ r\ <\ k,\qquad \text{ but }\ \mu\ \in \ \wt L\big(s_{i_k}\cdots s_{i_1}\bullet \lambda\big).
      \]
      \item[(4)] We set for convenience $\lambda':=s_{i_k}\cdots s_{i_1}\bullet\lambda$. Now condition (C1) forces $\lambda'-\mu\in \mathbb{Z}_{\geq 0}\Pi_{\mathcal{I}\setminus J_{\lambda'}}$.
      \end{itemize}\smallskip  
        \underline{Step 4}: In the notations in Step 3(1)--(4), here we show the result -- i.e., $\mu-x\in \wt V$  for the (fixed) root-sum $0\neq x=\sum c_t\alpha_t\in \mathbb{Z}_{\geq 0}\Pi_{\mathcal{I}\setminus J_V}$ -- in the smallest case when $k=2$ and $\mu=s_{i_2}s_{i_1}\bullet \lambda$.
        Note $i_1\neq i_2$, and moreover $\langle\alpha_{i_1},\alpha_{i_2}^{\vee}\rangle\langle\alpha_{i_2},\alpha_{i_1}^{\vee}\rangle\neq 0$ since $\supp(\lambda-\mu)$ is not independent.
        Our proof in this case provides the strategy for the general case.
        We begin by noting in $\mu$,        
        \[        c_{i_1}=\langle\lambda,\alpha_{i_1}^{\vee}\rangle+1>0\quad \text{ and }\quad c_{i_2}=\langle\lambda,\alpha_{i_2}^{\vee}\rangle-c_{i_1}\langle\alpha_{i_1},\alpha_{i_2}^{\vee}\rangle+1 \ >\ c_{i_1}. 
        \]
         $c_{i_2}+c_{i_1}\langle\alpha_{i_1},\alpha_{i_2}^{\vee}\rangle\geq \langle\lambda,\alpha_{i_2}^{\vee}\rangle+1$ implies by Lemma \ref{Lemma Serre} for $(J, i)=(\{i_2\},i_1)$, that $M(\lambda)_{\mu}$ (so $V_{\mu}$) is spanned by the vectors $F\cdot f_{i_2}^{\lambda(\alpha_{i_2}^{\vee})+1} m_{\lambda}$ for $F\in U(\mathfrak{n}^-)$.
        So, by $\mathfrak{sl}_2$-theory, $f_{i_2}^N v_{\lambda}\neq 0$ \big(so $\lambda-N\alpha_{i_2}\in \wt V$\big) $\forall N\in \mathbb{Z}_{\geq 0}$.
        By this and the induction hypothesis $\mu+c_{i_1}\alpha_{i_1}=\lambda-c_{i_2}\alpha_{i_2}-x\in \wt V$. 
        By the below implication using Lemma \ref{L3.1}, we set $\widehat{\mu}\ = \ \lambda-c_{i_2}\alpha_{i_2}$ for \eqref{Eqn strategy in Theorem C proof}: 
        \[
        0<c_{i_1}\leq \langle\lambda,\alpha_{i_1}^{\vee}\rangle-c_{i_1}\langle\alpha_{i_2},\alpha_{i_1}^{\vee}\rangle-  \langle x,\alpha_{i_1}^{\vee}\rangle \ \ \implies \ \  \mu-x\in \big[s_{i_1}(\lambda-c_{i_2}\alpha_{i_2}-x),\ \lambda-c_{i_2}\alpha_{i_2}-x\big]\subseteq \wt V.
        \]
        The below steps show the result for every $k\geq 2$.
        By Step 3(4), we define for  $\lambda'=(s_{i_k}\cdots s_{i_1})\bullet\lambda$ - 
        \[
        I\ :=\ \{i_t\ |\  1\leq t\leq k\}\ =\ \supp(\lambda-\lambda')\quad \text{ and }\quad \mu=\lambda'-\sum\limits_{j\in J_{\lambda'}^c\cap J_V}d_j\alpha_j\ \ \text{ for }\ \ d_j\in \mathbb{Z}_{\geq 0}.
        \]
        Observe, $\{j\in J_{\lambda'}^c\cap J_V \ |\ d_j>0  \}\subseteq I$, as $\langle\lambda',\alpha_j^{\vee}\rangle\in \mathbb{Z}_{\geq 0}$ if $j\in 
        J_V\setminus I\subseteq J_{\lambda}$.
        So, $I=\supp(\lambda-\mu)$ (and is not independent).
        Let $q\in \{1,\ldots, k\}$ be the smallest for which $\{i_{q+1},\ldots, i_k \}$ is independent, and $I':=\{i_{q+1},\ldots, i_k \}$.
        Recall $\{i_k\}$ is independent, and so $q\leq k-1$. 
         By definitions, $i_q$ has at least one edge to a node, say $i_{q+1}$ for simplicity, in $I'$.
         Moreover, observe by its independence $I'\subseteq J_{\lambda'}^c \cap J_V$. 
        We employ the same strategy of Step 4, using $i_q$ and $i_{q+1}$, via comparing $d_{i_q}$ and $d_{i_{q+1}}$.
        We proceed in two steps below, applying Theorem \ref{thmA}:
        \medskip\\
        \underline{Step 5}: 
        We assume $d_{i_{q+1}}< d_{i_q}$, and work inside $V_q$ (in Step 6 as well).
        By Step 3(2), $\Big( f_{i_k}^{p_k}\cdots f_{i_{q+1}}^{p_{q+1}}\Big) f_{i_q}^{p_q}$ $\Big( f_{i_{q-1}}^{p_{q-1}}\cdots f_{i_1}^{p_1}\Big)v_{\lambda}\neq 0\in V_{\lambda'}$.
        As $I'$ is independent, $0\neq \bigg(\prod\limits_{i_{q+1}\neq i_t\in I'}f_{i_t}^{p_t}\bigg) f_{i_q}^{p_q}$ $\Big( f_{i_{q-1}}^{p_{q-1}}\cdots f_{i_1}^{p_1}\Big)v_{\lambda}$ is a maximal vector generating a highest weight $\mathfrak{g}$-submodule of $V_q$ with highest weight $\lambda'+p_{q+1}\alpha_{i_{q+1}}$.
        So, $\wt L\big(\lambda'+p_{q+1}\alpha_{i_{q+1}}\big)\subseteq \wt V$. 
        For $j\in J_{\lambda'}^c\cap J_V\subseteq J_{\lambda}$, observe that $\lambda'(\alpha_j^{\vee})\in \mathbb{Z}_{<0}$, and so $\langle \lambda'+p_{q+1}\alpha_{i_{q+1}},\ \alpha_j^{\vee}\rangle <0$ \big(particularly for $j$ s.t. $d_j>0$\big).
        So, by formula \eqref{Simple wt formula for lambda'} for $\wt L(\lambda'+p_{q+1}\alpha_{i_{q+1}})$: 
        \[
        \gamma\ :=\ (\lambda'+p_{q+1}\alpha_{i_{q+1}}) -\qquad \sum\limits_{\mathclap{j\ \in\  J_{\lambda'}^c\cap J_V\ \setminus \{ i_q, i_{q+1} \}}}\ \  d_j\alpha_j\ \ \in \wt L\big(\lambda'+p_{q+1}\alpha_{i_{p+1}}\big)\ \subseteq\  \wt V_q \ \subset \ \wt V.
        \]
       Now $i_q\notin \supp(\gamma\ -\ s_{i_q}\cdot\cdot s_{i_1}\bullet\lambda)$, moreover, $i_q\in J_{s_{i_q}\cdot\cdot s_{i_1}\bullet \lambda}^c$. So by Theorem \ref{thmA} for $M\big(s_{i_q}\cdot \cdot  s_{i_1}\bullet 
 \lambda\big) \twoheadrightarrow V_q$: 
        \[
     \widehat{\mu}\ =\qquad   \gamma-d_{i_q}\alpha_{i_q}\ = \ \mu+ (p_{q+1}+d_{i_{q+1}})\alpha_{i_{q+1}} \  \in\  \wt V_q\ \subset \ \wt V.\hspace*{2cm}
        \]
        The induction hypothesis says $\gamma-d_{i_q}\alpha_{i_q}-x\in \wt V$.
        Observe: i) $p_{q+1} = \langle s_{i_q}\cdots s_{i_1}\bullet \lambda, \alpha_{i_{q+1}}^{\vee}\rangle +1 = \langle \lambda'+p_{q+1}\alpha_{i_{q+1}},\ \alpha_{i_{q+1}}^{\vee}\rangle+1$ (as $I'$ is independent); and ii) by this and assumption $d_{i_{q+1}}< d_{i_q}$,  $0\leq p_{q+1}+d_{i_{q+1}}\leq \langle \gamma - d_{i_q}\alpha_{i_q},\ \alpha_{i_{q+1}}^{\vee} \rangle\leq \langle \gamma - d_{i_q}\alpha_{i_q}-x,\ \alpha_{i_{q+1}}^{\vee} \rangle$.
        Thus, by Lemma \ref{L3.1}, $\mu-x=\gamma-d_{i_q}\alpha_{i_q}  -\big(p_{q+1}+d_{i_{q+1}}\big)\alpha_{i_{q+1}}-x\in \big[ s_{i_{q+1}}\big(\gamma-d_{i_q}\alpha_{i_q}-x\big),\  \gamma-d_{i_q}\alpha_{i_q}-x \big]\ \subset \ \wt V$, as desired.
        \medskip
        \\
        \underline{Step 6}:  In this final step, $d_{i_q}\leq d_{i_{q+1}}$.
       In Step 5, node $i_{q+1}$ was important. Now its analogue is   
         \[\hspace*{2.5cm} I'':=\big\{i_t\in I'\ |\  p_t+p_{q}\langle\alpha_{i_q},\alpha_{i_t}^{\vee}\rangle\in \mathbb{Z}_{> 0} \big\}\ \ \subseteq\  I'.\qquad (I''\text{ can be }\emptyset.)
         \]
    Check by the independence of $I''$: 
    i) $p_t+p_{q}\big\langle\alpha_{i_q},\alpha_{i_t}^{\vee}\big\rangle= \big\langle s_{i_{q-1}}\cdot\cdot\
    s_{i_1}\bullet \lambda,\ \alpha_{i_t}^{\vee}\big\rangle+1$ $\forall$ $i_t\in I''$ \big(so~$I''$ $\subseteq J_{s_{i_{q-1}}\cdots s_{i_1}\bullet \lambda}$\big); 
     ii) as in Step 5, $ \bigg(\prod\limits_{i_t\in I''} f_{i_t}^{p_t}\bigg) f_{i_q}^{p_q}\big( f_{i_{q-1}}^{p_{q-1}}\cdot\cdot f_{i_1}^{p_1}\big)v_{\lambda}\neq 0$; iii) by ii) and by Lemma \ref{Lemma Serre} and $\mathfrak{sl}_2$-theory, we get a triple $(\lambda'' ,v, V')$ defined below, with $v$ a maximal vector generating the highest weight $\mathfrak{g}$-submodule $V'$ (the candidate for $V_q$ in Step 5) of $V$, with highest weight $\lambda''$-
    \begin{align}\label{Eqn triple step 6 }
    \begin{aligned}[c]
    \lambda'' \ :=\  s_{i_{q-1}}\cdots s_{i_1}\bullet \lambda\  -\ \sum_{i_t\in I''}\big[ p_t+ p_q\big\langle \alpha_{i_q}, \alpha_{i_t}^{\vee}\big\rangle\big]\alpha_{i_t} \quad \in\ \lambda' +\mathbb{Z}_{\geq 0}\Pi_{I'\sqcup \{i_q\}},\hspace*{1.5cm}\\
      v:=\Bigg(\prod\limits_{i_t\in I''} f_{i_t}^{p_t+p_{q}\big\langle\alpha_{i_q},\alpha_{i_t}^{\vee}\big\rangle} \Bigg) \Big( f_{i_{q-1}}^{p_{q-1}}\cdots f_{i_1}^{p_1}\Big)v_{\lambda}\ \neq 0\ \in\  V_{\lambda''}, \qquad  \  M(\lambda'')\twoheadrightarrow\  V':= U(\mathfrak{n}^-)v.
    \end{aligned}
    \end{align}
 Observe by the definitions of $I''$, $\lambda''$, and the independence of $I'$: i) $I''\subseteq J_{\lambda''}^c$; ii)~$\big\langle s_{i_{q-1}}\cdots s_{i_1}\bullet\lambda, \alpha_{i_t}^{\vee}\big\rangle\ = \ p_t + p_q\alpha_{i_q}\big(\alpha_{i_t}^{\vee}\big)-1 \ <\ 0$ $\forall$ $i_t\in I'\setminus I''\subseteq J_{\lambda}$, and so $I'\setminus 
 I''\subseteq  J_{\lambda''}^c$; 
and iii) $J_{\lambda'}^c\setminus (I'\sqcup\{i_q\})\subseteq J_{\lambda''}^c$. 
 So, by \eqref{Eqn triple step 6 } and Theorem \ref{thmA} for $V'$:
\[
 \gamma' \ := \ \lambda''+ \sum_{i_t\in I''} \big[p_q\big\langle\alpha_{i_q},\ \alpha_{i_t}^{\vee}\big\rangle -d_{i_t}\big]\alpha_{i_t} - \sum_{\ \ \mathclap{i_t\in I'\setminus I''}}\ \big(p_t+d_{i_t}\big)\alpha_{i_t}-  \ \ \sum_{\quad \mathclap{\quad j\in J_{\lambda'}^c\setminus (I'\sqcup \{i_q\})}}\ \ d_j\alpha_j \quad \in   \wt V'\  \subset\  \wt V.
 \]
 \big(Above, $d_{j}=0$ whenever $j\notin J_{\lambda'}^c\cap J_V$.\big) 
  Notice, $\gamma'=\mu+ (p_q+d_{i_q})\alpha_{i_q}$. 
  Finally, we check that-
  \begin{align*} 
  \langle\gamma', \alpha_{i_q}^{\vee}\rangle&\  \geq  \Big\langle\lambda'' +\sum_{\mathclap{i_t\in I''}}\big(p_q\langle\alpha_{i_q}, \alpha_{i_t}^{\vee}\rangle-d_{i_t}\big)\alpha_{i_t}- \sum_{\mathclap{i_t\in I'\setminus I''}}\ \big(p_t+d_{i_t}\big)\alpha_{i_t},\  \alpha_{i_q}^{\vee}\Big\rangle \\ &\ = \big\langle s_{i_{q-1}}\cdots s_{i_1}\bullet \lambda,\ \alpha_{i_q}^{\vee}\big\rangle - \big(\ \underbrace{p_{q+1}}_{>0}+\underbrace{d_{i_{q+1}}}_{\geq d_{i_q}}\ \big)\underbrace{\big\langle \alpha_{i_{q+1}}, \alpha_{i_q}^{\vee}\big\rangle}_{<0}  -\ \  \sum_{\ \ \mathclap{i_{q+1}\neq i_t\in I'}}\ (p_t+d_{i_t})\big\langle \alpha_{i_t}, \alpha_{i_q}^{\vee}\big\rangle  \\ 
  &\ \geq \underbrace{\big\langle s_{i_{q-1}}\cdots s_{i_1}\bullet \lambda,\ \alpha_{i_q}^{\vee}\big\rangle + 1}_{=p_q}\ +\ d_{i_q} \ = p_q+d_{i_q}\ >\ 0.
  \end{align*}
 We set $\widehat{\mu}=\gamma' = \mu+(p_q+d_{i_q})\alpha_{i_q}\in \wt V$ for \eqref{Eqn strategy in Theorem C proof}.
  Finally, the proof of the theorem is complete. 
\end{proof}
\section{Proof of Theorem \ref{Corollary local Weyl group invariance 2}: From weights to weight vectors, to slice decompositions}\label{Section 6}
We prove an application Theorem~\ref{Corollary local Weyl group invariance 2} of our strategy~\eqref{Eqn strategy in Theorem C proof} in the proof of Theorem~\ref{thmC}; i) which might be useful in studying $\wt V$, ii) understanding roots-strings through weights and the top weights in them \big(applications (1) and (2) above Lemma \ref{Lemma local composition series for V}\big), iii) which also motivated Problem~\ref{Problem non-vanishing vectors in free-directions for any V} on enumerations for $J_V^c$. 
Stating loosely, it says $\wt V$ is a patch-up of integrable weight-strings of simple roots, hanging down from some 1-dim. weight $\left(\prod_{i\in H}s_i\right) \bullet \lambda$ for independent $H\subseteq \mathcal{I}$.
This solves the first step of Problem \ref{Problem non-vanishing vectors in free-directions for any V}, in going  ``fully up along $J_V$-directions'', which additionally leads us to slice decompositions generalizing formulas \eqref{Int. slice. decomp. PVM} and \eqref{wt formula for simples}, discussed here.
 \big(Proposition~\ref{Enumeration proposition- freeness at module level} takes us further up in $J_V^c$-directions.\big)
We begin with $\lambda\in P^+$ and $V=L(\lambda)$ (integrable) :
\begin{lemma}\label{lemma local Weyl group invariance}
   Fix $\lambda\in P^+$ and a weight $\mu\precneqq \lambda\ \in \wt L(\lambda) \ =\ \wt L^{\max}(\lambda)$. 
   \\ \hspace*{1.5cm}  There exist\ \  a sequence of weights \  $\mu_0\ =\ \mu \  \precneqq\  \mu_1 \  \precneqq \  \cdots \  \precneqq \  \mu_n \ = \ \lambda$  \ \ and \\ 
   \hspace*{0.5cm} (not necessarily distinct) \ \  nodes \  $i_1,\ldots, i_n\in \supp(\lambda-\mu)$, \quad $n\leq \height(\lambda-\mu)$ \ \ such that :
   \smallskip
     \begin{center}
  \hspace*{2cm}$\langle\mu_t,\  \alpha_{i_t}^{\vee}\rangle\ \in \mathbb{Z}_{>0}$ \quad  and\quad  $\mu_{t-1}\ \in \big[ s_{i_t}(\mu_t),\ \mu_t\big]\ \subseteq \ \wt L(\lambda)$ \qquad\quad   $\forall$ \ \ $1\leq t\leq n$. 
   \end{center}
\end{lemma}
\begin{proof}[{\bf Proof}]
    We induct on $\height(\lambda-\mu)\geq 1$. 
    In the base step, $\mu=\lambda-\alpha_i$ for some $i\in \mathcal{I} \ = J_{\lambda}$, and since $\mu\in \wt L(\lambda)$, we must have $\langle\lambda,\alpha_i^{\vee}\rangle >0$ as required.\\
    Induction step: Let $\height(\lambda-\mu)>1$. 
    If $\langle\mu, \alpha_i^{\vee}\rangle < 0 $ for some $i\in \mathcal{I}$, we apply the induction hypothesis to $s_{i}\mu \succneqq \mu$. 
    Else, $\langle\mu,\alpha_i^{\vee}\rangle \geq 0$ $\forall$ $i\in \mathcal{I}$, and let $\mu+\alpha_j\in \wt V$ for $j\in \mathcal{I}$ \big(by Lemma \ref{Lemma reaching lambda from mu}\big).
    Then $\mu  \in \big[s_j(\mu+\alpha_j),\ (\mu+\alpha_j)\big]$, and applying the induction hypothesis to $\mu+\alpha_j$ finishes the proof.  
\end{proof}
This lemma shows that every weight of integrable $L(\lambda)$ belongs to some integrable/finite-string of some simple root -- in particular we can go up from $\mu\precneqq \lambda$ to some weight strictly above $\mu$ -- which is expected as the module $L(\lambda)$ to begin with is integrable.
Next, note $J_{L(\lambda)}= \mathcal{I}$, and in the chain of weights above $\mu$, the top weight $\mu_n$ is the highest weight $\lambda$; which we will generalize in Proposition \ref{Corollary local Weyl group invariance 1} for arbitrary $V$ (but $\lambda\in P^+$) below.
This leads us to extend and show this lemma in Theorem \ref{Corollary local Weyl group invariance 2} for every $(\lambda, V)$ and for almost every weight in $V$ \big(studied by \eqref{Eqn strategy in Theorem C proof}\big): 
\begin{prop} \label{Corollary local Weyl group invariance 1} Let $\mathfrak{g}$ be Kac--Moody algebra, $\lambda\in P^+$, and fix a module $M(\lambda)\twoheadrightarrow V$ with $J_V=\mathcal{I}$. 
$\text{Fix a weight }\mu\in \wt V\text{  with }\supp(\lambda-\mu)\text{ non-independent}$.
 Then \big(as in result \eqref{Eqn strategy in Theorem C proof}\big),   there exist a weight $\widehat{\mu}\in \wt V$ and a node $i\ \in \ \supp(\lambda-\mu)$  such that :
    \begin{align}\label{Eqn local Weyl group action  on weights}
      \begin{aligned}
\langle\widehat{\mu},\alpha_i^{\vee}\rangle > 0,\qquad \widehat{\mu}\ \succneqq\  \mu  \qquad \text{and} \quad   \mu\ \in\  \big[s_i\widehat{\mu}, \ \widehat{\mu} \big]\ \subseteq \wt V.
    \end{aligned}
    \end{align}
    More generally,  there exist sequences of weights and of (non necessarily distinct) nodes
      \begin{align}\label{Eqn integrable simple root strings above every weight} 
    \begin{aligned}[c]
     \mu_0\ =\ \mu\ \precneqq \ &\mu_1\ \precneqq \  \cdots\ \precneqq  \ \mu_n\ \in \wt V \quad \text{ and }\ \  
\text{resp.}\ \ i_1,\ldots, i_n\in \supp(\lambda-\mu) \ \  \text{s.t. :}\\  
   \langle\mu_t,\ \alpha_{i_t}^{\vee}\rangle&\ \in \mathbb{Z}_{>0},\qquad  \mu_{t-1}\  \in  \big[ s_{i_t}(\mu_{t}),\ \mu_t \big], \qquad \supp(\lambda-\mu_n) \ \text{ is independent,}\\
    &\ \mu_n\ = \ \lambda\ - \ \sum_{i\in \supp(\lambda-\mu_n)}d_i\alpha_i\quad \text{ with }\ \ d_i \ >\ \lambda(\alpha_i^{\vee})\ \ \forall\ i.
    \end{aligned}
    \end{align}
    \end{prop}
    \begin{remark}
  Above, in contrast to lower weights $\mu_t$, observe $\mu_n$ belongs to the non-integrable $\alpha_i$-string \big(that of the infinite simple $\mathfrak{sl}_{\alpha_i}$-module $L_{\{i\}}(-\lambda(\alpha_i^{\vee})-2)$\big) for all $i\in \supp(\lambda-\mu_n)$.
        In other words, since $H=\supp(\lambda-\mu)$ is independent, $\mu_n$ is a weight of the non-integrable simple $\mathfrak{g}$-module $L\big([\prod_{i\in H}s_i]\bullet \lambda\big)$, or equivalently of the non-integrable $\mathfrak{g}_H$-simple $L_H\big([\prod_{i\in H}s_i]\bullet \lambda\big)$.  
    \end{remark}
    \begin{proof}[{\bf Proof of Proposition \ref{Corollary local Weyl group invariance 1}}]
Property \eqref{Eqn local Weyl group action  on weights} immediately follows by the proof of Claim \eqref{Eqn strategy in Theorem C proof} \big(Theo rem~\ref{thmC} Steps (2)--(5)\big). 
For \eqref{Eqn integrable simple root strings above every weight}, we set $\mu_0=\mu$, $i_1=i$ and $\mu_1=\widehat{\mu}$. 
    If $\supp(\lambda-\mu_1)$ is not independent, in turn \eqref{Eqn local Weyl group action  on weights} applied to $\mu_1$ gives  \big($i_2\in \mathcal{I}, \ \mu_2\in \wt V$\big) with $\langle\mu_2,\alpha_{i_2}^{\vee}\rangle>0$ and $\mu_1\in \big[ s_{i_2}\mu_2,\ \mu_2\big]$. 
    Proceeding this way yields a chain $\mu_1\ \precneqq \ \cdots \precneqq \ \mu_m \in \wt V$ satisfying the pairing and interval conditions in \eqref{Eqn integrable simple root strings above every weight}, and with $m$ the least number s.t. $\supp(\lambda-\mu_t)$ is independent iff $t=m$ (in this chain).
    Let $\mu_m\ = \ \lambda-\sum_{i\in \supp(\lambda-\mu_m)} d_i\alpha_i$ for $d_i\in \mathbb{Z}_{>0}$.
    For the last condition in \eqref{Eqn integrable simple root strings above every weight}, we consider $K=\big\{ i\in \supp(\lambda-\mu_m)\ |\ d_i\leq \lambda(\alpha_i^{\vee}) \big\}$, and let $K=\big\{j_1,\ldots, j_k\big\}$ (in any fixed order, as $K$ is independent) with $k=|K|$.
    Observe, we are done up on setting $n=k+m$, $i_{m+r}=j_r$ and $\mu_{m+r}\ =\ \mu_m + \sum_{t=1}^r d_{j_t}\alpha_{j_t}$ $\forall$ $1\leq r\leq k$; so that $\mu_{m+k}$ is the desired top weight in the statement.
    \end{proof}
    \begin{remark}\label{Remark sl2 theory using rasing operators of general V}
    By result \eqref{Eqn local Weyl group action  on weights} and by $\mathfrak{sl}_2$-theory (for any $\mu\in \wt V$ with non-independent support):\quad $e_i \cdot V_{\mu}\ \neq 0$\quad \big(as 
$e_i\cdot [f_i^k\cdot V_{\widehat{\mu}}]\ \neq 0$ for $0\leq k\leq\ \widehat{\mu}(\alpha_i^{\vee})$\big).\qquad  In other words, \ \ $\mathfrak{n}^+V_{\mu}\neq 0$.\\
    So, part (a) guarantees that we can go one-step above $\mu$ by acting $e_i$s.
    
   Now an we go further-up from $\mu$ successively by $e_i$s, as in the proof of Lemma \ref{L4.2} \big(which is Problem~\ref{Problem non-vanishing vectors in free-directions for any V}(b)\big)?
    For general $\mu$ this might not be possible, but for $\mu$ along the free-directions -- i.e. $\mu\in \wt_{J_V^c}V$ -- this stronger property holds \big(see discussion below Remark \ref{Remark ordering only for support for C(i)s}\big), showed in Section \ref{Section 7}.
    \end{remark}
    Below is the complete result of Proposition \ref{Corollary local Weyl group invariance 1}, for every $\lambda\in \mathfrak{h}^*$ and $M(\lambda)\twoheadrightarrow V$:
    \begin{theorem}\label{Corollary local Weyl group invariance 2}
    Let $\mathfrak{g}$ be a Kac--Moody algebra, $\lambda\in \mathfrak{h}^*$ and $M(\lambda)\twoheadrightarrow V$. Fix any $J_V\ \subseteq\ J\ \subseteq \ J_{\lambda}$.
    \begin{itemize}
    \item[(a)] Fix a weight $\mu =\lambda-\sum_{i\in \mathcal{I}}c(i)\alpha_i \ \in  \wt V$ and root-sum $\xi = \sum\limits_{i\in J^c}c(i)\alpha_i$. 
    The results analogous to those in Proposition \ref{Corollary local Weyl group invariance 1}, hold inside -
    \begin{equation}\label{Eqn defn Jlambda weight slice}
    \boldsymbol{J}\text{\bf -weight slice at }\boldsymbol{\xi} \ {\bf :}\qquad\qquad \wt_J^{\xi}V\quad :=\ \ \ \Big(\lambda-\xi - \mathbb{Z}_{\geq 0}\Pi_J\Big)\ \cap\ \wt V. \hspace*{2.5cm}
    \end{equation}
    Namely, we have a sequence of weights in this slice and (not necessarily distinct) nodes-
    \begin{align}\label{Eqn integrable simple root strings above every weight in general case} 
    \begin{aligned}[c]
     \mu_0\ =\ \mu\ \precneqq \ &\mu_1\ \precneqq \  \cdots\ \precneqq  \ \mu_n\ \in \wt_J^{\xi} V, \quad \text{ and }\ \  
\text{resp.}\ \ i_1,\ldots, i_n\in \supp(\lambda-\mu)\cap J \ \  \text{s.t. :}\\  
   &\
    \langle\mu_t,\ \alpha_{i_t}^{\vee}\rangle\ \in \mathbb{Z}_{>0},\quad  \mu_{t-1}\  \in  \big[ s_{i_t}(\mu_{t}),\ \mu_t \big], \quad
    \supp(\lambda-\mu_n)\cap J \ \text{ is independent,}
    \end{aligned}
    \end{align}
    \item[(b)] 
    The weight-slice \eqref{Eqn defn Jlambda weight slice} for ($\mu$, $\xi$) is the weight-set of (not necessarily unique) highest weight $\mathfrak{g}_{J}$-submodule $V[\xi, J]$ of $V$ with highest weight $\lambda-\xi\ \big|_{\mathfrak{h}_J}$, defined in below steps :\smallskip \begin{itemize}
    \item For $\mu$ as above, fix the weight chain $\mu_1,\ldots, \mu_n$ as in result \eqref{Eqn integrable simple root strings above every weight in general case}. 
    \item Write $\mu_n= \lambda-\sum\limits_{i\in J} d_i\alpha_i-\sum\limits_{j\in J^c}c(j)\alpha_j$ and $\xi = \sum\limits_{j\in J^c}c(j)\alpha_j$, for $d_i, c(j)\in \mathbb{Z}_{\geq 0}$, and call $H\ =\ \big\{ i\in \supp(\lambda-\mu_n)\cap J\ |\  d_i\ \geq\  \lambda(\alpha_i^{\vee})- \xi(\alpha_i^{\vee})+1 >0  \big\} \ \subseteq \ J$.
    \item By Algorithm \ref{Section 6 Algorithm for enumerations} in the next section, or by Algorithm \ref{Algorithm for enumerations} in Section \ref{Section 1} when $\lambda\in P^+$, we fix an enumeration for (the support of $\xi$) $J^c\ = \ \big\{j_1,\ldots, j_N \big\}$, $N=|J^c|$. 
    \end{itemize}
    Then $\mu$ is a weight of the below highest weight $\mathfrak{g}_J$-module $V[\xi, H, J]$ with highest weight $\Big(\prod_{i\in H}s_i\Big)\bullet [\lambda -\xi]\ \Big|_{\mathfrak{h}_J}$: \quad 
    \big(see Remark \ref{Remark seeking non-vanishijng weight vectors in free-directions for Theorem integrable strings for weights (c')} on the below monomials not killing $v_{\lambda}$\big)
    \begin{equation}\label{Defn module V' with J-weight-slice}
    \hspace*{1cm}
 V[\xi, H,J]\ :=\  U(\mathfrak{g}_J) \cdot {\color{black}\bigg(\prod_{i\in H} f_i^{[\lambda-\xi](\alpha_i^{\vee})+1} \bigg)}\cdot \Big(f_{j_1}^{c({j_1})}\ \cdots\ f_{j_N}^{c({j_N})}\Big)\cdot v_{\lambda}\ \ \neq \ 0.
    \end{equation}
   Hence, running over all ($\mu$ and the corresponding) finitely many independent sets $H$, yields:
    \begin{equation}
        \wt_J^{\xi} V\ \ =\ \  \wt\big( V[\xi, J]\big) \qquad \text{ for }\ \ \underset{\tiny \text{(Verma over }\mathfrak{g}_J)}{M_J(\lambda-\xi)}\ \twoheadrightarrow \  \ \ V[\xi, J]\ := \  U(\mathfrak{g}_J) \ \cdot\  \Big(f_{j_1}^{c_{j_1}}\ \cdots\ f_{j_N}^{c_{j_N}}\Big)\cdot v_{\lambda}. 
    \end{equation}
    \end{itemize}
\end{theorem}
     \begin{proof}[{\bf Proof of Theorem \ref{Corollary local Weyl group invariance 2} part (a)}]
   We postpone the proof of part (b) to the end of the Section \ref{Section 7} after writing-down Algorithm \ref{Section 6 Algorithm for enumerations}, and show here part (a) \big(chain property \eqref{Eqn integrable simple root strings above every weight in general case} inside slice \eqref{Eqn defn Jlambda weight slice}\big).
 Fix $\mu=\lambda-\sum_{i\in \mathcal{I}}c(i)\alpha_i$ as in part (a), and $J_V\subseteq J\subseteq J_{\lambda}$.
 The below steps show the existence of $i_1$ and $\mu_1$, and one can then complete the proof of this part via inducting on $\height_J(\lambda-\mu)$.
    To begin with, if $\mu\in \wt_{J_V} V$ -- in particular, here $\xi=0$ -- then the result follows by Proposition \ref{Corollary local Weyl group invariance 1}.
    So we assume: i) $\height_{J_V^c}(\lambda-\mu)=k>0$; ii) $\height_J(\lambda-\mu)>0$, and moreover $\supp(\lambda-\mu)\cap J$ is non-independent \big(as when $\supp(\lambda-\mu)$ is independent, particularly $\emptyset$, we are done by setting $\mu_0=\mu=\mu_n$\big).
    Thus, we write $\mu \ = \ \mu' - \big(\gamma_1+\cdots+\gamma_k\big)$ for $\mu'\in \wt_{J_V} V$ and $\gamma_t\in \Delta_{J_V^c,1}$ by formula \eqref{Minkowski formula for supp of holes}.
\smallskip
\\ \underline{Step 1}: Suppose $\supp(\lambda-\mu')$ is not independent; in particular here $\mu'\precneqq \lambda$.
 Our strategy now is similar to that in the proof of Theorem \ref{thmA}.
   We apply Proposition \ref{Corollary local Weyl group invariance 1} to $\mu'$, which yields us a weight $\mu_1'\in \wt_{J_V}V \ \subseteq \wt_J V$ and a node $i_1\in J\cap \supp(\lambda-\mu') \subseteq J_V\subseteq J$ such that $\mu_1'\big(\alpha_{i_1}^{\vee}\big)>0$ and $\mu'\ \in\ \big[s_{i_1}\mu_1',\ \mu_1'\big)\ = \ \big[s_{i_1}\mu_1',\ \mu_1'\big]\setminus \{\mu_1'\} $.
To see that this choice of $i_1$ works, we set $\gamma_t':=\begin{cases}
    \gamma_t\ & \text{if }\gamma_t\big(\alpha_{i_1}^{\vee}\big)\leq 0\\
    s_{i_1}(\gamma_t)\ & \text{otherwise}
\end{cases}$,\  which belong to $\Delta_{J_V^c,1}$ $\forall$ $1\leq t\leq k$.
Now check that $\mu\ \in\ \big[s_{i_1}\big(\mu_1'- (\gamma_1'+\cdots+\gamma_k') \big), \  
\mu_1'- (\gamma_1'+\cdots+\gamma_k')  \big]$.
So, we are done by setting $\mu_1\ = \ \mu_1'- (\gamma_1'+\cdots+\gamma_k')$.\smallskip \\
\underline{Step 2}: In this step we assume that $\supp(\lambda-\mu')$ is independent; this includes the case $\mu'=\lambda$.
Write $\mu'=\lambda-\sum\limits_{j\in \supp(\lambda-\mu')}c'_j\alpha_j$.
To begin with, for some node $j\in J_V$, suppose one of the following two cases arise: 1)~$0< c'_j\leq \lambda(\alpha_j^{\vee})$; or 2)~$c'_j=0$ but (without loss of generality) $\height_{J_V}(\gamma_1)>0$ and moreover $\gamma_1-\alpha_j\in \Delta_{J_V^c,1}$. 
Then we are done by setting $i_1=j$ and defining $\gamma_1',\ldots,\gamma_k'\in \Delta_{J_V^c,1}$ as in the above step, and then setting $\mu_1= \underbrace{(\mu'+c'_j\alpha_j)}_{\in \wt_{J_V}V}- \gamma'_1-\cdots - \gamma'_k$.
So, we assume by case 2) that whenever $\gamma_t-\alpha
_j\in \Delta_{J_V^c,1}$ for some $1\leq t\leq k$ and $j\in J_V$, then necessarily $c'_j > \lambda(\alpha_j^{\vee})$.
In view of this, we set $J'=\big\{ j\in \supp(\lambda-\mu')  \ | \ c'_j>\lambda(\alpha_j^{\vee})\big\}\ \subseteq J_V\subseteq J$, and choose (in any way) for each $1\leq t\leq k$ minimal roots $\gamma_t''\in\ \Delta_{J_V^c, 1} \ \cap\ \big(\gamma_t-\mathbb{Z}_{\geq 0}\Pi_{J'}\big)$ such that $\height(\gamma_t'')$ is least.
Finally, we set $\mu'' = \mu'- (\gamma_1-\gamma_1'')-\cdots- (\gamma_k-\gamma_k'')$, which can be checked to be a (one-dimensional) weight in $\wt L_J\big( \big[\prod_{j\in J'}s_j\big]\bullet \lambda\big)$ by the definition of $J'$ \big(and by ``$\mathfrak{sl}_2^{\oplus |J'|}$-theory''\big).
Moreover, $\wt L_J\big( \big[\prod_{j\in J'}s_j\big]\bullet \lambda\big) \subset \wt V$ because  $\supp(\lambda-\mu')$ is independent and as $\mu'$, and thereby $\big[\prod_{j\in J'}s_j\big]\bullet \lambda\succeq \mu'$, survive in $\wt V$.
Now $\mu = \mu' - \gamma_1-\cdots -\gamma_k \ = \ \mu' - \gamma_1''-\cdots - \gamma_k''$, and once again by case 2) and by Lemma \ref{Lemma PSP in unit I-ht roots} and by the minimality of $\gamma_1'',\ldots, \gamma_k''$, we may assume that all $\gamma_1'',\ldots, \gamma_k''$ are simple roots in $\Pi_{J_V^c}$.
Recall our assumption above Step 1,  $\supp(\lambda-\mu)\cap J$ has at least one edge.
Let it be between $j_1, j_2\in \supp(\lambda-\mu)\cap J$ and say $\height_{\{j_1\}}(\lambda-\mu)\leq \height_{\{j_2\}}(\lambda-\mu)$.
Then we set $i_1=j_1$, and set $\mu_1=\big(\mu''-\height_{\{i_1\}}(\lambda-\mu'')\alpha_{i_1}\big)-\sum_{\gamma''_t\neq \alpha_{i_1}} \gamma''_t$ which lies in $\wt V$ by the independence of $\supp(\lambda-\mu'')$ and by weight formulas \eqref{Minkowski formula for supp of holes} or \eqref{All Minknowski formulas}; and these choices here work because $ \mu_1\big(\alpha_{i_1}^{\vee}\big) \geq -\height_{\{j_2\}}(\lambda-\mu)\langle\alpha_{j_2},\alpha_{j_1}^{\vee}\rangle\geq \height_{\{i_1\}}(\lambda-\mu)$.
Hence the proof is complete.
    \end{proof}
    Assuming the validity of Theorem \ref{Corollary local Weyl group invariance 2}(b), we extend slice formulas of Khare et al. as follows:
\begin{theorem}\label{theorem J-slice decomp.}
  In the notations of Theorem \ref{Corollary local Weyl group invariance 2}, for any $\big(\mathfrak{g}, \ \lambda, \ V\big)$, we have generalizing the integrable slice decompositions \eqref{Int. slice. decomp. PVM} and \eqref{wt formula for simples}  of $V=$ parabolic Vermas and resp. simple $L(\lambda)$s:  
  \begin{equation}\label{Slice decomp. for all V}
 \text{\bf J-Slice-decomposition}\ \ \   \wt V \ =  \ \bigsqcup\limits_{\xi\ \in \mathbb{Z}_{\geq 0}\Pi_{J^c}}  \wt V[\xi, J] \  \   =\ \ \ \ \  \  \ \    \bigcup\limits_{\mathclap{\substack{H \text{ independent in }J\\
\xi\ \in \mathbb{Z}_{\geq 0}}\Pi_{J^c}}} \  \ \ \ \ \ \  \underbrace{\wt V\big[\xi, H, J\big]}_{\neq \emptyset}.
  \end{equation}
\end{theorem}
  Recall, the slices in \eqref{Int. slice. decomp. PVM} and \eqref{wt formula for simples} are weights of integrable simples over $\mathfrak{g}_J$.
  For arbitrary $V$, slices in \eqref{Slice decomp. for all V} are weights of $\mathfrak{g}_J$-submodules $V[\xi, H, J]$ and $ V[\xi, J]$ of $V$.
  The definitions and even non-vanishing of these $\mathfrak{g}_J$-submodules rely on enumerative Algorithms \ref{Algorithm for enumerations} and \ref{Section 6 Algorithm for enumerations}.
A discussion on their non-vanishing at a foundational level:
\begin{remark}\label{Remark seeking non-vanishijng weight vectors in free-directions for Theorem integrable strings for weights (c')}
   In definition \eqref{Defn module V' with J-weight-slice} of $V[\xi, H, J]$ (any $J_V \subseteq J\subseteq J_{\lambda}$),  we first enumerate $J^c$ in which $v'=\Big(\prod_{j\in J^c} f_j^{c(j)}\Big)\cdot v_{\lambda}\neq 0\in V_{\lambda
   -\xi}$.      
   The next issue is, {\color{black}$\prod_{i\in H}f_i^{[\lambda-\xi](\alpha_i^{\vee})+1}\cdot v''$} may not be non-zero for every $v''\neq 0 \ \in V_{\lambda-\xi}$; (despite) we know those $v''$ are maximal for the $\mathfrak{n}_J^+$-action. 
   Whence: 
   1)~We search for vectors $v''\in V_{\lambda-\xi}$ that are not killed this product. 
   2)~In particular, exploring orderings -- which subsumes Problem~\ref{Problem non-vanishing vectors in free-directions for any V} at a weaker level -- might be interesting in its own right, as writing even an example of such $v''$ is perhaps not straight forward; even though we know $\wt_J V$ contains full-weights $\lambda-\mathbb{Z}_{\geq 0}\Pi_{J^c}$.
     Proposition \ref{Enumeration proposition- freeness at module level} proved in Section \ref{Section 7}, solves this question. 
\end{remark}
\section{Proof of Proposition \ref{Enumeration proposition- freeness at module level}: On weight-multiplicities along free $J_V^c$-directions}\label{Section 7}
Fix any $\big(\mathfrak{g}, \lambda,\  M(\lambda)\twoheadrightarrow V\big)$ as above.
Through-out, $J_V^c\neq \emptyset$, and
we repeatedly make use of:
\begin{equation}\label{Eqn hole-freeness property of JV complement}
\lambda \ - \  \mathbb{Z}_{\geq 0}\Pi_H \ \ \subseteq \  \wt V \qquad \quad \text{for all independent }\ H \subseteq \ J_V^c.
\end{equation}
So, $\wt_{J_V^c}V\ = \ \lambda-\mathbb{Z}_{\geq 0}\Pi_{J_V^c}$, by Theorem \ref{thmB}\eqref{Theorem B V with full weights}.
Our first goal here is to find 
(at least one) enumeration $J_V^c\ = \ \{i_1,\ldots i_n\}$ with $n=|J_V^c|$, for which 
\begin{equation}\label{Section 6 non-vanishing ordering problem}
\hspace*{2cm} f_{i_1}^{c(i_1)}\ \cdots \  f_{i_n}^{c(i_n)}\cdot v_{\lambda}\ \neq 0\ \qquad \qquad \text{ for any }\ c(i_1),\ldots, c(i_n)\in \mathbb{Z}_{\geq 0}.  
\end{equation}
Thereby, we prove all the parts (a), (b), (c$'$) and (c) of Proposition \ref{Enumeration proposition- freeness at module level}. 
We encourage the reader to recall Algorithm \ref{Algorithm for enumerations}, which yields these enumerations in the case $\lambda\in P^+$ \big(for solving Proposition~\ref{Enumeration proposition- freeness at module level} parts (a)--(c$'$)\big).  
In the case with any $\lambda\in \mathfrak{h}^*$ \big(for solving part (c)\big), we propose Algorithm
\ref{Section 6 Algorithm for enumerations} strengthening Algorithm \ref{Algorithm for enumerations}, written down in the proof of proposition \ref{Enumeration proposition- freeness at module level}(c) below. 
\begin{remark}\label{Remark ordering only for support for C(i)s}
               In enumerating $J_V^c$  for result \eqref{Section 6 non-vanishing ordering problem}, we note: (1) It suffices to enumerate only the support $\{i\in J_V^c\ |\  c(i)>0 \}$.
               However working with full $J_V^c$ does not affect the non-vanishing of the vector, since for nodes $i$ outside this support $f_i^{c(i)=0}\ =\ 1\in U(\mathfrak{n}^-)$ conventionally. 
                   (2) Importantly, one can enumerate $J_V^c$ canonically, not using the actual sequences $c(i)$, but only the patterns of ` $\leq$  or $\geq $ ' relations across $c(i)$s \big(for $i$ with positive degree in the Dynkin subgraph on $J_V^c$\big).  
               \end{remark}
Before beginning the proof of Proposition \ref{Enumeration proposition- freeness at module level}, we recall the applications of these enumerations:
\begin{itemize}
\item[(1)] For $\mathfrak{sl}_2$-theory (going-up by $e_i$s) in general $\mathfrak{g}$-modules $V$,  discussed in Remark~\ref{Remark sl2 theory using rasing operators of general V} \big(thereby, application (2) below\big): 
Following \eqref{Section 6 non-vanishing ordering problem}, are there nodes $j_1,\ldots, j_r\in \{i_1,\ldots, i_n\}$ with-
\begin{equation}\label{Eqn going-up by e-is on weight vectors, such orderings}
e_{j_1}\ \cdots \   e_{j_r}\ \cdot \  f_{i_1}^{c(1)}\ \cdots\ f_{i_n}^{c(n)}\cdot v_{\lambda} \ \ \ \  \neq \ \  0\ \in \mathbb{C}v_{\lambda}\ ?   \qquad \big(\text{Which is Problem }\ref{Problem non-vanishing vectors in free-directions for any V}\text{(b)}.\big)    
\end{equation}
\end{itemize}
\begin{remark}
The answer to the question on the nose, is negative; ex: when $n=1$ and $c(1)=\lambda(\alpha_1^{\vee})+1\in \mathbb{N}$, the vector $f_{i_1}^{c(1)}v_{\lambda}$ in \eqref{Eqn going-up by e-is on weight vectors, such orderings} is a maximal vector.
However, there is a more conceptual answer to the question of how high one can go up, which reveals a better picture of weights and weight vectors in $V$.
Namely, our result \eqref{Eqn answer to going up by e-is problem}, a key step in the proof of Proposition \ref{Enumeration proposition- freeness at module level} below.
\end{remark}
\begin{itemize}
\item[(2)] For generalizing Lemma \ref{Lemma freeness at module level} to all $V$; thereby multiplicity-bounds Proposition \ref{Enumeration proposition- freeness at module level}\eqref{Eqn bound of free-directional weight- dimensions}.
\item[(3)] Strengthening Theorem \ref{Corollary local Weyl group invariance 2}(b) (see Remark \ref{Remark seeking non-vanishijng weight vectors in free-directions for Theorem integrable strings for weights (c')}), i.e., for writing submodules $V[\xi, J]$, etc.
\end{itemize}
                 \begin{proof}[{\bf Proof of Proposition \ref{Enumeration proposition- freeness at module level}}]
                    
                    We first show Parts (a) and (b) of the proposition, where-in we fix a weight $\lambda\in P^+$ and $M(\lambda)\twoheadrightarrow V$ with $|J_V^c|=n>0$.\\
                    \smallskip
                    \underline{Proof of Part (a)}: 
                    We fix $c(i)\in \mathbb{Z}_{\geq 0}$ $\forall$ $i\in J_V^c$, and correspondingly an enumeration $J_V^c = \{i_1,\ldots, i_n\}$ by Algorithm \ref{Algorithm for enumerations}.  
                    Let
$k$ be the least stage with $\{i_{k+1}, \ldots, i_n\}$ independent as there.
                  By this, proceeding via \eqref{Eqn going-up by e-is on weight vectors, such orderings}, we show for the non-vanishing property \eqref{Section 6 non-vanishing ordering problem}:                     \begin{equation}\label{Eqn Proposition (a) proof non-vanishing claim}
                    e_{i_k}^{c(i_k)}\ \cdots\ e_{i_1}^{c(i_1)} \ \ \cdot  \  \  \Big(f_{i_1}^{c(i_1)}\ \cdots \ f_{i_n}^{c(i_n)}\  \cdot \ v_{\lambda}\Big)\ \ \neq \ 0 \ \ \in \  \mathbb{C}\Big\{ \   f_{i_{k+1}}^{c(i_{k+1})}\ \cdots \ f_{i_n}^{c(i_n)}\cdot v_{\lambda} \  \Big\}, 
                    \end{equation}
                    via inducting on full height $\sum_{i\in J_V^c}c(i)\geq 0$.
(Note the reversal in the orderings of $e_i$s and $f_is$ above.\big)
Base step: $c(i)=0$ $\forall$ $i$, so there is nothing to prove. \\
Induction step: 
We assume that some $c(i)>0$. 
In view of Remark \ref{Remark ordering only for support for C(i)s}, for simplicity, we assume $c(i)>0$ for all $i\in J_V^c$.
          To begin with, if $J_V^c$ is independent ($k=0$), then we only have commuting $f_i$s and no $e_i$s in \eqref{Eqn Proposition (a) proof non-vanishing claim}.
          So setting $F_1=\prod_{i\in J_V^c}f_i^{c(i)}$ (multiplied in any order) shows the result $F_1\cdot v_{\lambda}\neq 0$, by property \eqref{Eqn hole-freeness property of JV complement} of $J_V^c$.
        Henceforth, $J_V^c$ is not independent, i.e. $k\geq 1$.
        Next, we fix $i_1$ as in Algorithm~\ref{Algorithm for enumerations} Step 1. 
        Notice, the (sub-)enumeration $\{i_2,\ldots, i_n\} = J_V^c\setminus \{i_1\}$  also satisfies the rules of our algorithms \big(if we start with $c'(i_t)\ :=\  c(i_t)$ $\forall$ $t\geq 2$ and $c'(i_1)=0$\big); and let us set for convenience $x = \Big(f_{i_2}^{c(i_2)}\ \cdots \ f_{i_n}^{c(i_n)}\Big)  \cdot v_{\lambda}$.
        Now by the induction hypothesis 
        \begin{equation}\label{Eqn Proposition (a) proof  induction step non-vanishing vector}
               e_{i_k}^{c(i_k)}\ \cdots\ e_{i_2}^{c(i_2)} \ \cdot \ x \ \  \neq \ 0 \ \  \in\ \mathbb{C}\{ \   f_{i_{k+1}}^{c(i_{k+1})}\ \cdots \ f_{i_n}^{c(i_n)}\cdot v_{\lambda} \  \}. 
               \end{equation}
               
               Recall $x$ is a maximal vector for $e_{i_1}$-action, and moreover $\big\langle \lambda-\sum_{t=2}^n c(i_t)\alpha_{i_t},\  \alpha_{i_1}^{\vee}\big\rangle\  \geq\  -c(i')\langle \alpha_{i'},\ \alpha_{i_1}^{\vee}\rangle $ $\geq c(i_1)$, where $i'$ is the node adjacent to $i_1$ with $c(i')\geq c(i_1)$ that we consider in Step 1 of Algorithm~\ref{Algorithm for enumerations}.
                   The below calculation by  $\mathfrak{sl}_2$-theory and  \eqref{Eqn Proposition (a) proof  induction step non-vanishing vector}  together finish the proof of \eqref{Eqn Proposition (a) proof non-vanishing claim}-
                    \begin{equation}\label{Eqn Proof of main proposition (b) sl2 calculation}
                    e_{i_1}^{c(i_1)}\cdot f_{i_1}^{c(i_1)}\cdot x \quad =\quad  \bigg[\ \prod_{k=1}^{c(i_1)}k\bigg( \Big\langle \lambda-\sum_{t=2}^n c(i_t)\alpha_{i_t}, \ \alpha_{i_1}^{\vee}\Big\rangle\ -\big(k-1\big) \bigg)\ \bigg] \ x \ \ \neq 0.
                    \end{equation}
                    \smallskip\\
    \underline{Proof of Part (b)}:
    $\big(\lambda, V ,n \big)$, and $\mu=\lambda-\sum_{i\in J_V^c}c(i)\alpha_i\in \wt V$ be as above and fix vector $y\ =\  f_{i_1}^{c(i_1)} \cdots f_{i_n}^{c(i_n)}\cdot v_{\lambda}\neq 0\ \in V_{\mu}$.
    We now focus on $f_{i_{k+1}}^{c(i_{k+1})}\cdots f_{i_n}^{c(i_n)}\cdot v_{\lambda}$ ($f_i$s commute).
                Observe by $\mathfrak{sl}_2$-theory, for $H:=\big\{ 
i_t \ |\  t\geq k+1 \text{ and }c(i_t)\geq \lambda(\alpha_{i_t}^{\vee})+1\big\}$ and $d(i):=
                \begin{cases}
                
                    c(i)-\lambda(\alpha_i^{\vee})-1 \   &\text{if } i\in H\\
                    c(i) \   & \text{if } i\in J_V^c\setminus H
                \end{cases}\ :$ 
                \begin{equation}            e_{i_{k+1}}^{d(i_{k+1})}\ \cdots \ e_{i_n}^{d(i_n)}\ \cdot\  f_{i_{k+1}}^{c(i_{k+1})}\ \cdots\  f_{i_n}^{c(i_n)}\ \cdot\  v_{\lambda} \ \ \neq 0 \ \  \in \  \mathbb{C}\bigg\{\bigg(\prod_{i\in H}f_i^{\lambda(\alpha_i^{\vee})+1}\bigg) \cdot v_{\lambda}\bigg\}.   
                \end{equation}
                One verifies this similar to \eqref{Eqn Proof of main proposition (b) sl2 calculation}.
                Moreover, putting together this property and \eqref{Eqn Proposition (a) proof non-vanishing claim} yields    \begin{equation}\label{Eqn answer to going up by e-is problem}    
            e_{i_n}^{d(i_n)}\ \cdots \ e_{i_1}^{d(i_1)}\ \cdot\  y \ \ \neq 0 \ \  \in \  \mathbb{C}\bigg\{\bigg(\prod_{i\in H}f_i^{\lambda(\alpha_i^{\vee})+1}\bigg) \cdot v_{\lambda}\bigg\}. 
                \end{equation}

               Our goal is recollected in \eqref{Eqn goal in proof of part (b)} below.
               Here, we collectively work with all enumerations by Algorithm \ref{Algorithm for enumerations}.
               In them, it is in view of the above relations, that we fixed in the statement of Part (b) distinct subsets $H_1,\ldots, H_s$ \big(given $\mu$\big) with the important property $c(i) > \lambda(\alpha_i^{\vee})$ $\forall$ $i\in H_1\cup \cdots \cup H_s$.
                We define for convenience $w_{H_t}\ :=\ \prod_{i\in H_t} s_i$.
                 The lower bound \eqref{Eqn bound of free-directional weight- dimensions}, which is
                 \begin{equation}\label{Eqn goal in proof of part (b)}
   \qquad \quad  \dim\big( V_{\mu} \big) \ \ \geq \ \ \sum_{t=1}^s\dim\left[ L\big( w_{H_t}\bullet \lambda \big)_{{\large\mu}} \right]\qquad \quad \big(\text{every summand is positive}\big),
\end{equation}
                is easily verified by showing the steps : (1) $\mu\in \wt L\big(w_{H_t}\bullet \lambda\big)$ $\forall$ $1\leq t\leq s$, and (2) the inclusion - 
                 \begin{equation}\label{Eqn direct sum of non-hole simples embedding into V}
                 \qquad \Bigg(\bigoplus\limits_{i=1}^s\  L\big( w_{H_t} \  \bullet \lambda \big)\Bigg)_{\eta}\quad \large\xhookrightarrow{ }\quad V_{\eta} \quad  \forall\ \eta \quad \big(\text{at vector space level, loosely speaking}\big).
                 \end{equation}
                 Step (1) easily follows by \eqref{Eqn Proof of main proposition (b) sl2 calculation}, in view of the property of simples $L(\lambda') = \frac{M(\lambda')}{N(\lambda')}$:\ \ \ For $x\in M(\lambda')$-
                 \[
x\neq 0 \ \text{in}\  L(\lambda')\ \ \Big(\text{i.e., }\ x\in M(\lambda')\setminus N(\lambda')\Big)\qquad \iff \quad  m_{\lambda'}\ \in  \ U(\mathfrak{n}^+)\cdot x.
\]
To see this equivalence, let $X:=U(\mathfrak{n}^+)x\  \supset\  \{x\}\ \neq 0$.
$U(\mathfrak{n}^-)X$ is a submodule of $M(\lambda')$ by the triangular decomposition of $\mathfrak{g}$.
Further, $\big(U(\mathfrak{n}^-)X\big)_{\lambda'}\subseteq X_{\lambda'}$, since $\wt\big(U(\mathfrak{n}^-)X\big) \preceq \lambda'$.
If $x\neq 0 \text{ in }L(\lambda')$ and $X_{\lambda'}=0$, then $U(\mathfrak{n}^-)X\neq 0$ leads to a proper submodule in $L(\lambda')$ $\Rightarrow\!\Leftarrow$.
So
\begin{equation}\label{Eqn raising property of simples for proposition (b)}
\text{ given } \ v\ \neq\ 0 \ \in \ L(\lambda'), \qquad\quad  e_{j_1}\ \cdots \ e_{j_r} \ \cdot \ v \ \neq 0 \ \  \in L(\lambda)_{\lambda} \quad \text{ for some }\ \ j_1,\ldots, j_r\ \in\ \mathcal{I}.
\end{equation}
Conversely, if $X_{\lambda'}\neq 0$, then $U(\mathfrak{n}^-)X = M(\lambda')$ by definitions, and so $x\neq 0\text{ in } L(\lambda')$.

For step (2), we begin by applying Lemma \ref{Lemma local composition series for V} for the pair \big($V$ in Category $\mathcal{O}$, \ $w_{H_1}\bullet \lambda\in \wt V$\big), for a filtration $V_{m_1}\xhookrightarrow{} \cdots\xhookrightarrow{} V\twoheadrightarrow L(\lambda)$ with $\frac{V_{m_1-1}}{V_{m_1}}\simeq L\big(w_{H_1}\bullet \lambda\big)$; here $w_{H_1}\bullet \lambda$ must be top weight in simple $\frac{V_{m_1-1}}{V_{m_1}}$ since $V_{w_{H_1}\bullet \lambda}$ is maximal and 1-dim.  
                 We are done if the sub-quotients here include all the other $L\big(w_{H_t}\bullet \lambda\big)$s (covering $s=1$ case). 
                Else, we iteratively repeat this procedure for \big($V_{m_1}\in\mathcal{O}$, \ $w_{H_r}\bullet \lambda \in \wt V_{m_1}$\big) for every $2\leq r\leq s$. 
                This yields a filtration $V_{m_s}\xhookrightarrow{} \cdots\xhookrightarrow{} V_{m_r}\xhookrightarrow{}\cdots \xhookrightarrow{} V\twoheadrightarrow L(\lambda)$, where-in the sub-quotients include all the simples $L\big(w_{H_1}\bullet \lambda\big), \ldots, L\big( 
w_{H_s}\bullet \lambda \big)$, as desired, completing the proof of Part (b).
Part (c$'$) follows by definitions and the above proofs.
\medskip\\
    \underline{Proof of Part (c)}:
    We extend and show parts (a) and (b), for any $\big(\lambda, M(\lambda)\twoheadrightarrow V\big)$ and $\mu=\lambda-\sum_{i\in J_V^c}c(i)\alpha_i\in \wt_{J_V^c}V$, using the enumerations $J_V^c= \{j_1,\ldots, j_n\}$ in Algorithm \ref{Section 6 Algorithm for enumerations} :
     \begin{algorithm}\label{Section 6 Algorithm for enumerations}
          Fix any$\big(\mathfrak{g},  \lambda, M(\lambda)\twoheadrightarrow V\big)$, with $n=|J_V^c|>0$, and numbers $c(i)\in \mathbb{Z}_{\geq 0}$ $\forall\ i \in J_V^c$. 
           \begin{itemize}
           \item[(1)] To begin with, if $J_V^c\subseteq J_{\lambda}^c$, then any enumerations for $J_V^c$ works for our purpose by Lemma~\ref{Lemma freeness at module level}. 
          So we assume that $0<m =|J_{\lambda}\cap J_V^c|\leq n$, and write (in any order) 
          \[
          J_{\lambda}^c\cap J_V^c\ = \ \{ j_{m+1}, \ldots, j_n\}.
          \]
          \item[(2)] The below steps are for enumerations $J_{\lambda}\cap J_V^c=\{j_1,\ldots, j_m\}$.
          Observe, the input(s) in Algorithm \ref{Algorithm for enumerations} is only  the sequence $c(i)\in \mathbb{Z}_{\geq 0}$ $\forall$ $i\in I=J_V^c$ (and the Dynkin subgraph structure on $I$), and does not involve our assumption $\lambda\in P^+$ there-in.
           So, we can enumerate in this way (for our applications) any  subset $I\subseteq \mathcal{I}$ using any sequence $c: I\longrightarrow \mathbb{Z}_{\geq 0}$. 
           \item[(3)] By point (2) and Algorithm \ref{Algorithm for enumerations}, we arbitrarily enumerate for any $\lambda$ here $I=J_{\lambda}\cap  J_V^c\ = \ \{i_1,\ldots, i_m\}$;
           a first approximation to $\{j_1,\ldots, j_m\}$ desired. 
           Let $k\leq m$ be the least stage at which $H'=\{i_{k+1}, \ldots, i_m\}\ = J_{\lambda}\cap J_V^c\setminus \{i_1,\ldots, i_k\}$ is independent.
           We consider
           \[
            \hspace*{1.5cm} H:=\bigg\{ i\in H'\ \Big|\ \     c(i)>\lambda(\alpha_i^{\vee})- \Big[ c(i_{m+1})\alpha_{i_{m+1}}+\cdots +c(i_n)\alpha_{i_n}\Big]\big(\alpha_i^{\vee}\big)\  \in \mathbb{Z}_{\geq 0}
 \bigg\}\ \ \  \subseteq H'\cap J_{\lambda}.
 \]
          \item[(4)] Now enumerate $H'\setminus H =\{j_{k+1},\ldots, j_r\}$ and $H= \{j_{r+1},\ldots, j_m\}$, for $k+1\leq r\leq m$, and finally set $j_t=i_t$ $\forall$ $1\leq t\leq k$.
                                 \end{itemize}
                               \end{algorithm}
                             
We fix ordering $J_V^c=\{j_1,\dots, j_n\}$, independent set $H$ (possibly $\emptyset$) and stages $m, k, r\leq n$, as in the above algorithm.
Now $x\ = \ f_{j_1}^{c(j_1)}\ \cdots\ f_{j_n}^{c(j_n)}\ v_{\lambda} \neq 0 $, follows by the non-vanishing property:
 \begin{align}\label{Eqn non-vanishing 1 part (c)}
              \begin{aligned}[]      0\neq &  \ \overbrace{ e_{j_{n}}^{c(j_{n})} \  \cdots \  e_{j_{m+1}}^{c(j_{m+1})}}^{J_{\lambda}^c\cap J_V^c}\ \cdot\   \overbrace{e_{j_{r}}^{c(j_r)}  \cdots\ e_{j_{k+1}}^{c(j_{k+1})}}^{H'\setminus H}\ \cdot\ 
                    \overbrace{e_{j_k}^{c(j_k)}\ \cdots\  e_{j_1}^{c(j_1)}}^{J_{\lambda}} \ \ \cdot  \  \  \Big(f_{j_1}^{c(j_1)}\ \cdots \ f_{j_n}^{c(j_n)}\  \cdot \ v_{\lambda}\Big) \\
                       & \in\  e_{j_{n}}^{c(j_{n})}  \cdot\cdot \  e_{j_{m+1}}^{c(j_{m+1})}\ \cdot\       \mathbb{C}\Big\{ \   f_{j_{r+1}}^{c(j_{r+1})} \cdot\cdot \ f_{j_n}^{c(j_n)}\cdot v_{\lambda} \  \Big\}\               
                    \   =  \  \mathbb{C}\Big\{ \   \underbrace{f_{j_{r+1}}^{c(j_{r+1})} \cdot\cdot \ f_{j_m}^{c(j_m)}}_{H}\cdot v_{\lambda} \  \Big\}.
    \end{aligned}
    \end{align}
    This can be proved following the steps in the proof of part (a),  and using Lemma~\ref{Lemma freeness at module level} and  the definition of $H$.
    To begin with, note for $i\in H'\setminus H$ that $c(i)\leq \lambda(\alpha_i^{\vee})-\sum_{t=m+1}^n c(i_t)\alpha_{i_t}\big(\alpha_i^{\vee}\big) \in \mathbb{Z}_{\geq 0}$ by the definition of $H$, and so as in \eqref{Eqn Proof of main proposition (b) sl2 calculation} $e_i^{c(i)}\cdot f_i^{c(i)}\cdot v_{\lambda}\neq 0$.
Next, as in the proof of part (b) above,
\begin{equation}\label{Eqn non-vanishing 2 in part (c)}
e_{j_{r+1}}^{c(j_{r+1}) -  \lambda(\alpha_{j_{r+1}}^{\vee})-1}\ \cdots \ e_{j_m}^{c(j_m)-\lambda(\alpha_{j_m}^{\vee})-1}  \ \cdot \  f_{j_{r+1}}^{c(j_{r+1})}\ \cdots \ f_{j_m}^{c(j_m)}\cdot v_{\lambda}\ \ \in\ \mathbb{C}\bigg\{ \prod_{h\in H} f_{h}^{\lambda(\alpha_h^{\vee})+1}\ \cdot v_{\lambda} \bigg\}\ \setminus \{0\}.
\end{equation}
We set $v_1\ = \prod_{h\in H} f_{h}^{\lambda(\alpha_h^{\vee})+1}\ \cdot v_{\lambda}$ and consider the highest weight $\mathfrak{g}$-module $V_1\ = \ U(\mathfrak{n}^-)\cdot v_1$.
By the definition (and independence) of $H$, Lemma \ref{Lemma Serre} forces $x$ to belong to $\mathfrak{g}$-submodule $V_1\subset V$:
\begin{align}
\begin{aligned}[]
x \ & =  \ f_{j_1}^{c(j_1)}\cdots f_{j_r}^{c(j_r)}\ \cdot\  \bigg(\prod_{h\in H} f_{h}^{c(h)}\bigg) \cdot \Big( f_{j_{m+1}}^{c(j_{m+1})}\cdots f_{j_n}^{c(j_n)} \Big)\cdot v_{\lambda}\\
&\in 
f_{j_1}^{c(j_1)}\cdot\cdot f_{j_r}^{c(j_r)} \cdot\ U(\mathfrak{n}^-) \cdot \prod_{h\in H} f_h^{c(h)+\sum_{t=m+1}^n c(i_t)\alpha_{i_t}(\alpha_{h}^{\vee})}\cdot v_{\lambda}\ \in\  U(\mathfrak{n}^-)\cdot\prod_{h\in H} f_{h}^{\lambda(\alpha_h^{\vee})+1} \cdot v_{\lambda}. 
\end{aligned}
\end{align}
By this and by $v_1 \ \in U(\mathfrak{n}^+)\cdot x$ \big(via \eqref{Eqn non-vanishing 1 part (c)} and \eqref{Eqn non-vanishing 2 in part (c)}\big), observe as in the proof of part (b) Step (1) above,  $\mu\ = \ \lambda- \sum_{t=1}^n c(j_t)\alpha_{j_t}\ \in \wt L\bigg(\prod_{h\in H}s_h\ \bullet \lambda\bigg)\subseteq \wt V_1$.
The proof for part (c) lower bounds \eqref{Eqn goal in proof of part (b)} in this general case, is similar to that in part (b), completing the proof of the proposition.
          \end{proof}
We conclude the paper proving Theorem \ref{Corollary local Weyl group invariance 2}(b), using the above tools for ``going-up''.
\begin{proof}[{\bf Proof of Theorem \ref{Corollary local Weyl group invariance 2} (b)}]
The notations in the theorem and some simplifications to proceed:  
\begin{itemize}
\item Fix any $\mathfrak{g},\ \lambda\in\mathfrak{h}^*,\ M(\lambda)\twoheadrightarrow V,\  \xi=\sum\limits_{i\in J^c}c_i\alpha_i \in \mathbb{Z}_{\geq 0}\Pi_{J^c}$. 
We work with (any) $ J_V\subseteq J \subseteq J_{\lambda}$. 
\item Fix $\mu\ = \ \lambda - \sum\limits_{i\in \mathcal{I}}c(i)\alpha_i \ \in \wt V$ and the corresponding chain $\mu_0=\mu \prec \cdots \prec \ \mu_n\ = \ \lambda-\xi- \sum\limits_{i\in H'}d_i\alpha_i$, with $\supp(\mu_t-\mu_{t-1})=\{i_t\}$ for $1\leq t\leq n$ and $H'=\supp(\lambda-\mu_n)\cap J$ independent.
\big(Note $n=0$ if $\supp(\lambda-\mu)\cap J$ is independent to begin with.\big)
\item Recall, $H\ :=\ \big\{ j\in H'\cap J\ |\  d_j\ \geq\  \lambda(\alpha_j^{\vee})- \xi(\alpha_j^{\vee})+1 >0  \big\} \ \subseteq \ J$.
\item To prove result \eqref{Eqn non vanishing result in Step 1 in Theorem 6.5(b) proof}, we assume \big(or construct by repeatedly applying \eqref{Eqn local Weyl group action  on weights}\big) $\mu_t$ to be maximal in the string $\mu_{t-1}+\mathbb{Z}\alpha_{i_t}$ \big(so $e_{i_t}V_{\mu_t}=0$\big) and property \eqref{Eqn integrable simple root strings above every weight in general case} $\forall$ $t$.
Despite the unclear case \cite[Lemma 4.7(c)]{WFHWMRS}, $\mu_t$ are the indeed top weights of $\mu_t+\mathbb{Z}\alpha_{i_t}$, as $\mu_t\big(\alpha_{i_t}^{\vee}\big)>0$.
\end{itemize}
We fix throughout below, an ordering $J^c\ = \ \big\{j_1,\ldots, j_N\big\}$ by Algorithm \ref{Section 6 Algorithm for enumerations} and let the  independent set there-in be $K=\big\{j_{R+1},\ldots, j_M\big\}$ for $R\leq M\leq N$ with $\big\{j_{M+1},\ldots, j_N\big\}\subseteq \ J^c\cap J_{\lambda}^c$. 
Note the re-labellings: independent set \& stages $\big(H, r, m, n\big)$ in Algorithm \ref{Section 6 Algorithm for enumerations}(3) \ $\rightsquigarrow$\ $\big(K,R,M, N\big)$ here. 
We work with two independent sets $H\subseteq \supp(\lambda-\mu_n)\cap J$ above and $K$ in the enumeration of $J^c$.  
\[
\text{Note, }\ \ c(j)\ = \ d_j \ +\  \height_{\{j\}}(\mu_n-\mu)\quad \forall \ j\in J.
\]
Let us enumerate for convenience $H= \big\{ h_1,\ldots, h_m\big\}\subseteq J$ for $m=|H|\geq 0$.
We set 
\[
  x\ :=  \    \overbrace{\Big(f_{i_1}^{\height(\mu_1-\mu_0)}\ \cdot\cdot \ f_{i_n}^{\height(\mu_n-\mu_{n-1})}\Big) \bigg( \prod_{j\in H'\setminus H}f_j^{d_j}\bigg)     \Big( f_{h_1}^{d_{h_1}}\ \cdot\cdot  \ f_{h_m}^{d_{h_m}} \Big)}^{J} \  \times  \overbrace{\Big( f_{j_1}^{c(j_1)}\ \cdot\cdot  \ f_{j_N}^{c(j_N)}\Big)}^{\xi \ (\text{ in } J^c)}\cdot v_{\lambda}\ \in V_{\mu}.
 \]
Observe $x\neq 0$ in $V_{\mu}$ implies: i) $x$ (by its form) survives in $V[\xi, J]\ :=\ U(\mathfrak{n}_J^-)\cdot \Big(f_{j_1}^{c({j_1})}\ \cdots\ f_{j_N}^{c({j_N})}\Big)\cdot v_{\lambda}$, and then
 ii) $0\neq \bigg(\prod\limits_{t=1}^m f_{h_t}^{d_{h_t}} \bigg)\cdot \Big(f_{j_1}^{c({j_1})}\ \cdots\ f_{j_N}^{c({j_N})}\Big)\cdot v_{\lambda}\ \in  U(\mathfrak{n}_J^-) \cdot {\color{black}\bigg(\prod\limits_{i\in H} f_i^{[\lambda-\xi](\alpha_i^{\vee})+1} \bigg)}\cdot \Big(f_{j_1}^{c({j_1})}\ \cdots\ f_{j_N}^{c({j_N})}\Big)\cdot v_{\lambda}\ =: V[\xi, H J]$ \big(which can be proved by the generalization of identity \eqref{ESerre2} in Step 2 below in the proof\big) implies $x\in V[\xi, H, J]$.
So, in the reminder of the proof, we assume on the contrary $x=0$ and exhibit contradiction to the existence of $\mu_n\in \wt V$ in Property \eqref{Eqn integrable simple root strings above every weight in general case}. 
\smallskip\\
\underline{Step 1}: We observe by the calculations/verification similar to those in \eqref{Eqn non-vanishing 1 part (c)}, that :
\begin{align}\label{Eqn non vanishing result in Step 1 in Theorem 6.5(b) proof}
\begin{aligned} &\quad y := \ \Big( e_{j_N}^{c(j_N)} \cdot\cdot \ \Big(\underbrace{\cancel{e_{j_M}^{c(j_M)}}\cdot\cdot\ \cancel{e_{j_{R+1}}^{c(j_{R+1})}}}_K\Big)\cdot\cdot \ e_{j_1}^{c(j_1)}\Big) \cdot \bigg( \prod_{j\in H'\setminus H}e_j^{d_j}\bigg)    \cdot  \Big(\ e_{i_n}^{\height(\mu_n-\mu_{n-1})}\cdots\\
&\cdot\cdot\  e_{i_1}^{\height(\mu_1-\mu_0)}\Big)  \cdot x \ \ \   \in \Big(\mathbb{C}\setminus \{0\}\Big)\bigg\{\ z\ :=\ \Big( \underbrace{f_{h_1}^{d_{h_1}} \cdot\cdot \ f_{h_m}^{d_{h_m}}}_{H} \Big)\ \cdot \ \Big( \underbrace{f_{j_{R+1}}^{c(j_{R+1})} \cdot\cdot \ f_{j_M}^{c(j_M)}}_{K}\Big)\ \cdot \ v_{\lambda}\ \bigg\}.
\end{aligned}
\end{align}
Clearly the vanishing of $x$ forces $y = z =0$.
Our strategy now is to move $f_{h_t}$s past $f_{j_t}$s, and show $z=0$ for the desired contradiction.
We begin by noting that $f_{h_1},\ldots, f_{h_m}$ commute, similarly $f_{j_{R+1}},\ldots, f_{j_M}$ commute.
Given a sequence $\bar{b}=\big( b(1), \ldots, b(m)\big)\in (\mathbb{Z}_{\geq 0})^{m}$ and integer $R+1\leq k\leq M$, recall by $\mathfrak{sl}_2^{\oplus m}$-theory: i) $\beta(\bar{b}, k):= \alpha_{j_k}+\sum_{t=1}^m b(t)\alpha_{h_t}$ is a root of $\mathfrak{g}$ iff $0\leq b(t)\leq -\alpha_{j_k}\big(\alpha_{h_t}^{\vee}\big)$ for all $1\leq t\leq m$; ii) moreover, $\mathfrak{g}_{-\beta\big(\bar{b}, k\big)}$ is 1-dim. and is spanned by the Lie word 
\[
F\big(\bar{b}, k\big)
\ := \ \Big[\ \ad_{f_{h_1}}^{\circ b(1)}\ \circ\ \cdots\  \circ\  \ad_{f_{h_m}}^{\circ b(m)}\ \Big]\ \big(f_{j_k}\big).
\]
When $K\sqcup H$ is independent, note $F\big(\bar{b}, k\big)=0$ $\forall$ $\bar{b}\neq 0$. 
So, we set $K'=\big\{j_k\ \big|\ \{j_k\}\sqcup H \text{ is independent} \big\}$ $\subseteq K$; and say $K'=\big\{j_{R'},\ldots j_M\big\}$ for $R\leq R'\leq M$.
In the below lines of the proof, we assume $K'=\emptyset$ for simplicity.
While $K'\neq \emptyset$, the results in Steps 1 and 2 can be showed with $v_{\lambda}$ replaced by $f_{j_{R'}}^{c(j_{R'})}\cdots f_{j_M}^{c(j_M)}v_{\lambda}$ there-in, and the procedure for Step 3 is explained at the end of the proof. 
As $f_{h_1},\ldots, f_{h_m}$ commute, for $F\big(\bar{b}, k\big)\neq 0$, by $\mathfrak{sl}_2$-theory importantly:
\begin{align}\label{Eqn clearing j-ks for contradiction in Theorem 6.5 proof} 
\begin{aligned}[c]
\big[e_{j_k},\ F\big(\bar{b}, k\big)\big]\  \neq  0  & \ \iff \ \ \big[e_{j_k},\ F\big(\bar{b}, k\big)\big] \ = \  \underbrace{-\alpha_{h_t}\big(\alpha_{j_k}^{\vee}\big)}_{>0}f_{h_t} \ \  \text{for some }1\leq t\leq m \\  \iff \ \  & b(t) =1 \ \text{ for some }\ t\in \{1, \ldots, m\} \ \ \text{ and }\  \ \  b(t')=0 \ \ \forall\ \  t'\neq t.
\end{aligned}
\end{align}
\[
\text{Recall,}\ \ \   d_j \ \geq \ \lambda(\alpha_j^{\vee})-\xi(\alpha_j^{\vee})+1\ \geq\ \lambda(\alpha_j) +1  \ - \ \sum_{k=R+1}^M c(j_k)\alpha_{j_k}\big(\alpha_j^{\vee}\big) \ \ \forall\ j\in H \quad (\text{by definitions}).
\]
\underline{Step 2}: The derivation actions of $f_{h_t}$s on commuting $ \underbrace{f_{j_{R+1}}\cdots f_{j_{R+1}}}_{c(j_{R+1})\text{-times}}$ $\cdots  \underbrace{f_{j_M}\cdots f_{j_M}}_{c(j_M)\text{-times}} $ in $z$ yields a higher rank version of identity \eqref{ESerre2} \big(which was for single $f_i$\big) in the proof of Lemma \ref{Lemma Serre} we need now:
Namely, observe $z$ is a sum of $\mathbb{Z}_{>0}$-multiples -- which are products of binomial coefficients \big(as in \eqref{ESerre2}\big) corresponding to partitions $\sum_{t,j,k}b_j^k(t)$ of integer $\sum_{t=1}^m d_{h_t}$ below -- of products of $\big(c(j_{R+1})+\cdots+ c(j_M)\big)$-many Lie words of the below form-
\begin{align}\label{Eqn Lie words summing to z}
\begin{aligned}
    \Big( F\big(\bar{b}^{R+1}_1, R+1\big)\ \cdot\cdot \ F\big(\bar{b}^{R+1}_{c(j_{R+1})}, R+1\big) \Big) &\   \cdots  \Big( F\big(\bar{b}^{k}_1, k\big)\ \cdot\cdot \ F\big(\bar{b}^{k}_{c(j_k)}, k\big) \Big) \ \times \\
    \ \cdots\  \times & \Big( F\big(\bar{b}^{M}_1, M\big) \ \cdot\cdot \   F\big(\bar{b}^{M}_{c(j_M)}, M\big) \Big) \ \times \  \prod_{t=1}^m f_{h_t}^{d'_{h_t}} \ \cdot \ v_{\lambda} , 
\end{aligned}
\end{align}
where integer sequences $\bar{b}^k_j\ = \ \big(\bar{b}^k_j(1), \ldots, \bar{b}^k_j(m) \big)$ are dominated by $\Big(\big|\alpha_{j_k}\big(\alpha_{h_1}^{\vee}\big)\big| \  ,\ldots,\  \big|\alpha_{j_k}\big(\alpha_{h_m}^{\vee}\big)\big| \Big)$ entry-wise so that $F\big(\bar{b}^k_j, k\big)\ \neq 0$. 
Thus, $d'_{h_t}\geq \lambda(\alpha_{h_t}^{\vee})+1\in \mathbb{Z}_{>0}$. 
Now observe $ e_{j_{R+1}}^{c(j_{R+1})} \times \cdots \times  e_{j_M}^{c(j_M)}\ \cdot\ z $ \big($e_j$s commute here\big) is a $\mathbb{Z}_{> 0}$-linear combination of the below products of Lie words acting on $v_{\lambda}$: \\
\begin{align}\label{Eqn Lie words summing to e's acting on z}
\begin{aligned}
    &\Big( \big[e_{j_{R+1}}, \ F\big(\bar{b}^{R+1}_1, R+1\big)\big]\  \cdot\cdot \ \big[e_{j_{R+1}}, \ F\big(\bar{b}^{R+1}_{c(j_{R+1})}, R+1\big)\big] \Big)  \times   \cdot\cdot \times    \Big( \big[e_{j_k},\ F\big(\bar{b}^{k}_1, k\big)\big]\ \cdot\cdot \\
     \cdot \cdot  \big[ e_{j_k}, & \ F\big(\bar{b}^{k}_{c(j_k)}, k\big)\big] \Big) 
   \times  \cdot\cdot \times  \Big( \big[e_{j_M}, \ F\big(\bar{b}^{M}_1, M\big)\big] \ \cdot\cdot \   \big[e_{j_M},\ F\big(\bar{b}^{M}_{c(j_M)}, M\big)\big] \Big) \ \times \  \prod_{t=1}^m f_{h_t}^{d'_{h_t}}\ \cdot\  v_{\lambda}, \end{aligned}
\end{align}
For a proof of this result, below we show how all $f_{j_{R+1}}$s can be cleared proceeding similar to in the proof of \eqref{ESerre2}, but in the opposite direction in a sense; and repeating this step iteratively for other $f_{j_k}$s yields the above property: acting $e_{j_{R+1}}^{c(j_{R+1})}\cdots e_{j_{M}}^{c(j_{M})}$ on the vector in \eqref{Eqn Lie words summing to e's acting on z} yields
\[
e_{j_{R+1}}^{c(j_{R+1})}\cdots e_{j_{M}}^{c(j_{M})}  \ \times  \  \Big( F\big(\bar{b}^{R+1}_1, R+1\big)\ \cdots \ F\big(\bar{b}^{R+1}_{c(j_{R+1})}, R+1\big) \Big)  \ *\ *\ *\hspace*{1.5cm}  \]
\[ \hspace*{1.5cm}  =\ \
e_{j_{R+1}}^{c(j_{R+1})}\ \times \  \Big( F\big(\bar{b}^{R+1}_1, R+1\big)\ \cdots \ F\big(\bar{b}^{R+1}_{c(j_{R+1})}, R+1\big) \Big) \times\  \Big(e_{j_{R+2}}^{c(j_{R+2})}  \cdots e_{j_{M}}^{c(j_{M})}\Big) \ *\ *\ * \ . 
\]
 The latter is a $\mathbb{Z}_{>0}$-linear combination \big(involving binomial coefficients as in \eqref{ESerre2}\big) of vectors - 
\begin{align}\label{Eqn non-vanishing for F(b, R+1)s and ***}
\begin{aligned}
\ad_{e_{j_{R+1}}}^{\circ p(1)}\Big( F\big(\bar{b}^{R+1}_1, R+1\big)\Big)\ \cdots&  \ \ad_{e_{j_{R+1}}}^{\circ p(R+1)}\Big( F\big(\bar{b}^{R+1}_{c(j_{R+1})}, R+1\big) \Big) \ \times e_{j_{R+1}}^{c(j_{R+1})-p(1)-\cdot\cdot- p(R+1) }\\
& \ \times\  \Big(e_{j_{R+2}}^{c(j_{R+2})}  \cdots e_{j_{M}}^{c(j_{M})}\Big) \ *\ *\ *\ \ ,
\end{aligned}
\end{align}
where-in $p(1), \ldots, p(R+1)\in \mathbb{Z}_{\geq 0}$ and $p(1)+\cdots + p(R+1)\leq c(j_{R+1})$.
Importantly, the above vector for the sequence $\big(p(1),\ldots, p(R+1)\big)$ is non-zero if and only if $p(t)=1$ $\forall$ $t$, because:
1) $f_{j_{R+1}}$ occurs exactly once in $F\big(\bar{b}^{R+1}_t, R+1\big)$ and
2) next, some $p(t)=0$ implies $e_{j_{R+1}}^{c(j_{R+1})-p(1)-\cdots -p(R+1)}\neq 1$ commutes with all the root vectors on its right in \eqref{Eqn non-vanishing for F(b, R+1)s and ***} and kills $v_{\lambda}$. 
\smallskip\\
\underline{Step 3}: \big(Recall we assumed $K'=\emptyset$ above \eqref{Eqn clearing j-ks for contradiction in Theorem 6.5 proof}.\big)
By \eqref{Eqn clearing j-ks for contradiction in Theorem 6.5 proof}, every weight vector in \eqref{Eqn Lie words summing to z} is a positive integer multiple of $v\ :=\ f_{h_1}^{d_{h_1}}\cdots f_{h_m}^{d_{h_m}}\cdot v_{\lambda}$.
Our assumption $z=0$ in $V$ forces $v=0$, and so $v_{H}:= f_{h_1}^{\lambda(\alpha_{h_1}^{\vee})+1}\cdots f_{h_m}^{\lambda(\alpha_{h_m}^{\vee})+1}\ \cdot v_{\lambda}\ =0\ = \   V_{s_{h_1}\cdots s_{h_m}\bullet \lambda}$ (by $\mathfrak{sl}_2^{\oplus m}$-theory).
This contradicts the existence of $\mu_n\in \wt V$ at the beginning, explained below: 
Suppose $\mu_n= u -\sum\limits_{r}\gamma_r$ for some $u\in \wt_H V$, $\gamma_r\in \ \Delta_{H^c, 1}\cap \big( \alpha_{i_r} + \mathbb{Z}_{\geq 0}\Pi_H\big)$ and $i_r\in H^c$ (by Lemma \ref{Lemma forward inclusion in wt formulas}); with all the $\alpha_{i_r}$s occurring in $\sum_{r}\gamma_r$ sum to $\xi+\sum_{i\in H'}d_i\alpha_i$.
Then for every $t$: 1)~$\height_{\{h_t\}}(\gamma_r)\leq -\alpha_{i_r}(\alpha_{h_t}^{\vee})$ by $\mathfrak{sl}_2$-theory $\forall$ $r$; 2)~thereby $\height_{\{h_t\}}(\lambda-u) = d_{h_t}-\sum_{r}\height_{\{h_t\}}(\gamma_r) \geq  d_{h_t}- \xi(\alpha_{h_t}^{\vee}) \geq \lambda(\alpha_{h_t}^{\vee})+1$. 
Now $u\preceq s_{h_1}\cdots s_{h_m}\bullet\lambda$, but $u\in \wt_H V$ contradicts $s_{h_1}\cdots s_{h_m}\bullet \lambda\notin \wt_H V$.  
This proves the theorem when $K'=\emptyset$.

Following the discussion above \eqref{Eqn clearing j-ks for contradiction in Theorem 6.5 proof}, it remains to show Step 3 when $\emptyset\neq K'=\big\{j_{R'},\ldots, j_M\big\}\subseteq K$, with
$v'\ := f_{h_1}^{d_{h_1}}\cdots f_{h_m}^{d_{h_m}}\cdot  f_{j_{R'}}^{c(j_{R'})}\cdots f_{j_M}^{c(j_M)} \cdot  v_{\lambda}$ in place of $v$. 
Again by $\mathfrak{sl}_2^{\oplus m+M-R'+1}$-theory, $v'=0$ implies $v_{H\sqcup K'}:=f_{h_1}^{\lambda(\alpha_{h_1}^{\vee})+1}\cdots f_{h_m}^{\lambda(\alpha_{h_m}^{\vee})+1}\ \cdot \  f_{j_{R'}}^{\lambda(\alpha_{j_{R'}}^{\vee})+1}\cdots f_{j_M}^{\lambda(\alpha_{j_M}^{\vee})+1} \cdot v_{\lambda}\ =0\ = \   V_{s_{h_1}\cdots s_{h_m}\cdot s_{j_{R'}}\cdots s_{j_M}\bullet \lambda}$.
Fix $H_0\subseteq H\sqcup K'$ a minimal subset (so $H_0\subseteq J_{\lambda}$) with $\prod_{h\in H_0}f_h^{\lambda(\alpha_h^{\vee})+1}v_{\lambda}=0$ .
The definition of $J_V$, $H\subseteq J_V\subseteq J$ and $K'\subseteq J^c\cap J_{\lambda}$, together force $H_0 \subseteq H\subseteq J_V$.
We repeat the lines in the above paragraph in Step 3, with $H$ replaced by $H_0$ everywhere, which yields us the desired contradiction in this case.   
Hence, finally, the proof of Theorem \ref{Corollary local Weyl group invariance 2}(b) is complete.
\end{proof}
	 \subsection*{Acknowledgements} My deepest thanks to my Ph.D. advisor A. Khare, for his invaluable encouragement and discussions which helped to refine the results and also improve the exposition.
I sincerely thank  G. Dhillon and A. Prasad for posing some of the questions we solved in this paper; as well as R. Venkatesh for helpful discussions on the Parabolic-PSP and S. Viswanath for suggesting some references; also S. Kumar, V. Chari and S. Pal for their helpful feedback on the initial drafts of the paper.
     This work was supported by two scholarships from the National Board for Higher Mathematics \big(Ref. No. 2/39(2)/2016/NBHM/R{\&}D-II/11431 and 0204/16(8)/2022/R\&D-II/11979\big), and an institute Post Doctoral fellowship at the Harish-Chandra research Institute.
	 
  \bigskip
  \noindent
  	 \address{(G. Krishna Teja) \textsc{S - 16, Stat.\, \&\, Math. Unit, Indian Statistical Institute Bangalore Center, Bangalore, 560059,  India.}}
	 
	 \textit{E-mail address}: \email{\texttt{tejag@alum.iisc.ac.in}}


\begin{thebibliography}{16}
	 	\bibitem{Venkatesh} G. Arunkumar, D. Kus and R. Venkatesh, Root multiplicities for Borcherds algebras and graph coloring, J. Algebra, 499 (2018) 538--569.
\bibitem{Shushma}
S. Rani and G. Arunkumar, A study on free roots of Borcherds-Kac-Moody Lie superalgebras, J. Combin. Theory Ser. A 204 (2024), Paper No. 105862, 48 pp.
   
	 	\bibitem{Borel} A. Borel and J. Tits, Groupes r{\'e}ductifs. Inst. Hautes {\'E}tudes Sci. Publ. Math., 27 (1) (1965) 55--150.
 	 	\bibitem{svis} L. Carbone, K. N. Raghavan, B. Ransingh, K. Roy and S. Viswanath, $\pi$-systems of symmetrizable Kac--Moody algebras, Lett. Math. Phys., 111 (5) (2021).
	 	\bibitem{Casselman} W. A. Casselman, Geometric rationality of Satake compactifications. In Algebraic groups and Lie groups, volume 9 of Austral. Math. Soc. Lect. Ser., pages 81--103. Cambridge University Press, Cambridge, 1997.
	 	\bibitem{Cellini} P. Cellini and M. Marietti, Root polytopes and Borel subalgebras, Int. Math. Res. Not., (12) (2015) 4392--4420.
	 	\bibitem{Chari_contm} V. Chari, R.J. Dolbin and T. Ridenour, Ideals in parabolic subalgebras of simple Lie algebras, Contemp. Math., 490 (2009) 47--60.
	 	\bibitem{Chari_Adv} V. Chari and J. Greenstein, A family of Koszul algebras arising from finite-dimensional representations of simple Lie algebras, Adv. Math., 220 (4) (2009) 1193--1221.
	 	\bibitem{Chari_JGeom} V. Chari and J. Greenstein, Minimal affinizations as projective objects, J. Geom. Phys., 61 (3) (2011) 594--609.  
	 	\bibitem{Chari_JPAA} V. Chari, A. Khare and T. Ridenour, Faces of polytopes and Koszul algebras, J. Pure Appl. Algebra 216 (7) (2012) 1611--1625.
	 	\bibitem{Khare_Ad}
	 	G. Dhillon and A. Khare, Faces of highest weight modules and the universal Weyl polyhedron, Adv. Math., 319 (2017) 111--152.
	 	\bibitem{Dhillon_arXiv}
	 	G. Dhillon and A. Khare,
	 	The weights of simple modules in Category $\mathcal{O}$ for Kac–Moody algebras,
	 	J. Algebra, 603
	 	(2022) 164--200.
	 	\bibitem{Fernando}
	 	S.L. Fernando, Lie algebra modules with finite-dimensional weight spaces, I, Trans. Amer. Math. Soc., 322 (1990) 757--781.
	 	\bibitem{Futorny}
	 	V. Futorny, Classification of irreducible nonzero level modules with finite-dimensional weight spaces for affine Lie algebras, J. Algebra, 238 (2) (2001) 426--441.
	 	 \bibitem{GaLe}
 H. Garland and J. Lepowsky, Lie algebra homology and the Macdonald--Kac formulas,
  Invent. Math., 34 (1976) 37--76.
 
	 	\bibitem{Hump} 
	 	J. E. Humphreys, Introduction to Lie algebras and representation theory, Graduate Texts in Mathematics,
	 	no. 9, Springer-Verlag, Berlin-New York, 1972.
	 	\bibitem{Hump_BGG}
	 	J. E. Humphreys, Representations of  semisimple Lie Algebras in the BGG Category $\mathcal{O}$, Graduate Studies in Mathematics, vol. 94, American Mathematical Society, Providence, RI, 2008.
	 	\bibitem{Kac}
	 	V. G. Kac, Infinite-dimensional Lie algebras, Cambridge University Press, Cambridge, third edition,
	 	1990.
	 	\bibitem{Khare_JA}
	 	A. Khare, Faces and maximizer subsets of highest weight modules, J. Algebra, 455 (2016) 32--76.
	 	\bibitem{Khare_Trans}
	 	A. Khare, Standard parabolic subsets of highest weight modules, Trans. Amer. Math. Soc., 369 (4) (2017) 2363--2394.
	 	\bibitem{Khare_AR}
	 	A. Khare and T. Ridenour, Faces of weight polytopes and a generalization of a theorem of Vinberg, Algebras
	 	and Representation Theory 15 no. 3 (2012) 593–611.
	 	\bibitem{WFHWMRS} G. Krishna Teja, Weak faces of highest weight modules and root systems, Transform. Groups, 
         Published online (2023) - DOI: 10.1007/s00031-022-09786-w.
      \bibitem{Teja_ArXiv}
		G. Krishna Teja, Moving between weights of weight modules, (2021) ArXiv ref - arXiv:2012.07775.
  \bibitem{Teja_Fpsac}
 G. Krishna Teja, Weak Faces and a Formula for Weights of Highest Weight Modules Via Parabolic Partial Sum Property for Roots, 
  S\'{e}minaire Lotharingien de Combinatoire, 86B.66 (2022).
	 	\bibitem{KuBGG}
S. Kumar, Bernstein--Gel'fand--Gel'fand resolution for arbitrary Kac--Moody
algebras, Math. Ann.,
286 (4) (1990) 709--729.
	 	
	 	\bibitem{lepo1}
J. Lepowsky, Generalized Verma modules, the Cartan--Helgason theorem, and
the Harish-Chandra homomorphism, J. Algebra, 49 (2) (1977) 470--495.
	 	\bibitem{L_JA} J. Lepowsky, A generalization of the Bernstein--Gelfand--Gelfand resolution, J. Algebra, 49 (2) (1977) 496--511.
	 	 	\bibitem{Mathieu}
	 	O. Mathieu, Classification of irreducible weight modules, Ann. Inst. Fourier (Grenoble), 50 (2) (2000) 537--592.
	 	\bibitem{Satake} I. Satake, On representations and compactifications of symmetric Riemannian spaces, Ann. of Math., 71 (2) (1960) 77--110.
	    \bibitem{Stembridge} J. R. Stembridge, The partial order of dominant weights, Adv. Math., 136 (2) (1998) 340--364. 
	 	\bibitem{Vinberg} E. B. Vinberg, Some commutative subalgebras of a universal enveloping algebra, Izv. Akad. Nauk SSSR Ser. Mat., 54 (1) (1990) 3--25, 221.
	 	\bibitem{WaRo}
N. R. Wallach and A. Rocha-Caridi, Projective modules over graded Lie algebras I, Math. Z., 180 (1982) 151--178.
	 \end{thebibliography}
	 \end{document}